\documentclass[12pt]{article}

\usepackage[utf8]{inputenc}
\usepackage{amsthm}
\usepackage{amsmath}
\usepackage{amsfonts}
\usepackage{amssymb}

\usepackage{etex}
\reserveinserts{128}
\usepackage[maxfloats=100]{morefloats}
\usepackage{graphicx}
\usepackage[all]{xy}				
\usepackage{float}

\usepackage[colorlinks=true,citecolor=black,linkcolor=black,urlcolor=blue]{hyperref}

\theoremstyle{plain}
\newtheorem{thm}{Theorem}[section]
\newtheorem{lem}[thm]{Lemma}
\newtheorem{cor}[thm]{Corollary}
\newtheorem{prop}[thm]{Proposition}

\theoremstyle{definition}
\newtheorem{defn}[thm]{Definition}
\newtheorem{ex}[thm]{Example}
\newtheorem{conj}[thm]{Conjecture}

\theoremstyle{remark}
\newtheorem{rk}[thm]{Remark}

\numberwithin{equation}{section}
\numberwithin{figure}{section}
\numberwithin{table}{section}

\newcommand{\scalingconstant}{.559}
\renewcommand{\geq}{\geqslant}
\renewcommand{\leq}{\leqslant}
\renewcommand{\ge}{\geqslant}

\title
{Diagrammatic description of c-vectors and d-vectors of cluster algebras of
	finite type\\
{\centering \small{To the memory of Andrei Zelevinsky}}
}

\author{Tomoki Nakanishi\\
	\small Graduate School of Mathematics\\[-0.8ex]
	\small Nagoya University\\[-0.8ex]
	\small Chikusa-ku, Nagoya, Japan\\
	\small\tt nakanisi@math.nagoya-u.ac.jp
	\and
	Salvatore Stella\thanks{Partially supported by A. Zelevinsky's NSF grant DMS-1103813 and Northeastern University}\\
	\small Department of Mathematics\\[-0.8ex]
	\small North Carolina State University\\[-0.8ex]
	\small Raleigh, NC, USA\\
	\small\tt sstella@ncsu.edu
}

\date{Jan 12, 2014}

\begin{document}

\maketitle

\begin{abstract}
	We provide an explicit Dynkin diagrammatic description of the $c$-vectors and
	the $d$-vectors (the denominator vectors) of any cluster algebra of finite
	type with principal coefficients and any initial exchange matrix.  We use the
	surface realization of cluster algebras for types $A_n$ and $D_n$, then we
	apply the folding method to $D_{n+1}$ and $A_{2n-1}$ to obtain types $B_n$ and
	$C_n$.  Exceptional types are done by direct inspection with the help of a
	computer algebra software. We also propose a conjecture on the root property
	of $c$-vectors for a general cluster algebra.
\end{abstract}

\section{Introduction}
\subsection{Background}

For a given skew-symmetrizable integer matrix $B$, let $\mathcal{A}_\bullet(B)$
be the cluster algebra with {\em principal coefficients\/} whose initial
exchange matrix is $B$ \cite{Fomin02,Fomin07}.  Note that
$\mathcal{A}_\bullet(B)$ depends on $B$ itself (not on its mutation equivalence
class) due to the presence of principal coefficients. There are two important
families of integer vectors associated with $\mathcal{A}_\bullet(B)$: {\em
$c$-vectors\/} and {\em $d$-vectors}.  The former are the column vectors in the
bottom half square matrices ({\em $C$-matrices\/}) of the {\em extended exchange
matrices\/} of $\mathcal{A}_\bullet(B)$.  The latter are also called the {\em
denominator vectors}; they are the tuples of the exponents in the denominators
of the Laurent expansions of the cluster variables of $\mathcal{A}_\bullet(B)$
in terms of the initial cluster.  An alternative way to introduce them is:
$c$-vectors are the tropicalized versions of coefficients ($y$-variables) and
$d$-vectors are the tropicalized version of cluster variables ($x$-variables),
respectively.  See Section \ref{sect:back.vectors} for details.  

Fix an indexing set $I$. Following \cite{Fomin03a}, to each skew-symmetrizable
matrix $B=(b_{ij})_{i,j\in I}$, we assign a symmetrizable matrix
$A(B)=(a_{ij})_{i,j\in I}$ called the {\em Cartan counterpart\/} of $B$, by
setting 
\begin{align}
	a_{ij}=
	\begin{cases}
	2 & i=j\\
	-|b_{ij}|& i \neq j.
	\end{cases}
\end{align}
The matrix $A(B)$ is a symmetrizable (generalized) Cartan matrix in the sense of
Kac \cite{Kac90}.  It has been partially recognized and proved that, the $c$-
and $d$-vectors of $\mathcal{A}_\bullet(B)$ are  roots of the root system of the
Cartan matrix $A(B)$.  When $B$ is skew-symmetric, thanks to Kac's theorem
\cite{Kac80}, it is enough to prove that the vectors (or their negatives) are
identified with the {\em dimension vectors\/} of some indecomposable modules of
the path algebra $kQ(B)$ for the quiver $Q(B)$ corresponding to $B$.  In fact,
this is a common method of proving many known cases.  We are going to discuss
this subject in more detail in Section \ref{sect:back}.

Cluster algebras of {\em finite type}, i.e., the ones with finitely many seeds,
form one of the most basic and important classes of cluster algebras
\cite{Fomin03a}. They have been intensively studied in particular in the cases
when $B$ is skew-symmetric, i.e. when $\mathcal{A}_\bullet(B)$ is of one of the
{\em simply-laced types} $A_n$, $D_n$, $E_6$, $E_7$, $E_8$ according to the
classification of \cite{Fomin03a}.  In these cases the {\em cluster-tilted algebra
$\Lambda(B)$}, introduced in \cite{Buan04} as a certain quotient of the path
algebra $kQ(B)$, plays a key role in the study of $\mathcal{A}_\bullet(B)$
\cite{Caldero04,Caldero04b,Buan06,Buan04b,Buan05}. A $c$-vector is said to be
{\em positive} if it is a nonzero vector and its components are all nonnegative.
A $d$-vector is {\em non-initial} if it is the $d$-vector of a non-initial
cluster variable.  It was proved by \cite{Caldero04b,Buan04} that the set of all
the non-initial $d$-vectors of $\mathcal{A}_\bullet(B)$ coincides with the set
of the dimensions vectors of all the indecomposable $\Lambda(B)$-modules.
Moreover, it was recently proved by \cite{Najera12,Najera12b} that the set of
all the positive $c$-vectors of $\mathcal{A}_\bullet(B)$  also coincides with
the same set.  See Theorems \ref{thm:d-fin} and \ref{thm:c-fin}.

In spite of this beautiful and complete, representation-theoretic description of
$c$- and $d$-vectors for finite type, little is known about their {\em
explicit} form, except for type $A_n$  \cite{Caldero04,Parsons11,Tran}.  The
purpose of this paper is to fill this gap and to provide an {\em explicit Dynkin
diagrammatic description} of the $c$- and $d$-vectors of cluster algebras
of any finite type with any initial exchange matrix.

It is our hope that the lists presented here will be useful for studying cluster
algebras, as the appendix of \cite{Bourbaki02} is for studying Lie algebras.

\subsection{Main results}
We present here the main results of the paper.  Recall
that,  for a skew-symmetrizable matrix $B$, the cluster algebra $\mathcal{A}_\bullet(B)$
is of finite type if and only if $B$ is mutation equivalent to a matrix $B'$
whose Cartan counterpart $A(B')$ is a Cartan matrix of finite type, $A_n$,
$B_n$, $C_n$, $D_n$, $E_6$, $E_7$, $E_8$, $F_4$, $G_2$ \cite{Fomin03a}.  We say
that such a skew-symmetrizable matrix $B$  is of {\em cluster finite type}, and
also, more specifically,  of {\em cluster type $Z$}, according to the type $Z$
of $A(B')$ above.  For any skew-symmetrizable matrix $B$ of cluster finite type, we
present the Cartan matrix $A(B)$ as a Dynkin diagram $X(B)$ in the usual way
following \cite{Kac90}. Note that, in general, $X(B)$ is not a finite type
Dynkin diagram.

For each finite type $Z$,
we provide the following two lists explicitly:
\par
$\bullet$
the list $\mathcal{X}(Z)$ of the Dynkin diagrams $X(B)$ of all the
skew-symmetrizable matrices $B$ of cluster type $Z$ (for each $B$ the vertices
of $X(B)$ are naturally identified with elements of $I$),
\par
$\bullet$
the list $\mathcal{W}(Z)$ of the ``templates'' of positive $c$-vectors and
non-initial $d$-vectors in the form of  {\em
weighted Dynkin diagrams}, namely, Dynkin diagrams with a positive integer
attached to each vertex.
\par
 
For a pair $X(B)\in \mathcal{X}(Z)$ and $ W\in \mathcal{W}(Z)$, an embedding of
the diagram part of $W$ into $X(B)$ as a full sub-diagram is denoted by
$W\subset X(B)$. Such an embedding is not necessarily unique if it exists; we
distinguish them up to isomorphism of $W$. To each embedding $W\subset X(B)$ 
corresponds an integer vector $v=(v_i)_{i\in I}$: its $i$-th component $v_i$ 
is the weight of $W$ at $i$.  

For each skew-symmetrizable matrix $B$ of cluster type $Z$, let us introduce the sets
\begin{align}
 \begin{split}
  \mathcal{V}(B):=& \{ \,
  W\subset X(B)  \mid 
  W\in \mathcal{W}(Z)
 \,
  \},\\
  \mathcal{C}(B):=& \{ \,
\mbox{all $c$-vectors of $\mathcal{A}_\bullet(B)$}
 \,
  \},\\
    \mathcal{C}_+(B):=& \{ \,
\mbox{all positive $c$-vectors of $\mathcal{A}_\bullet(B)$}
 \,
  \},\\
   \mathcal{D}(B):=& \{ \,
\mbox{all non-initial $d$-vectors  of $\mathcal{A}_\bullet(B)$}
 \,
  \}.
  \end{split}
\end{align}
For finite type cluster algebras, it turns out that
\begin{align}
	\label{eqn:sign-coherence}
	\mathcal{C}(B)=  \mathcal{C}_+(B)\sqcup   (-\mathcal{C}_+(B)),
\end{align}
therefore, we can concentrate on $ \mathcal{C}_+(B)$. Our main result is stated
as follows.
\begin{thm}
\label{thm:main}
Let $B$ be any skew-symmetrizable matrix of cluster finite type.  
Then, the sets $\mathcal{C}_+(B)$, $\mathcal{D}(B)$, and
$\mathcal{V}(B)$ coincide.
\end{thm}

Let us illustrate the content of Theorem \ref{thm:main} by mean of a
baby example; the reader can find slightly bigger examples at the end of Section
\ref{sect:sets}.
\begin{ex}
	The matrix 
	\[
		B= 
		\left( 
			\begin{array}[h]{ccccc} 
				0 & 1 & -1 \\ 
				-1 & 0 & 1 \\ 
				1 & -1 & 0 \\ 
			\end{array} 
		\right) 
	\]
	is of cluster type $A_3$ and the Dynkin diagram $X(B)$ corresponding to it is
	\begin{center}
		\includegraphics[scale=\scalingconstant]{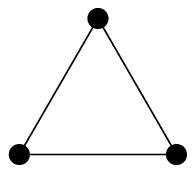}
	\end{center}
	There are precisely three templates in $\mathcal{W}(A_3)$:
	\begin{center}
		\includegraphics[scale=\scalingconstant]{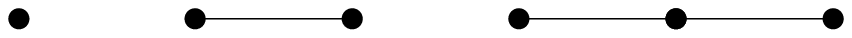}
	\end{center}
	The first and second of them can be embedded as full sub-diagram into $X(B)$ in
	three different ways each while the third one can't be embedded into $X(B)$. 
	We get therefore 6 vectors:
	\[
		\mathcal{V}(B)=
		\left\{(1,0,0),(0,1,0),(0,0,1),(1,1,0),(1,0,1),(0,1,1)\right\}
	\]
	they are both the positive $c$-vectors and the non-initial $d$-vectors of
	$\mathcal{A}_\bullet(B)$.
\end{ex}

An immediate and important corollary of Theorem \ref{thm:main} is that, for
simply laced types, the set $\mathcal{V}(B)$ also coincides with the set of the
dimension vectors of all the indecomposable modules of the cluster-tilted
algebra $\Lambda(B)$, thereby yielding a representation-theoretic result.

To prove Theorem \ref{thm:main} we use the surface realization of cluster
algebras \cite{Fock,Fomin08, Fomin08b} for types $A_n$ and $D_n$.  The case $A_n$ is
easy, but the case $D_n$ is (much) more involved.  Then we apply the folding
method \cite{Dupont08,Demonet11} to types $D_{n+1}$ and $A_{2n-1}$ to obtain types $B_n$
and $C_n$, respectively.  Exceptional types are
studied by direct inspection with the help of the software by Keller
\cite{Keller08c} and the cluster algebra package \cite{Musiker10} of Sage
\cite{sage} written by Musiker and Stump; we rely on Corollaries \ref{cor:C=D}
and \ref{cor:independent1} to simplify computations in type $E_8$. 
In classical types our derivation is
purely combinatorial and does not refer to any results from representation
theory.  On the one side, this may be unsatisfactory due to the lack of a direct
representation-theoretic explanation; on the other side, this is the reason why
we get the result easily. In particular, we obtain an alternative proof of  the
known equality  $\mathcal{C}_+(B)=\mathcal{D}(B)$ for types $A_n$ and $D_n$, and
also several results on non-simply laced types, for which the
representation-theoretic method is not yet fully available.  

From the explicit list of positive $c$-vectors and non-initial $d$-vectors
provided by Theorem \ref{thm:main} we deduce the following result. The
statements (\ref{thm:corollaries-1}) and (\ref{thm:corollaries-3}) generalize
to all finite types properties known only for simply-laced types (cf.
Corollaries \ref{cor:schur} and \ref{cor:independent1}).
\begin{thm}
	\label{thm:corollaries}
	Let $B$ be any skew-symmetrizable matrix of cluster finite type.
	\begin{enumerate}
		\item 
			\label{thm:corollaries-1}
			All $c$-vectors and $d$-vectors of $\mathcal{A}_\bullet(B)$ are roots of
			the root system of $A(B)$. For simply-laced types they are Schur roots.
	
		\item
			\label{thm:corollaries-2}
			A $c$-vector ($d$-vector) of $\mathcal{A}_\bullet(B)$ is a real root if
			and only if its support in $X(B)$ is a tree.

		\item
			\label{thm:corollaries-3}
			The cardinality $|\mathcal{C}_+(B)|=|\mathcal{D}(B)|$ depends only on the
			cluster type $Z$ of $B$ and it is equal to the number of positive roots in
			the root system of type $Z$. Explicitly it is equal to $nh/2$, where $n$
			and $h$ are the rank and the Coxeter number of type $Z$ (see Table
			\ref{tab:coxeter_numbers}).

		\item
			\label{thm:corollaries-4}
			The set
			$\mathcal{C}_+(B)=\mathcal{D}(B)$ only depends on $A(B)$, the Cartan
			counterpart of $B$.
	\end{enumerate}
\end{thm}
\begin{table}[htbp]
	\caption{Coxeter numbers and numbers of positive roots.}
	\label{tab:coxeter_numbers}
	\centering
	\begin{tabular}{|c||c|c|c|c|c|c|c|c|c|}
		\hline
		Type	&	$A_n$	&	$B_n$	&	$C_n$	&	$D_n$	&	$E_6$	&	$E_7$	&	$E_8$	&	$F_4$	& $G_2$ \\
		\hline
		$h$		&	$n+1$	&	$2n$	&	$2n$	&	$2n-2$&	$12$	&	$18$	&	$30$	&	$12$	&	$6$		\\
		\hline
		$nh/2$&	$n(n+1)/2$	& $n^2$	&	$n^2$	&	$n(n-1)$	&	$36$	&	$63$	&	$120$	& $24$	&	$6$	\\
		\hline
	\end{tabular}
\end{table}

While proving Theorem \ref{thm:main} we also obtain the following interesting
result. A skew-symmetrizable
integer matrix $B$ is said to be \emph{bipartite} if the corresponding valued quiver
has only sinks and sources; by extension a seed whose $B$-matrix is
bipartite is also called bipartite.

\begin{thm}
	Let $B$ be any skew-symmetrizable matrix of cluster finite type. Any
	$c$-vector ($d$-vector) of $\mathcal{A}_\bullet(B)$ occurs in a bipartite
	seed.
	\label{thm:bipartite}
\end{thm}

This paper is structured as follows. In Section \ref{sect:back} we give more
background and a short survey of the known results on $c$- and $d$-vectors and
their consequences in order to connect our result  to representation theory of
quivers.
In Section \ref{sect:sets} we describe the sets $\mathcal{X}(Z)$ and
$\mathcal{W}(Z)$ for all the classical finite type $Z$ (i.e. for $A_n$, $B_n$,
$C_n$ and $D_n$). We postpone the exceptional types to Appendix
\ref{app:exceptional} due to their length.

The proofs of Theorems \ref{thm:main} and \ref{thm:bipartite} 
for classical types are split into several
Propositions and use different techniques. In Section \ref{sect:surfaces-ad} we
use the surface realization (\cite{Fock,Fomin08,Fomin08b})
of cluster algebras to prove the results for
types $A_n$ and $D_n$. In Section \ref{sect:folding-bc} we extend the folding
construction of \cite{Dupont08} to deal with types $B_n$ and $C_n$.

The paper is concluded by Section \ref{sect:corollaries} where we prove Theorem
\ref{thm:corollaries}. In Appendix \ref{app:type-Dn} we add the complete analysis needed in the
proof of Propositions \ref{prop:An_Dn-support} and \ref{prop:bipartite}.

\section{More background}
\label{sect:back}
Let us give more background and a short survey of the known results on $c$- and
$d$-vectors and their consequences in order to connect our result  to
representation theory of quivers.  We also propose a conjecture on the root
property of $c$-vectors.

\subsection{\texorpdfstring{$c$}{c}-vectors and \texorpdfstring{$d$}{d}-vectors}
\label{sect:back.vectors}
We quickly recall the definitions and the basic properties of $c$-vectors and
$d$-vectors, which are the main subject of this paper.  All the formulas are
taken from \cite{Fomin07}.

Let $\mathbb{Q}(x)$ be the rational function field of algebraically independent
variables $x=\{x_i\}_{i\in I}$ over $\mathbb{Q}$, and let
$\mathbb{Q}_+(x)$ be the subset of $\mathbb{Q}(x)$ which consists
of the functions having subtraction-free expressions.  The set
$\mathbb{Q}_+(x)$ is a semifield, and it is called the {\em
universal semifield of $x$.} We also introduce the {\em tropical semifield
$\mathbb{P}_{\mathrm{trop}}(x)$ of $x$} as the multiplicative free abelian group
generated by $x$ with the addition $\oplus$ defined by
\begin{align}
\prod_{i\in I} x_i ^ {a_i}
\oplus 
\prod_{i\in I} x_i ^ {b_i}
:=
\prod_{i\in I}x_i ^ {\min(a_i,b_i)}.
\end{align} 
Let $\pi_{\mathrm{trop}}: \mathbb{Q}_+(x)\rightarrow
\mathbb{P}_{\mathrm{trop}}(x)$ be the canonical homomorphism, $x_i \mapsto x_i$,
$c \mapsto 1$ ($c\in \mathbb{Q}_{+}$).

We first describe the $d$-vectors.  Since the presence of coefficients is
irrelevant, for simplicity, we describe them for a cluster algebra with trivial
coefficients.  As usual, we start from the initial seed $(B,x)$ with a given
skew-symmetrizable integer matrix $B$ and a tuple of algebraically independent
variables $x=\{x_i\}_{i\in I}$ called the {\em initial cluster variables}.  We
obtain a new seed $(B',x')$ by the mutation at $k$,
\begin{align}
 \label{eqn:b-mutation}
 b'_{ij}&=
 \begin{cases}
 -b_{ij} & \mbox{$i=k$ or $j=k$}\\
 b_{ij}
 + b_{ik}[b_{kj}]_+
 +[-b_{ik}]_+  b_{kj}  & i,j\neq k,
 \end{cases}
 \\
 \label{eq:xrel}
 x'_i &=
 \begin{cases}
 \displaystyle
  x_k{}^{-1}
  \left(
\prod_{j\in I}
 x_j^{[b_{jk}]_+}
 +
 \prod_{j\in I}
 x_j^{[-b_{jk}]_+}
 \right)
 & i=k\\
 x_i
 & i\neq k,
 \end{cases}
\end{align}
where $[a]_+=a$ for $a>0$ and $0$ otherwise.  The elements
obtained by sequences of mutations from $x$ are called {\em cluster variables}.  They
are in $\mathbb{Q}_+(x)$ since the right hand side of (\ref{eq:xrel}) is
subtraction-free. For any cluster variable $x_j'$ in some cluster
$x'=\{x'_i\}_{i\in I}$, we define the corresponding $d$-vector
$d'_j=(d'_{ij})_{i\in I}$ by
\begin{align}
	\pi_{\mathrm{trop}}(x_j')=\prod_{i\in I} x_i^{-d'_{ij}}.
	\label{eqn:tropicalization-x}
\end{align}
The matrix $D'=(d'_{ij})_{i,j\in I}$ is called the \emph{$D$-matrix} of $x'$.
This definition of the $d$-vector $d'_j$ agrees with an alternative and more
familiar definition as the tuple of the exponents of the ``denominator'' of the
Laurent polynomial expression of $x'_j$,
\begin{align}
 x'_j
 =
 \frac{P(x)}{\prod_{i\in I}x_i^{d'_{ij}}},
\end{align}
where $P(x)$ is a polynomial in $x=\{x_i\}_{i\in I}$ not divisible by any $x_i$.
(Note that the celebrated Laurent phenomenon \cite{Fomin02} does not necessarily
imply that the components of the $d$-vector for a non-initial cluster variable
are all nonnegative.) For cluster variables $x''_j$ and $x'_j$ which are
connected by a mutation $(B'',x'')=\mu_k (B',x')$, we have a recursion relation
for the corresponding $d$-vectors, which is the tropicalization of
\eqref{eq:xrel},
\begin{align}
	\label{eqn:d-recursion}
  d''_{ij} &=
 \begin{cases}
 \displaystyle
  -d'_{ik}
  +
\max  \left(
\sum_{\ell\in I}
d'_{i\ell}{}  {[b'_{\ell k}]_+},
 \sum_{\ell\in I}
d'_{i\ell}  {[-b'_{\ell k}]_+}
\right)
 & j=k\\
 d'_{ij}
 & j\neq k.
 \end{cases}
\end{align}
 
Next we describe the $c$-vectors.  We need another tuple of
algebraically independent variables $y=(y_i)_{i\in I}$ called the {\em initial
coefficients}.  They mutate, along with the mutation of  the exchange
matrix $B$, with the exchange relation at $k$ given by
\begin{align}
 \label{eq:yrel}
 y'_{i}&=
 \begin{cases}
 y_{i}^{-1} &i=k\\
 \displaystyle
 y_{i}
 \frac{(1+y_k)^{[-b_{ki}]_+ }}
  {(1+y_k^{-1})^{[b_{ki}]_+}}
 & i \neq k.
 \end{cases}
\end{align}
The elements of $\mathbb{Q}_+(y)$ obtained by successive mutations 
are called {\em coefficients}. For any coefficient $y_j'$ in a
coefficient tuple $y'=(y'_i)_{i\in I}$, we define the corresponding $c$-vector
$c'_j=(c'_{ij})_{i\in I}$ by
\begin{align}
	\pi_{\mathrm{trop}}(y_j')=\prod_{i\in I} y_i^{c'_{ij}}.
	\label{eqn:tropicalization-y}
\end{align}
The matrix $C'=(c'_{ij})_{i,j\in I}$ is called the \emph{$C$-matrix} of $y'$.

For coefficients $y''_j$ and $y'_j$ which are connected by a mutation
$(B'',y'')=\mu_k (B',y')$, we have a recursion relation for the corresponding
$c$-vectors, which is the tropicalization of \eqref{eq:yrel},
\begin{align}
	\label{eqn:c-recursion}
 c''_{ij}&=
 \begin{cases}
 -c'_{ij} & j=k\\
 c'_{ij}
 + c'_{ik}[b'_{kj}]_+
 +[-c'_{ik}]_+  b'_{kj}  & j\neq k.
 \end{cases}
\end{align}
This definition of $c$-vectors agrees with an alternative and more familiar
definition as column vectors of the bottom half square matrix of the extended
exchange matrices of $\mathcal{A}_\bullet(B)$ (cf. (\ref{eqn:b-mutation})).

\subsection{Sign-coherence Conjecture}
Fomin and Zelevinsky made the following fundamental conjecture
on $c$- and $d$-vectors,
which plays an important role in the structure theory of
cluster algebras (e.g., \cite{Fomin07,Nakanishi11a}).

\begin{conj}[Sign-coherence Conjecture] 
\label{conj:sign}
Let $B$ be any skew-symmetrizable matrix.
\par
(i) \cite[Conjecture 5.5 \& Proposition 5.6]{Fomin07} Any $c$-vector
of $\mathcal{A}_\bullet(B)$ is a nonzero vector,
and its components are either all non-negative
or all non-positive.

(ii) \cite[Conjectures 7.4 \& 7.5]{Fomin07}
Any non-initial $d$-vector of $\mathcal{A}_\bullet(B)$ is a nonzero vector,
and its components are all nonnegative.
\end{conj}

The first part of the conjecture is equivalent to the fact that {\em the
constant term of any $F$-polynomial of $\mathcal{A}_\bullet(B)$ is one\/}
\cite{Fomin07}, which is proved for any skew-symmetric matrix $B$
\cite{Derksen10,Nagao10,Plamondon10b}, and also for a large class of
skew-symmetrizable matrices \cite{Demonet10}, in particular, for any
skew-symmetrizable matrix which is mutation equivalent to an acyclic one.

The second part of the conjecture is proved, for example, for any skew-symmetric
matrix $B$ arising from a surface \cite{Fomin08}, and more cases follow from
the results in the rest of this section.

\subsection{Root Conjecture}
Recall that a skew-symmetric matrix $B=(b_{ij})_{i,j\in I}$ can be identified
with a quiver $Q(B)$ without loops and 2-cycles by attaching $b_{ij}$ arrows
from vertex $i$ to vertex $j$ if $b_{ij}>0$.  This correspondence can be
extended to the one between {\em skew-symmetrizable\/} matrices and {\em valued
quivers} (see \cite{Dlab76}).
 
Let $\Delta(A)$ be the {\em root system\/} associated with a symmetrizable Cartan matrix $A$,
and let $\{\alpha_i\}_{i\in I}$ be its simple roots \cite{Kac90}.  A root
$\alpha=\sum_{i\in I}c_i\alpha_i$ of $\Delta(A)$ is naturally identified with,
either all nonnegative or all non-positive, nonzero integer vector
$(c_i)_{i\in I}$. 
It is said to be {\em real\/} if there is an element
$w$ of the Weyl group of $\Delta(A)$ such that $w(\alpha)$ is a simple root; otherwise it
is said to be {\em imaginary}.  It is known that a root $\alpha$ is 
real if and only if $(\alpha,\alpha)_{TA}= {}^t\alpha TA \alpha >0$, where $T$
is any diagonal matrix with positive diagonal entries such that $TA$ is symmetric.
See \cite{Kac90} for details.

In the study of cluster algebras, it becomes more and more apparent that there
is some intimate interplay among three kinds of algebras, namely, cluster
algebras, path algebras, and  (quantized) Kac-Moody algebras.  Naturally, root
systems provide the common underlying structure.  The starting point of the
interplay is Kac's theorem, which generalizes celebrated Gabriel's theorem.  Let
$k$ be an algebraically closed field below.

\begin{thm}[{{\bf Kac's Theorem} \cite{Kac80,Kac82}}]
\label{thm:Kac}
Let $B$ be any skew-symmetric matrix.  Then, there exists an indecomposable
module of the path algebra $kQ(B)$ with dimension vector $\alpha$ if and only if
$\alpha$ is a positive root of $ \Delta(A(B))$.
\end{thm}
In the above correspondence, if a positive root is the dimension vector of some
indecomposable $kQ(B)$-module $M$ such that $\mathrm{End}_{kQ(B)}(M)=k$, then it
is called a {\em Schur\/} root.  We use this notion later.

In view of cluster algebras, the extension of Theorem \ref{thm:Kac} to the
valued quivers is desired and expected.  Unfortunately, it is not fully
achieved yet \cite{Hubery02,Deng08}.  Nevertheless, the perspective presented
above guides us to  the following natural refinement of Conjecture
\ref{conj:sign}, jointly proposed with Andrei Zelevinsky.

\begin{conj}[Root Conjecture] 
\label{conj:root}
For any skew-symmetrizable matrix $B$ any $c$-vector
of $\mathcal{A}_\bullet(B)$ is a root of $\Delta(A(B))$.
\end{conj}

As for $d$-vectors, they also satisfy the same root property in many known
cases. However Marsh and Reiten recently found, in cluster affine type $A$, an
example of a $d$-vector which is not a root of $\Delta(A(B))$ \cite{Marsh12}. We
thank Robert Marsh and Idun Reiten for sharing with us this counterexample.

\subsection{Results for finite type}
Cluster algebras of finite type were studied in detail by various authors.  Here
we collect some of the known properties of their $c$- and $d$-vectors along with
some consequences which are relevant to the present paper.  For simplicity, we
assume that a skew-symmetrizable matrix $B$ is {\em indecomposable\/} in this
section.

The connection between the $d$-vectors and the root systems of finite type was
first discovered by Fomin and Zelevinsky \cite{Fomin03a}. Recall that a skew-symmetrizable
integer matrix $B$ is said to be \emph{bipartite} if the corresponding valued quiver
has only sinks and sources.

\begin{thm}[{\cite[Theorem 1.9]{Fomin03a}}]
\label{thm:dvecfin}
For any skew-symmetrizable bipartite matrix $B$ whose Cartan counterpart $A(B)$
is of finite type, the set  $\mathcal{D}(B)$ coincides with the set of all the
positive roots of $\Delta(A(B))$.
\end{thm}
The requirement of $B$ being bipartite was lifted later on in \cite{shih}.
In particular, in the  skew-symmetric case, combining the above result  with
Gabriel's theorem, we get that the set $\mathcal{D}(B)$ also coincides with the
set of all the dimension vectors of the path algebra $kQ(B)$.  This result
triggered the intensive representation-theoretic study of cluster algebras in
the past decade.

For a skew-symmetric matrix $B$ of cluster finite type, let $\Lambda(B)$ be the
corresponding cluster-tilted algebra, which is the path algebra of the quiver
$Q(B)$ modulo the relations described by \cite[Theorem 4.2]{Buan05}.  Note that any
indecomposable $\Lambda(B)$-module can also be regarded as an indecomposable
$kQ(B)$-module.  Let $\mathrm{Dim}(\Lambda(B))$ be the set of the dimension
vectors of all the indecomposable $\Lambda(B)$-modules.

The following theorem by Caldero, Chapoton, and Schiffler \cite{Caldero04b}, and
by Buan, Marsh, and Reiten \cite{Buan04}, extended Theorem  \ref{thm:dvecfin} to
any skew-symmetric matrix $B$ of cluster finite type.

\begin{thm}[{\cite[Theorem 4.4 \& Remark 4.5] {Caldero04b},
\cite[Theorem 2.2]{Buan04}}]
\label{thm:d-fin}
For any skew-symmetric matrix $B$ of cluster finite type, the sets
$\mathcal{D}(B)$ and  $\mathrm{Dim}(\Lambda(B))$ coincide.
\end{thm}

On the other hand, N\'ajera Ch\'avez recently proved a parallel theorem for
$c$-vectors.

\begin{thm}[{\cite[Theorem 4] {Najera12},\cite{Najera12b}}]
\label{thm:c-fin}
For any skew-symmetric matrix $B$ of cluster finite type, the sets
$\mathcal{C}_+(B)$ and $\mathrm{Dim}(\Lambda(B))$ coincide.
\end{thm}
The inclusion $\mathcal{C}_+(B)\subset \mathrm{Dim}(\Lambda(B))$ is a special
case of  \cite[Theorem 4] {Najera12} (see Theorem \ref{thm:c-general}), while
the opposite inclusion is due to a yet unpublished result communicated to us by
Alfredo N\'ajera Ch\'avez  \cite{Najera12b}.

We have the following immediate corollary of Theorems \ref{thm:d-fin} and
\ref{thm:c-fin}.

\begin{cor} 
\label{cor:C=D}
For any skew-symmetric matrix $B$ of cluster finite type, the sets
$\mathcal{C}_+(B)$ and $\mathcal{D}(B)$ coincide.
\end{cor}

It is known that, for any indecomposable $\Lambda(B)$-module $M$,
$\mathrm{End}_{\Lambda(B)}(M)=k$ holds (and therefore
$\mathrm{End}_{kQ(B)}(M)=k$) \cite[Section 8]{Buan06}.  Thus, we have
another corollary of Theorems \ref{thm:d-fin} and \ref{thm:c-fin}.
\begin{cor} 
\label{cor:schur}
For any skew-symmetric matrix $B$ of cluster finite type, all positive
$c$-vectors and all non-initial $d$-vectors are Schur roots of $\Delta(A(B))$.
\end{cor}

For any skew-symmetric matrix $B$ of cluster finite type, let us introduce the
set
\begin{align}
\mathrm{Ind}(\Lambda(B))
=
\{\,
\mbox{all indecomposable $\Lambda(B)$-modules}
\,\}.
\end{align}

The following remarkable fact holds.

\begin{thm}[{\cite[Corollary 2.4]{Buan04}}]
\label{thm:independent}
For any skew-symmetric matrix $B$ of cluster finite type, the cardinality
$|\mathrm{Ind}(\Lambda(B))|$ only depends on the cluster type $Z$ of $B$; it
is equal to the number of positive roots of the root system of type $Z$.
\end{thm}

The dimension map
\begin{align}
	\label{eq:dim}
 	\underline{\mathrm{dim}}:
 	\mathrm{Ind}(\Lambda(B))
 	\rightarrow
	\mathrm{Dim}(\Lambda(B))
\end{align}
is surjective by definition.  Actually, it is bijective by the following
theorem.

\begin{thm} \cite[Theorem 1]{Ringel09}
\label{thm:ringel}
For any skew-symmetric matrix $B$ of cluster finite type, the map
$\underline{\mathrm{dim}}$ in \eqref{eq:dim} is injective.
\end{thm}

We have an immediate corollary of Theorems \ref{thm:d-fin}, \ref{thm:c-fin},
\ref{thm:independent}, and \ref{thm:ringel}.  

\begin{cor}
\label{cor:independent1}
For any skew-symmetric matrix $B$ of cluster finite type, the cardinality
$|\mathcal{C}_+(B)|= |\mathcal{D}(B)|$ only depends on the cluster type $Z$ of $B$,
and it is equal to the number of positive roots of the root system of type $Z$.
\end{cor}

\subsection{More general results}
For completeness, we summarize some general results on $c$- and $d$-vectors
beyond finite type and also give some examples, though we do not use them in the
rest of the paper.

A skew-symmetrizable matrix $B$ is {\em acyclic\/} if the corresponding valued
quiver $Q(B)$ is acyclic, i.e., without oriented cycles.  Let us first discuss
the case of an acyclic skew-symmetric matrix $B$.  Under this hypothesis, the
cluster tilted algebra $\Lambda(B)$ is the path algebra $kQ(B)$ itself because
there is no relation to be imposed.  A $kQ(B)$-module $M$ is said to be {\em
rigid\/} if $\mathrm{Ext}^1_{kQ(B)}(M,M)=0$.

The following two theorems completely describe the $c$- and $d$-vectors in this
case:

\begin{thm}[{\cite[Theorem 4] {Caldero05},
\cite[Theorem 2.3]{Buan05c}}]
\label{thm:d-acyclic}
For any acyclic skew-symmetric matrix $B$, the set  $\mathcal{D}(B)$ coincides
with the set of the dimension vectors of all the rigid indecomposable
$kQ(B)$-modules.
\end{thm}

\begin{thm}[{\cite[Theorem 1] {Najera12}}]
\label{thm:c-acyclic}
For any acyclic skew-symmetric matrix $B$, the set $\mathcal{C}_+(B)$ coincides
with the set of the dimension vectors of all the rigid indecomposable
$kQ(B)$-modules.
\end{thm}

Recall that, when $Q(B)$ is acyclic, the following formula holds \cite{Assem06}:
\begin{align}
\frac{1}{2}(\underline{\dim}\,M, \underline{\dim}\,M)_{A(B)}
=\dim \mathrm{End}_{kQ(B)}(M)
-
\dim \mathrm{Ext}^1_{kQ(B)}(M,M).
\end{align}
It follows that $\alpha$ is the dimension vector of a rigid indecomposable
$kQ(B)$-module if and only if it is a real Schur root.  Therefore, we have an
alternative form of Theorems \ref{thm:d-acyclic} and \ref{thm:c-acyclic}.

\begin{cor}[{\cite[Theorem 1] {Najera12}}]
\label{cor:c-acyclic}
For any acyclic skew-symmetric matrix $B$, both the sets $\mathcal{D}(B)$ and
$\mathcal{C}_+(B)$ coincide with the set of all the real Schur roots of
$\Delta(A(B))$.
\end{cor}

Both Theorem \ref{thm:c-acyclic} and Corollary \ref{cor:c-acyclic} are partially
extended to the acyclic skew-symmetrizable matrices.  (The sign-coherence of
$c$-vectors is covered by \cite{Demonet10}.)

\begin{thm}[{\cite[Theorem 1.1]{Reading11}, \cite[Theorem 1] {Speyer12}}]
\label{thm:c-acyclic2}
For any acyclic skew-symmetrizable matrix $B$, any positive $c$-vector is a real
positive root of $\Delta(A(B))$; moreover, it is the dimension vector of a rigid
indecomposable representation of the valued quiver $Q(B)$.
\end{thm}

Finally, beyond finite type and the acyclic case, the following result is so far
the most general result on $c$-vectors; in particular, it ensures and
strengthens Conjecture \ref{conj:root} for any skew-symmetric matrix $B$.

\begin{thm}[{\cite[Theorem 4] {Najera12}}]
\label{thm:c-general}
For any skew-symmetric matrix $B$, any positive $c$-vector of $\mathcal{A}_\bullet(B)$
is the dimension vector of some rigid indecomposable module $M$ of the Jacobian
algebra $J(Q(B),W)$ of the quiver $Q(B)$ with generic potential $W$ such that
$\mathrm{End}_{J(Q(B),W)}(M)=k$.  In particular, any positive $c$-vector of
$\mathcal{A}_\bullet(B)$ is a Schur root of $\Delta(A(B))$.
\end{thm}

On the other hand the behavior of the $d$-vectors is rather complicated as
studied in \cite{Buan07,Buan08b,Marsh12}.  What was observed therein is a
deficiency phenomenon: in some situations the $d$-vector of a cluster variable
$x'_i$ is smaller than  the  dimension vector of the rigid indecomposable
$\Lambda(B)$-module associated with $x'_i$.

We conclude this short survey by presenting  two illuminating examples beyond
finite type and  the acyclic case.

\begin{ex} {\em Type $A^{(1)}_2$.}
Consider the skew-symmetric matrix $B$ corresponding the following
non-acyclic quiver:
\[
\begin{xy}
(0,0)*\cir<2pt>{},
(5,8.7)*\cir<2pt>{},
(10,0)*\cir<2pt>{},
(-4,0)*{1},
(14,0)*{3},
(5,12)*{2},
\ar (4,7);(1,1.7)
\ar (9,1.7);(6,7)
\ar (2,0.7);(8,0.7)
\ar (2,-0.7);(8,-0.7)
\end{xy}
\]
It is mutation equivalent to the following acyclic quiver whose Cartan
counterpart is the Cartan matrix of affine type $A^{(1)}_2$.
\[
\begin{xy}
(0,0)*\cir<2pt>{},
(5,8.7)*\cir<2pt>{},
(10,0)*\cir<2pt>{},
(-4,0)*{1},
(14,0)*{3},
(5,12)*{2},
\ar (1,1.7);(4,7)
\ar (6,7);(9,1.7)
\ar (2,0);(8,0)
\end{xy}
\]
This cluster algebra $\mathcal{A}_\bullet(B)$ is studied in detail by \cite{Cerulli09}.
In particular, the non-initial $d$-vectors of $\mathcal{A}_\bullet(B)$ are given by
\cite[Lemma 3.3]{Cerulli09}:
\begin{align}
\label{eq:d-vecA12}
(0,1,0), (1,1,1), (a,0,a+1), (a+1,0,a), (a,1,a+1), (a+1,1,a), \quad a\geq 0.
\end{align}
Moreover, it is not difficult to show that the positive $c$-vectors of
$\mathcal{A}_\bullet(B)$ are also given by the same list.  Therefore, in this case
$\mathcal{C}_+(B)= \mathcal{D}(B)$ holds, even though $B$ is not acyclic.  Thus,
any non-initial $d$-vector is a Schur root of $\Delta(A(B))$ by Theorem
\ref{thm:c-general}.  Note that the $d$-vector $(1,1,1)$ in \eqref{eq:d-vecA12}
is the simplest example which shows the deficiency phenomenon \cite[Example
7.2]{Buan07}, where the dimension vector of the corresponding representation is $(1,2,1)$.  Nevertheless,
the $d$-vector $(1,1,1)$ is still a Schur root.  We also note that among the vectors in
\eqref{eq:d-vecA12}, the last two are imaginary roots for $a\geq 1$.
\end{ex}

\begin{ex} {\em Markov quiver.}
\label{ex:markov}
We consider the skew-symmetric matrix $B$ such that the corresponding quiver is
the following non-acyclic one.
\[
\begin{xy}
(0,0)*\cir<2pt>{},
(5,8.7)*\cir<2pt>{},
(10,0)*\cir<2pt>{},
(-4,0)*{1},
(14,0)*{3},
(5,12)*{2},
\ar (0.3,1.9);(3.3,7.2)
\ar (1.6,1.4);(4.6,6.7)
\ar (5.4,6.7);(8.4,1.3)
\ar (6.7,7.2);(9.7,1.9)
\ar (8,0.7);(2,0.7)
\ar (8,-0.7);(2,-0.7)
\end{xy}
\]
This is known as the {\em Markov quiver}, and the positive $c$-vectors of
$\mathcal{A}_\bullet(B)$ are given by  the permutations of the following vectors
\cite[Theorem 3.1.2]{Najera11}:
\begin{align}
\label{eq:c-vecMar}
(1,2,2), (a+1,b+1,a+b+1), (a-1, b-1, a+b-1), 
\end{align}
where $1 \leq a\leq b$, and $a$ and $b$ are coprime.  The cluster algebra
$\mathcal{A}_\bullet(B)$ has a surface realization by a once-punctured torus.
Using the same technique as in Section \ref{sect:surfaces-ad}, it can be shown
that the non-initial $d$-vectors are given by the permutations of the vectors in
\eqref{eq:c-vecMar} of the form
\begin{align}
\label{eq:d-vecMar}
(a-1, b-1, a+b-1). 
\end{align}
So this gives the first example in which the sets $\mathcal{D}(B) $ and
$\mathcal{C}_+(B)$ do not coincide.  Nevertheless, $\mathcal{D}(B) \subset
\mathcal{C}_+(B)$ so any non-initial $d$-vector is still a Schur root of
$\Delta(A(B))$ by Theorem \ref{thm:c-general}.
\end{ex}

The above examples may suggest that the property
$\mathcal{D}(B)\subset\mathcal{C}_+(B)$ holds in general but this is not true
due to the counterexample of \cite{Marsh12}.

\section{The sets \texorpdfstring{$\mathcal{X}(Z)$}{X(Z)} and
	\texorpdfstring{$\mathcal{W}(Z)$}{W(Z)} for classical types} 
\label{sect:sets}
Let $Z$ be any type in one of the four infinite families (i.e. $Z$ is one of 
$A_n$, $B_n$, $C_n$, or $D_n$ for some positive integer $n$). In this section we provide a
description of all the diagrams in $\mathcal{X}(Z)$ and define the list
$\mathcal{W}(Z)$ of allowed weighted diagram for each type required by Theorem
\ref{thm:main}. The analogous sets for the remaining finite types will be
presented in Appendix \ref{app:exceptional}.

\subsection{Type \texorpdfstring{$A_n$}{An}}
The following is a direct consequence of Proposition 2.4 in \cite{buan-vatne}.
\begin{prop}
	\label{prop:dynkin-An}
	A diagram $X$ is in $\mathcal{X}(A_n)$ if and only if the following
	conditions are satisfied:
	\begin{itemize}
		\item 
			$X$ has $n$ vertices, is simply laced and connected;

		\item
			every cycle in $X$ is a triangle;

		\item
			each vertex in $X$ has at most four neighbours;
		
		\item
			if a vertex has three neighbours then exactly two of them are adjacent;

		\item
			if a vertex has four neighbours then they can be partitioned into two
			disjoint sets,	containing two elements each, and such that the two
			neighbouring vertices $i$ and $j$ are adjacent if and only if $\{i,j\}$ is
			one of those sets.
	\end{itemize}
\end{prop}
An example of Dynkin diagram in $\mathcal{X}(A_n)$ 
is presented  in Figure \ref{fig:dynkin-An} to illustrate its ``quasi-tree''
nature.

\begin{figure}[htbp]
	\begin{center}
		\includegraphics[scale=\scalingconstant]{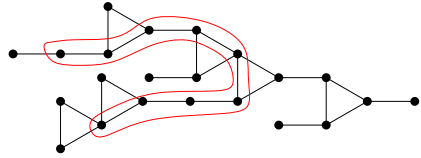}
	\end{center}
	\caption{A typical element $X$ in $\mathcal{X}(A_{22})$. The highlighted part is an
	element of $\mathcal{W}(A_{22})$ embedded in $X$. }
	\label{fig:dynkin-An}
\end{figure}

The set $\mathcal{W}(A_n)$ consists of type $A$ Dynkin diagrams (\emph{strings}) with at most $n$ vertices.  All
the multiplicities are $1$. Elements of $\mathcal{W}(A_n)$ are pictorially
presented as follows.
\begin{figure}[h]
	\begin{center}
		\includegraphics[scale=\scalingconstant]{allowed_diagrams_An.eps}
	\end{center}
\end{figure}

An example of an embedding of such a string in a
diagram of $\mathcal{X}(A_n)$ is highlighted in Figure \ref{fig:dynkin-An}. Note
that, as explained in the introduction, an embedding of an element of
$\mathcal{W}(A_n)$ in a diagram $X$ is given by a full sub-diagram; therefore at
most two vertices of each triangle of $X$ can belong to it. It follows that an
embedding of a string is uniquely determined by the positions of its endpoints
\cite{Parsons11}.
Note that the equality $\mathcal{D}(B)=\mathcal{V}(B)$ is known in this case by
\cite{Caldero04,Parsons11,Tran}.

The building block of Dynkin diagrams for classical types is given by diagrams
of type $A_n$. 
While stating the analogous results for other types we will use the convention 
$\mathcal{X}(A_0)=\emptyset$.

\subsection{Type \texorpdfstring{$B_n$}{Bn}}
As usual for Dynkin diagrams we put  $a_{ij}a_{ji}$ edges between $i$ and $j$
and the inequality sign on the edges refers to the relation among the lengths of
the corresponding simple roots. This convention  agrees with \cite{Kac90} and it
is the opposite to the convention
used in \cite{Bourbaki02}. To make it more explicit the Cartan matrix
\[
	\left( 
	\begin{array}{cc}
		2 & -1 \\
		-2 & 2 \\
	\end{array} 
	\right)
\]
corresponds in this paper to the following Dynkin diagram (labels correspond to
the rows of $B$).
	\begin{center}
		\includegraphics[scale=.7]{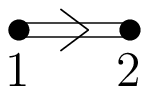}
	\end{center}
	\label{fig:convention}

\begin{prop}
	\label{prop:dynkin-Bn}
	A diagram with $n$ vertices ($n\ge2$) is in $\mathcal{X}(B_n)$ if and only if
	it is one of the two in Figure \ref{fig:dynkin-Bn} where $X^{(i)}$ is any
	diagram in $\mathcal{X}(A_m)$ for a suitable $m\ge0$.
\end{prop}
We postpone the proof to Section \ref{sect:folding-bc-dynkin}.

The weighted diagrams in $\mathcal{W}(B_n)$ are those in Figure
\ref{fig:allowed_diagrams_Bn}. As we will see they are obtained from (some of) those in
$\mathcal{W}(D_{n+1})$ by folding.

\begin{figure}[htbp]
	\begin{center}
		\includegraphics[scale=\scalingconstant]{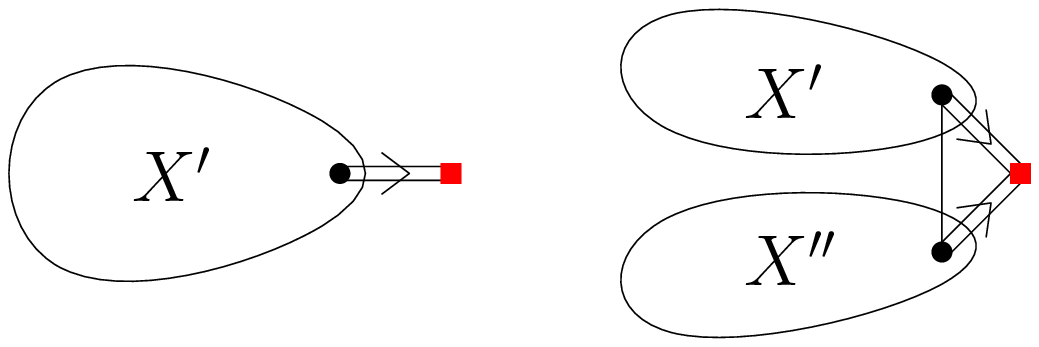}
	\end{center}
	\caption{Elements of $\mathcal{X}(B_n)$ for $n\ge2$; $X^{(i)}$ is any diagram
		in $\mathcal{X}(A_m)$ for a suitable $m\ge0$. The nodes marked as red
		squares  are the	images of those permuted by $\sigma$ in Proposition
		\ref{rk:our_cases}. }
	\label{fig:dynkin-Bn}
	\vspace{\baselineskip}
	\begin{center}
		\includegraphics[scale=\scalingconstant]{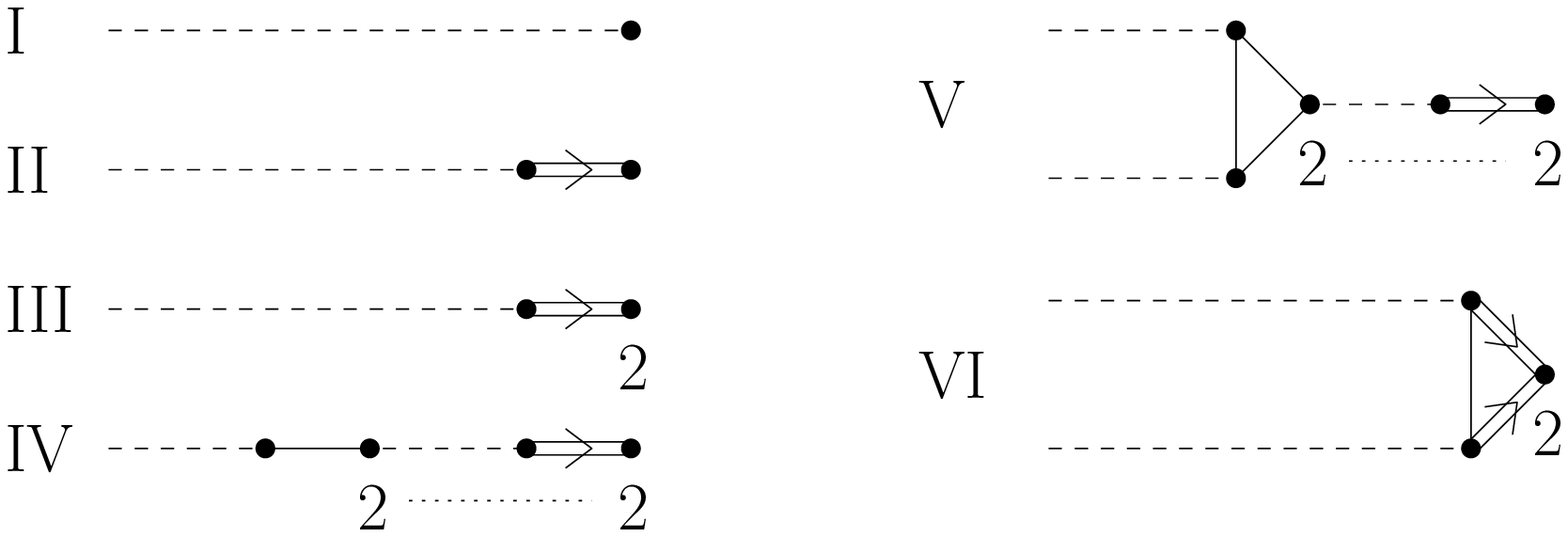}
	\end{center}
	\caption{The set $\mathcal{W}(B_n)$. Dotted lines are strings of any length;
	multiplicity of all the nodes of each such string are the same as their ending
	points. Solid lines can't be omitted. We will use the above drawing conventions
	thorough the rest of the paper. 
	}
	\label{fig:allowed_diagrams_Bn}
\end{figure} 

\subsection{Type \texorpdfstring{$C_n$}{Cn}}
\begin{prop}
	\label{prop:dynkin-Cn}
	A diagram with $n$ vertices ($n\ge2$) is in $\mathcal{X}(C_n)$ if and only if
	it is one of the two in Figure \ref{fig:dynkin-Cn} where $X^{(i)}$ is any
	diagram in $\mathcal{X}(A_m)$ for a suitable $m\ge0$. 
\end{prop}
We postpone the proof to Section \ref{sect:folding-bc-dynkin}.
\begin{rk}
	The results of Propositions \ref{prop:dynkin-Bn} and  \ref{prop:dynkin-Cn}
	were claimed in \cite{Musiker10} and encoded in the cluster algebra package of
	Sage. The details will appear in \cite{Stump}. The same result also appeared
	in \cite{Henrich}.
\end{rk}
The weighted diagrams in $\mathcal{W}(C_n)$ are those in Figure
\ref{fig:allowed_diagrams_Cn}; as we will see they are all the weighted diagrams that can be
obtained by folding a string embedded on a diagram in $\mathcal{X}(A_{2n-1})$.

\begin{figure}[htbp]
	\begin{center}
		\includegraphics[scale=\scalingconstant]{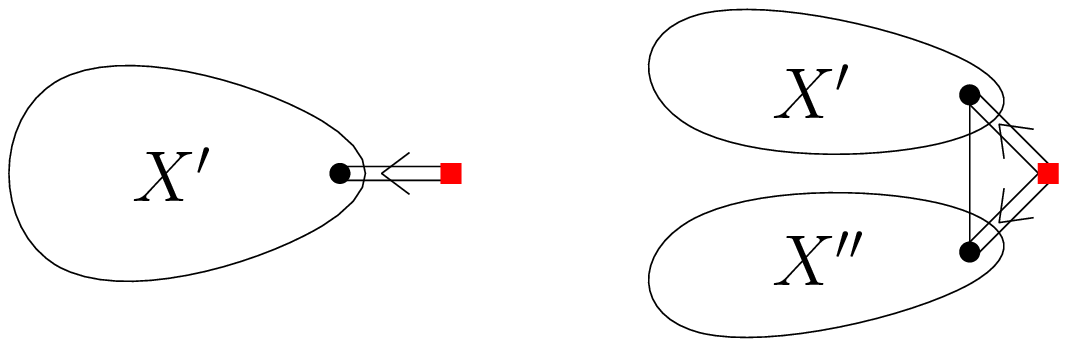}
	\end{center}
	\caption{Elements of $\mathcal{X}(C_n)$ for $n\ge2$; $X^{(i)}$ is any diagram
		in $\mathcal{X}(A_m)$ for a suitable $m\ge0$. The nodes marked as red
		squares are the images of the fixed point under the action of $\sigma$ in
		Proposition \ref{rk:our_cases}.}
	\label{fig:dynkin-Cn}
	\begin{center}
		\includegraphics[scale=\scalingconstant]{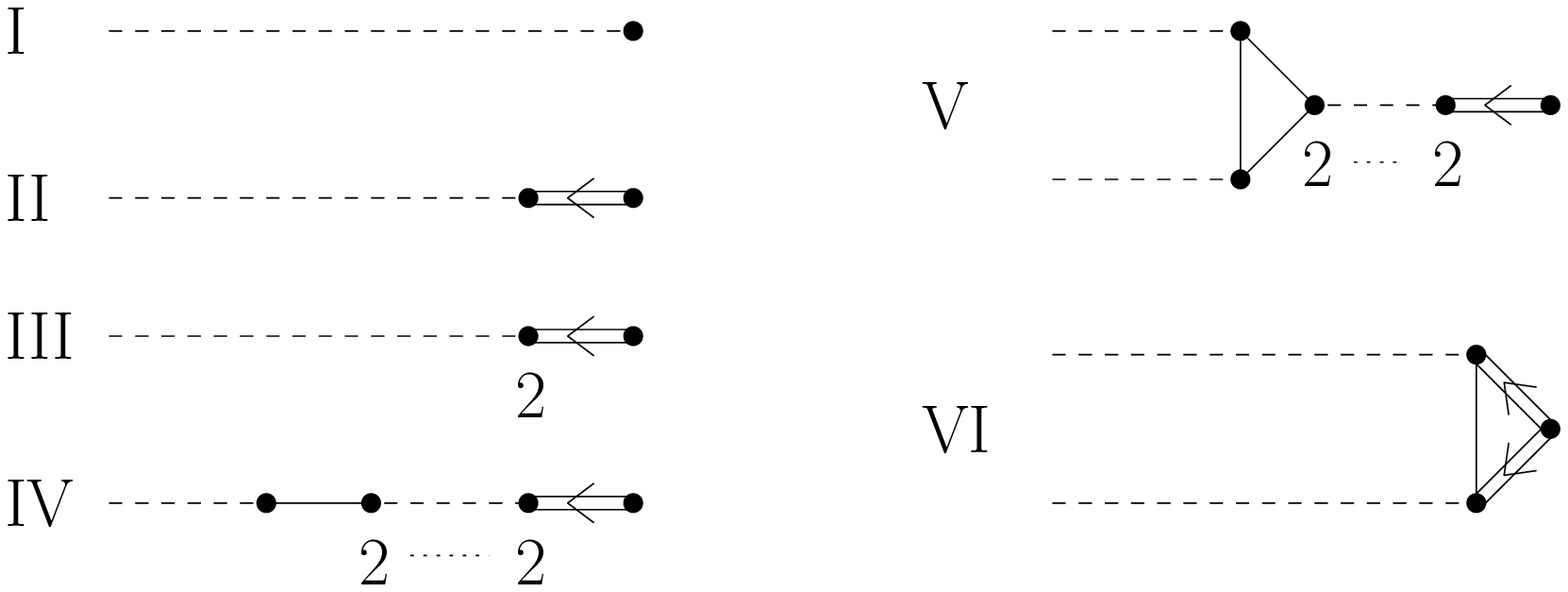}
	\end{center}
	\caption{The set $\mathcal{W}(C_n)$. We use the same drawing conventions of
	Figure \ref{fig:allowed_diagrams_Bn}.}
		\label{fig:allowed_diagrams_Cn}
\end{figure}

\subsection{Type \texorpdfstring{$D_n$}{Dn}}
From Theorem 3.1 in \cite{vatne} together with Proposition \ref{prop:dynkin-An}
we get the following description of $\mathcal{X}(D_n)$. Note that the same
result can also be obtained easily from the surface realization we use in
Section \ref{sect:surfaces-ad}.
\begin{prop}
	A diagram with $n$ vertices ($n\ge4$) is in $\mathcal{X}(D_n)$ if and only if
	it is one of the four in Figure \ref{fig:dynkin-Dn} where $X^{(i)}$ is any
	diagram in $\mathcal{X}(A_m)$ for a suitable $m\ge0$. 
\end{prop}
\begin{figure}[htbp]
	\begin{center}
		\includegraphics[scale=\scalingconstant]{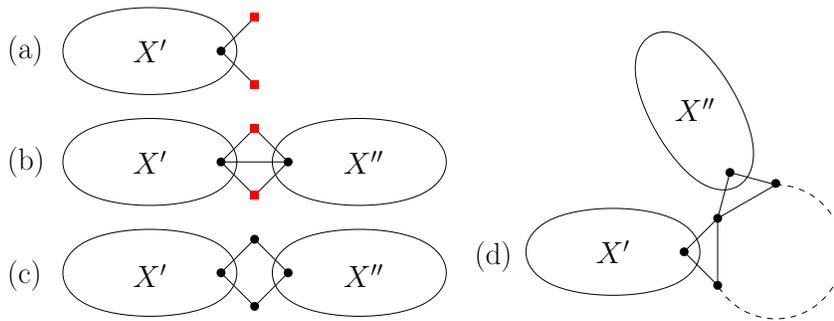}
	\end{center}
	\caption{Elements of $\mathcal{X}(D_n)$ for $n\ge4$; $X^{(i)}$ is any diagram
	in $\mathcal{X}(A_m)$ for a suitable $m\ge0$. The nodes marked as red
	squares are the one	permuted by $\sigma$ in Remark \ref{rk:our_cases}. Case
	(d) consists of a central cycle with, possibly, type-$A$ components attached
	to its sides.}
	\label{fig:dynkin-Dn}
\end{figure}

The set $\mathcal{W}(D_n)$ consists of all the weighted diagrams in Figure
\ref{fig:allowed_diagrams_Dn}.

\begin{figure}[htbp]
	\begin{center}
		\includegraphics[scale=\scalingconstant]{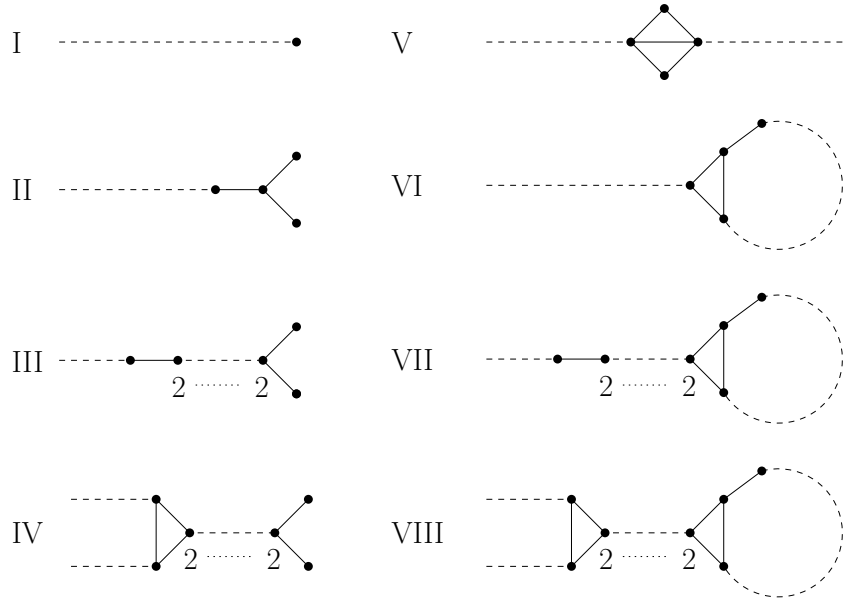}
	\end{center}
	\caption{The set $\mathcal{W}(D_n)$. We use the same drawing conventions of
	  Figure \ref{fig:allowed_diagrams_Bn}.}
		\label{fig:allowed_diagrams_Dn}
\end{figure}

\subsection{Examples}
To illustrate how to read the data presented in this section let us consider two
examples.
\begin{ex}
	Let $B$ be the matrix
	\[
	\left(
	\begin{array}[h]{ccccc}
		0 & 1 & 0 & 0 & 0 \\
		-1 & 0 & 1 & -1 & 0 \\
		0 & -1 & 0 & 1 & -1 \\
		0 & 1 & -1 & 0 & 1 \\
		0 & 0 & 1 & -1 & 0 \\
	\end{array}
	\right)
	\]
	of cluster type $D_5$. The diagram $X(B)$ and the set of positive
	$c$-vectors (and non-initial $d$-vectors) of $\mathcal{A}_\bullet(B)$ are
	shown in Figures \ref{fig:example_D5-dynkin} and \ref{fig:example_D5-diagrams}
	respectively. Note that any skew-symmetric matrix of cluster type $D_5$  whose
	entries are the same as the entries of $B$ in absolute value produces the same
	$X(B)$ and $\mathcal{V}(B)$.
	\label{ex:D5}
\end{ex}
\begin{figure}[htbp]
	\begin{center}
		\includegraphics[scale=\scalingconstant]{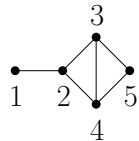}
	\end{center}
	\caption{$X(B)$ for Example \ref{ex:D5}. Labels refer to the rows of $B$.}
	\label{fig:example_D5-dynkin}
\end{figure}
\begin{figure}
	\begin{center}
      \includegraphics[scale=\scalingconstant]{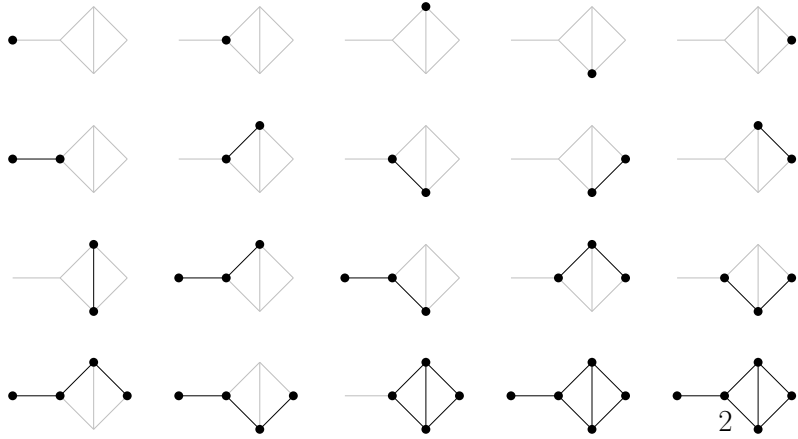}
    \end{center}
	\caption{$\mathcal{V}(B)$ for Example \ref{ex:D5}.}
	\label{fig:example_D5-diagrams}
\end{figure} 

\begin{ex}
	Let $B$ be the matrix
	\[
		\left(\begin{array}{rrrrr}
				2 & -1 & 0 & 0 & 0 \\
				-1 & 2 & -1 & 0 & 0 \\
				0 & -1 & 2 & -2 & -1 \\
				0 & 0 & -1 & 2 & -1 \\
				0 & 0 & -1 & -2 & 2
			\end{array}\right)
	\]
	of cluster type $C_5$. The diagram $X(B)$ and the set of positive
	$c$-vectors (and non-initial $d$-vectors) of $\mathcal{A}_\bullet(B)$ are
	shown in Figures \ref{fig:example_C5-dynkin} and \ref{fig:example_C5-diagrams}
	respectively. 
	\label{ex:C5}
\end{ex}
\begin{figure}[htbp]
	\begin{center}
		\includegraphics[scale=\scalingconstant]{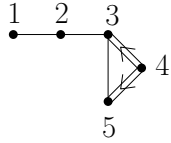}
	\end{center}
	\caption{$X(B)$ for Example \ref{ex:C5}. Labels refer to the rows of $B$.}
	\label{fig:example_C5-dynkin}
\end{figure}
\begin{figure}[htbp]
	\begin{center}
      \includegraphics[scale=\scalingconstant]{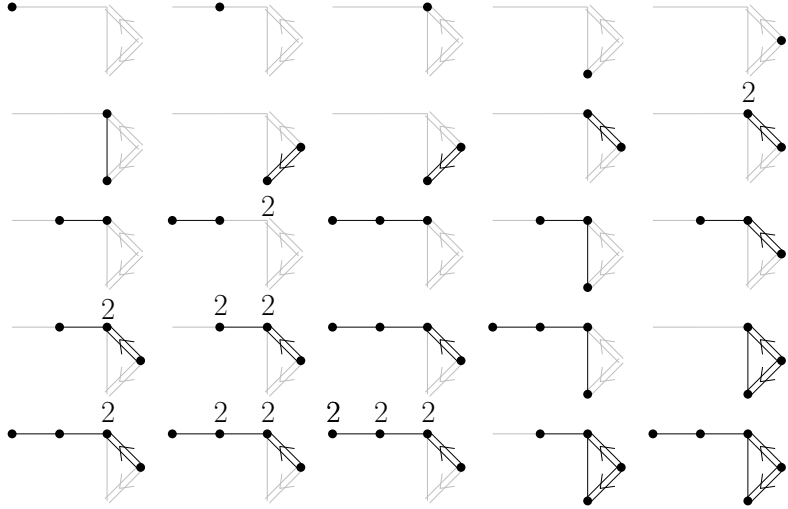}
    \end{center}
	\caption{$\mathcal{V}(B)$ for Example \ref{ex:C5}.}
	\label{fig:example_C5-diagrams}
\end{figure}

\section{Types \texorpdfstring{$A_n$}{An} and \texorpdfstring{$D_n$}{Dn}: the surface method}
\label{sect:surfaces-ad}
In this section we prove Theorem \ref{thm:main} for types $A_n$ and $D_n$.

\subsection{The surface method for types \texorpdfstring{$A_n$}{An} and
	\texorpdfstring{$D_n$}{Dn}.}
To describe $c$-vectors and $d$-vectors in types $A_n$ and $D_n$ we do
not need to use the construction of \cite{Fomin08} in its full generality so we can slightly
simplify the definitions; the reader interested in the general theory can find a
comprehensive review in \cite{msw-positivity_from_surfaces}.

Unless otherwise specified, by \emph{surface} $S$ we mean one of the following:
\begin{itemize}
	\item (type $A_n$)
		a disk with $n+3$ marked points on its boundary ($n\ge1$);
	\item (type $D_n$)
		a disk with $n$ marked points on the boundary ($n\ge4$) and one, the
		\emph{puncture}, in its	interior.
\end{itemize}
We denote the set of marked points by $M$.

\begin{defn}
	A \emph{(tagged) arc} is an homotopy class of curves $\gamma$ in the interior
	of $S\setminus M$ having no self intersections, connecting two distinct points
	of $M$, and not cutting out (together with a boundary component of $S$) an
	unpunctured bigon.  
	Due to the limitations imposed on the kinds of surfaces we consider there are
	only two possible types of arcs: \emph{chords}, connecting two marked point on
	the boundary of $S$, and \emph{radii}, connecting a point on the boundary with
	the puncture. Radii comes in two flavours: \emph{plain} and \emph{notched}; to
	distinguish them in figures we will put a cross on notched arcs.
	\label{defn:arc}
\end{defn}

\begin{rk}
	Note that this is not the usual definition of tagged arcs, in particular for
	general surfaces there is a tagging attached to each endpoint of any $\gamma$.
	Another difference from the general case is that we are not allowing loops
	(arcs with coinciding endpoints).
\end{rk}

We need not consider ideal arcs as defined by \cite{Fomin08} so we can drop the
adjective ``tagged'' without generating confusion.  To any pair of arcs $\gamma$
and $\delta$ we can associate an integer as follows.
\begin{defn}[{\cite[Definition 8.4]{Fomin08}}] 
	The \emph{intersection pairing} of $\gamma$ and $\delta$ is the integer
	$(\gamma|\delta)$ defined according to these rules:
	\begin{enumerate}
		\item 
			if $\gamma$ and $\delta$ coincide then $(\gamma|\delta)=-1$;
		\item
			if $\gamma$ and $\delta$ are homotopic radii with different tagging then
			$(\gamma|\delta)=0$;
		\item
			if $\gamma$ and $\delta$ are non-homotopic radii then $(\gamma|\delta)=0$
			if they are tagged in the same way and $(\gamma|\delta)=1$ if their
			tagging is different;
		\item
			in any other case, set $(\gamma|\delta)$ to be the minimal number of
			intersections between $\gamma$ and $\delta$.
	\end{enumerate}
	\label{defn:intersection_pairing}
\end{defn}

Two arcs $\gamma$ and $\delta$ are said to be \emph{compatible} if their
intersection pairing is non-positive.  A \emph{triangulation} $\Gamma$ of $S$ is
a maximal (by inclusion) set of pairwise compatible arcs.

\begin{rk}
	Definition \ref{defn:intersection_pairing} is symmetric; this is not the case
	for a general surface where loops are allowed (see \cite[Example
	8.5]{Fomin08}).
\end{rk}

In view of \cite[Theorem 7.9]{Fomin08} each triangulation of $S$
has $n$ arcs in it and given a triangulation $\Gamma$ and one of its arcs $\gamma$,
there is a unique other arc $\gamma'$ such that 
\[
\Gamma'=
\left(\Gamma\setminus\left\{ \gamma \right\}  \right)\cup\left\{ \gamma' \right\}
\]
is again a triangulation of $S$. The operation of replacing $\gamma$ with
$\gamma'$ is called a \emph{flip}.

To any triangulation $\Gamma$ associate a skew-symmetric  matrix
$B(\Gamma)=(b^\Gamma_{\gamma\delta})_{\gamma,\delta\in\Gamma}$ setting
\begin{equation}
	b^\Gamma_{\gamma\delta}:=\left\{
	\begin{array}[]{ll}
		1 & \mbox{if $\gamma$ rotates counterclockwise to $\delta$} \\
		-1 & \mbox{if $\gamma$ rotates clockwise to $\delta$} \\
		0 & \mbox{if both or none of the previous conditions hold}
  \end{array}
	\right.
	\label{eqn:b-matrix_from_triangulation}
\end{equation}
where $\gamma$ is said to rotate counterclockwise (resp. clockwise) to $\delta$
if they are not homotopic, they share an endpoint and, in a neighbourhood of this
point, $\gamma$ can be deformed counterclockwise (resp. clockwise), without
crossing any other arc of $\Gamma$, to coincide with $\delta$. 

\begin{figure}[htbp]
	\begin{center}
  	\includegraphics[scale=.6]{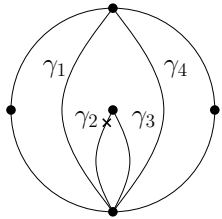}
	\end{center}
	\caption{In this triangulation $b_{21}=b_{31}=b_{14}=b_{42}=b_{43}=1$ while
	$b_{23}=0$.}
    \label{fig:b-matrix}
\end{figure}

By \cite[Theorem 7.11]{Fomin08} the above assignment
produces a bijection between triangulations of a type $A_n$ (resp. $D_n$)
surface and unlabeled seeds of the coefficient-free cluster algebra of the same
type. In particular cluster variables are in bijection with arcs and if two
seeds are obtained from one another exchanging the cluster variables $x_\gamma$
and $x_{\gamma'}$ then the corresponding triangulations are related by the flip
of $\gamma$ into $\gamma'$.

To keep track of principal coefficients 
we use \emph{laminations} as explained in \cite{Fomin08b}. 
For each marked point $p$ on the boundary of $S$
fix a neighbouring point $p'$ obtained sliding $p$ clockwise on the boundary.
\begin{defn}
	(see Figure \ref{fig:elementary-laminations}) The \emph{elementary lamination}
	$\lambda_\gamma$ corresponding to an arc $\gamma$ is the homotopy class of
	curves, contained in a
	neighbourhood of $\gamma$, defined as follows:
	\begin{itemize}
		\item 
			if $\gamma$ is a chord connecting $p$ and $q$ then $\lambda_\gamma$ connects
			$p'$ and $q'$;

		\item
			if $\gamma$ is a radius tagged plain (resp. notched) starting from $p$
			then $\lambda_\gamma$ starts from $p'$ and winds counterclockwise (resp.
			clockwise) infinitely many times around the puncture.
			
	\end{itemize}
  \begin{figure}[htbp]
		\begin{center}
    	\includegraphics[scale=.5]{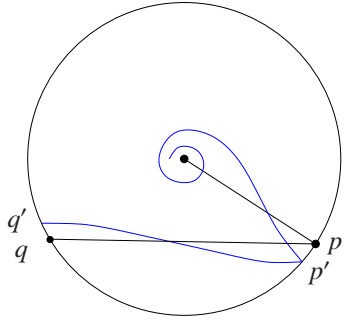}
		\end{center}
    \caption{Examples of elementary laminations.}
    \label{fig:elementary-laminations}
  \end{figure}
\end{defn}

The \emph{shear coordinates} of an elementary lamination $\lambda$ with respect
to a triangulation $\Gamma$ are the integers in the $n$-tuple
$(b_{\lambda,\gamma}^\Gamma)_{\gamma\in\Gamma}$ defined in terms of
intersections between $\lambda$ and the unique quadrilateral in $\Gamma$ of
which $\gamma$ is the diagonal. 

More precisely assume, at first, that $\Gamma$
contains at most one notched radius; each segment of $\lambda$ cutting through
the quadrilateral enclosing $\gamma$ as in Figure \ref{fig:shear-coordinates}
contributes either $+1$ or $-1$ to $b_{\lambda,\gamma}^\Gamma$. All other
crossings do not contribute. Note in particular that, if $\gamma$ is a radius of a
digon then, to have have nonzero shear coordinate, a	lamination as to "go around 
the puncture". When the digon has two non homotopic
radii this means that the lamination has to intersect both of them; in the other
case the lamination has to cross the dotted line joining the puncture to the
boundary. We will continue to draw this dotted line whenever we have a digon with
homotopic radii. Note also that flipping $\gamma$ interchanges positive
and negative crossings.  

To extend the definition to all possible triangulations
it suffices to impose that, if $\Gamma^\vee$ is obtained from $\Gamma$ by
changing all the tags at the puncture and $\lambda^\vee$ is obtained from
$\lambda$ inverting its winding direction (if any), then for any
$\gamma\in\Gamma$
\begin{align}
	b_{\lambda^\vee,\gamma^\vee}^{\Gamma^\vee}=
	b_{\lambda,\gamma}^\Gamma.
	\label{eqn:change_tagging}
\end{align}
\begin{figure}[htbp]
	\begin{center}
	  \includegraphics[scale=.6]{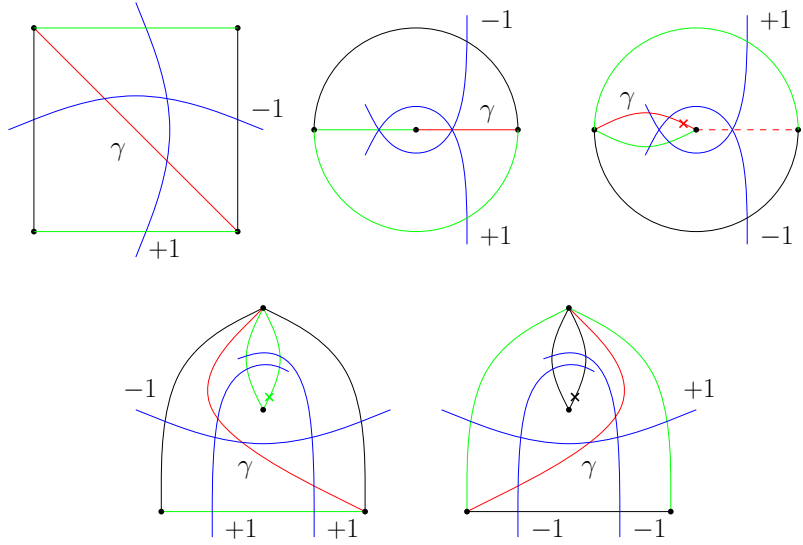}
	\end{center}
    \caption{Intersections giving non-zero shear coordinates. 
		The highlighted edges are those crossed by laminations $\lambda$ giving positive
		coordinates	$b^\Gamma_{\lambda,\gamma}$.}
    \label{fig:shear-coordinates}
\end{figure}

Given a triangulation $\Gamma$ let $\Lambda(\Gamma)=\left\{ \lambda_\gamma
\right\}_{\gamma\in \Gamma}$
be the \emph{multilamination} associated to it, i.e. the collection of the
elementary laminations corresponding to the arcs of $\Gamma$. Let
$\widetilde{B}_{\Gamma}(\Gamma')$ be the
extended $B$-matrix having top part $B(\Gamma')$ defined by
(\ref{eqn:b-matrix_from_triangulation}) and bottom part given by the shear
coordinates of $\Lambda(\Gamma)$ with respect to $\Gamma'$.

\begin{prop}[{\cite[Proposition 16.3]{Fomin08b}}]
	In the principal-coefficients cluster algebra
	$\mathcal{A}_\bullet\left(B(\Gamma)\right)$ the extended exchange matrix
	corresponding to the triangulation $\Gamma'$ is given by the above
	$\widetilde{B}_{\Gamma}(\Gamma')$.
\end{prop}

We can now describe the sets $\mathcal{C}(B)$ and $\mathcal{D}(B)$.
For the rest of this section fix a skew-symmetric integer matrix $B$ of type
$A_n$ or $D_n$.  Let $\Lambda_0=\left\{ \lambda_i \right\}_{i\in I}$ be the
multilamination corresponding to a triangulation
$\Gamma_0=\left\{\gamma_i\right\}_{i\in I}$ of $S$ realizing $B$. In
view of the last Proposition the set of $c$-vectors of the principal-coefficients
cluster algebra $\mathcal{A}_\bullet(B)$ is
\[
\mathcal{C}(B)=\left\{ 
c_{\gamma,\Gamma}:=\left(b_{\lambda_i,\gamma}^{\Gamma}\right)_{i\in I}
\right\}
\]
as $\Gamma$ runs over all possible triangulations of $S$ and $\gamma$ is an arc
in $\Gamma$.
The parametrization of $\mathcal{C}(B)$ by pairs of arcs and triangulations is
not one to one; indeed for any given $c$-vector there are in
general many pairs $\gamma$, $\Gamma$ realizing it. We will see that
$\Gamma$ can always be chosen to be bipartite (see Proposition
\ref{prop:bipartite}). 

As we already noted, in a cluster algebra coming from a surface, cluster variables
are in bijection with tagged arcs. Their denominator vectors can be read
directly from the surface: they are given in terms of their intersection pairing
with the arcs of the initial triangulation.

\begin{thm}[{\cite[Theorem 8.6]{Fomin08},\cite[Theorem 3.4]{Gekhtman}}]
	\label{thm:description_of_d-vectors}
	Let $\mathcal{A}_\bullet(B)$ be any cluster algebra of type $A_n$ or $D_n$ and
	let $\Gamma_0=\{\gamma_i\}_{i\in I}$ be a triangulation corresponding to
	$B$.
	
	If $x_{\gamma}$ is the cluster variable corresponding to the tagged arc
	$\gamma$
	then its $d$-vector is 
	\[
	d_{\gamma}=\left((\gamma_i|\gamma)\right)_{i\in I}.
	\]
\end{thm}
The set of non-initial $d$-vectors of $\mathcal{A}_\bullet(B)$ is therefore
\[
	\mathcal{D}(B)=\left\{
	d_{\gamma}=\left((\gamma_i|\gamma)\right)_{i\in I}
	\right\}
\]
as $\gamma$ runs over all arcs of $S$ not in $\Gamma_0$.

\subsection{Proof of Theorem \ref{thm:main}}
We begin by providing an alternative and immediate proof of
(\ref{eqn:sign-coherence}) for types $A_n$ and $D_n$.
\begin{lem}
	All the vectors in $\mathcal{C}(B)$ are sign-coherent. 
	\label{lemma:c-vectors_are_sign_coerent}
\end{lem}
\begin{proof}
	By contradiction let $c_{\gamma,\Gamma}$ be a $c$-vector that is not sign-coherent
	i.e. there are two elementary laminations in $\Lambda_0$, say $\lambda_i$ and
	$\lambda_j$ such that 
	$b^\Gamma_{\lambda_i,\gamma}>0$ and $b^\Gamma_{\lambda_j,\gamma}<0$. 
	
	Assume at first that $\Gamma$ contains at most one
	notched arc. From Figure \ref{fig:shear-coordinates} it is clear that
	$\lambda_i$ and $\lambda_j$ intersect and if they both spiral to the puncture
	then they do not come from homotopic radii. This is in contradiction with the
	hypothesis that $\Lambda_0$ came from a triangulation of $S$: the intersection
	pairing of the arcs	corresponding to $\lambda_i$ and $\lambda_j$ is positive.

	The results extends immediately to all the possible triangulation if we
	observe that changing the windings of all the laminations spiraling to the
	puncture does not affect the intersection relations among elements of
	$\Lambda_0$. 
\end{proof}

Note that, if $c_{\gamma,\Gamma}$ is a $c$-vector and $\Gamma'$ is the
triangulation obtained from $\Gamma$ by flipping $\gamma$ into $\gamma'$, then
\[
c_{\gamma,\Gamma}=-c_{\gamma',\Gamma'}.
\]
From now on we concentrate on the set $\mathcal{C}_+(B)$ of positive $c$-vectors
of $\mathcal{A}_\bullet(B)$.

\begin{lem}
	The weighted diagram of any positive $c$-vector in $\mathcal{A}_\bullet(B)$ is
	connected.
	\label{lemma:support_is_connected}
\end{lem}
\begin{proof}
	By contradiction assume that the weighted diagram of $c_{\gamma,\Gamma}$ has
	two disjoint components. Let $i$ be a node in one of them and $j$ a node
	in the other such that they are at minimal distance in $X(B)$. By hypothesis 
	$i$ and $j$ are not adjacent. Let $\lambda_i$ and $\lambda_j$ be the
	corresponding elementary laminations in $\Lambda_0$. 
	
	Three cases are possible (in type $A_n$ only the last one occurs).
	\begin{enumerate}
		\item 
			If $\lambda_i$ and $\lambda_j$ have two endpoints in common then they
			spiral to the puncture in opposite directions. In this case, since both
			$b^\Gamma_{\lambda_i,\gamma}\neq 0$ and $b^\Gamma_{\lambda_j,\gamma}\neq
			0$, the arc $\gamma$ cannot be incident to the puncture. The multilamination
			$\Lambda_0$ contains then a bigon enclosing $\lambda_i$ and $\lambda_j$;
			at least one side of this bigon (say $\lambda_k$) crosses positively the
			quadrilateral enclosing $\gamma$.

		\item
			If $\lambda_i$ and $\lambda_j$ share exactly one endpoint, since $i$ and
			$j$ are not adjacent, there are two possible configurations.
			If there is no other lamination sharing that endpoint then they 
			both spiral to the puncture and they are enclosed in a bigon; at
			least one side of this bigon (again say $\lambda_k$)
			intersects positively the quadrilateral enclosing
			$\gamma$. Otherwise at least one  lamination
			$\lambda_k$ among those sharing the same endpoint is such that 
			$b^\Gamma_{\lambda_k,\gamma}>0$.

		\item
			Finally if $\lambda_i$ and $\lambda_j$ do not share any endpoint then
			there is at least one lamination $\lambda_k$ starting from one of those
			four points,  lying in between $\lambda_i$ and $\lambda_j$ and
			crossing positively the quadrilateral that encloses $\gamma$ (otherwise
			such an elementary lamination could be added to $\Lambda_0$ in
			contradiction to the assumption that the multilamination corresponds to a
			triangulation). 

	\end{enumerate}
	In all of the cases there is a vertex $k$ in between $i$ and $j$ such that the
	$k$-th component of $c_{\gamma,\Gamma}$ is non-zero in contradiction with the
	assumption of minimal distance between $i$ and $j$.
\end{proof}

\begin{prop}
	In types $A_n$ and $D_n$ we have
	\[
		\mathcal{C}_+(B)\subset\mathcal{V}(B).
	\]
	\label{prop:An_Dn-support}
\end{prop}
\begin{proof}
	We deal first with type $A_n$.
	It is clear that, having no puncture, any lamination $\lambda\in\Lambda_0$ can
	intersect any given	arc $\gamma$ at most once so
	$b^\Gamma_{\lambda,\gamma}\in\left\{ 0,1 \right\}$.  In view of Proposition
	\ref{prop:dynkin-An} it suffices to show that no $c$-vector can have a
	triangle in its weighted diagram. 
	But this follows directly from the fact that, since $S$ has no puncture, at
	least one of the sides of each triangle in $\Lambda_0$ does not intersect any
	given arc $\gamma$. 
	\begin{figure}[htbp]
		\begin{center}
			\includegraphics[scale=.35]{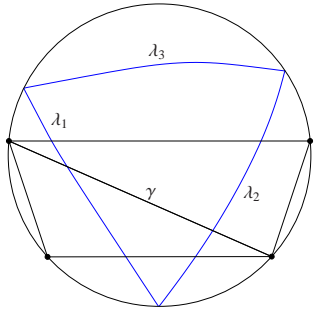}
		\end{center}
		\caption{Any triangle in a lamination of a surface of type $A_n$ intersects
			at most twice any arc $\gamma$.} 
		\label{fig:prop-An-support}
	\end{figure}

	For type $D_n$ the proof proceeds by case analysis. We need first some considerations. 
	In view of condition (\ref{eqn:change_tagging}) we can assume that the
	quadrilateral enclosing $\gamma$ is one of those in Figure
	\ref{fig:shear-coordinates}.

	Note that, given a multilamination $\Lambda_0$ coming from a triangulation, a once
	punctured disk can be decomposed into pieces: it will contain exactly one
	piece in which all the elementary laminations spiral to the puncture (one of
	the five in Figure \ref{fig:wheel-decomposition}); all the other pieces, if
	any, will contain only elementary laminations corresponding to chords.
	\begin{figure}[htbp]
		\begin{center}
	  	\includegraphics[scale=.7]{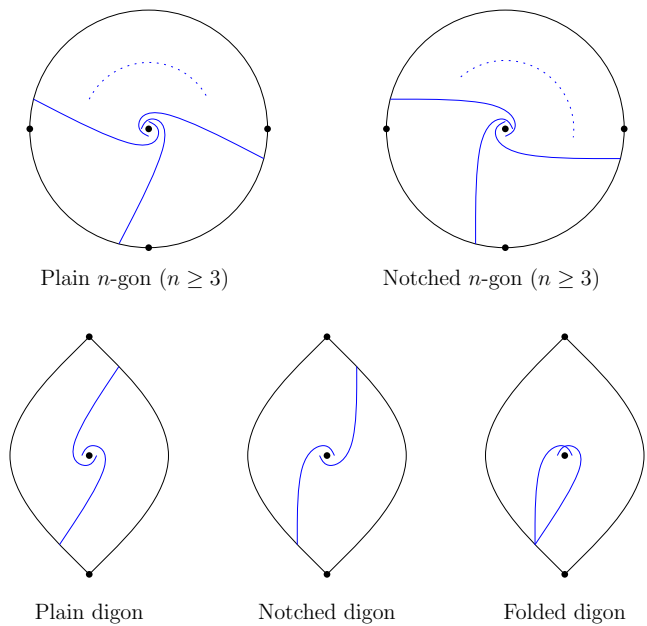}
		\end{center}
    \caption{Multilaminations with all elementary laminations spiralling to
			the puncture.}
    \label{fig:wheel-decomposition}
	\end{figure}
	Any such piece can only be glued to the one
	containing the puncture as shown in Figure
	\ref{fig:example_glueing}.
	\begin{figure}[htbp]
		\begin{center}
	  	\includegraphics[scale=.4]{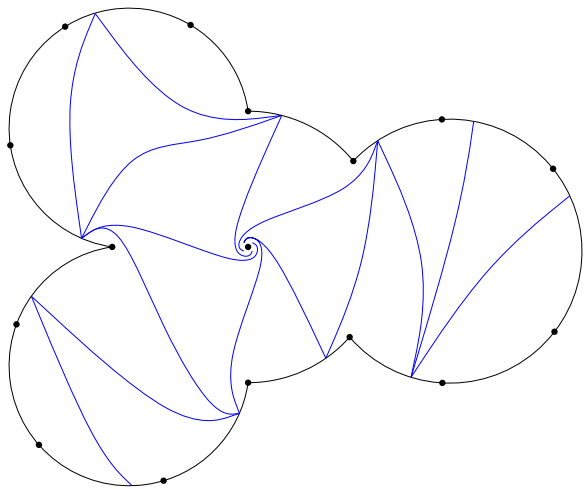}
		\end{center}
    \caption{Example of a decomposition of a surface of type $D_n$ according
			to a multilamination.}
    \label{fig:example_glueing}
	\end{figure}

	Any elementary lamination of $\Lambda_0$ not corresponding to a glued edge will
	be contained, up to a small neighbourhood of one endpoint, in exactly one piece
	in this decomposition. This implies that any given piece must contain at least
	a section of $\gamma$ and of two opposite sides of the quadrilateral enclosing
	$\gamma$ in order for any of the laminations it contains to give rise to a
	positive coordinate. In particular a quadrilateral of a triangulation can
	intersect non trivially at most three pieces in this decomposition. 
	
	We need therefore to consider all the possible ways a quadrilateral from Figure
	\ref{fig:shear-coordinates} can be fitted into a surface with at most three
	pieces. This is a straightforward but tedious check; a complete analysis of the
	various cases (87 nontrivial cases in total) is contained in Appendix
	\ref{app:type-Dn}.
\end{proof}

To connect $\mathcal{C}_+(B)$ with $\mathcal{D}(B)$ let us improve on the
parametrization of $c$-vectors of $\mathcal{A}_\bullet(B)$.
A triangulation $\Gamma$ of $S$ is said to be
\emph{bipartite} if every node of the corresponding quiver is either a sink or a
source. Note that, since in finite type any chordless cycle must be oriented
(\cite{Barot06} Theorem 1.2), bipartite triangulations correspond to bipartite
orientations of the Dynkin diagram of the given type.

Not every quadrilateral can appear in a bipartite triangulation; indeed it is
clear from the assignment (\ref{eqn:b-matrix_from_triangulation}) that the only
allowed one are those in Figure \ref{fig:bipartite_quadrilaterals}. Moreover,
given  any such quadrilateral, there exists a unique bipartite triangulation in 
which it appears.
\begin{figure}[htpb]
	\begin{center}
		\includegraphics[scale=.6]{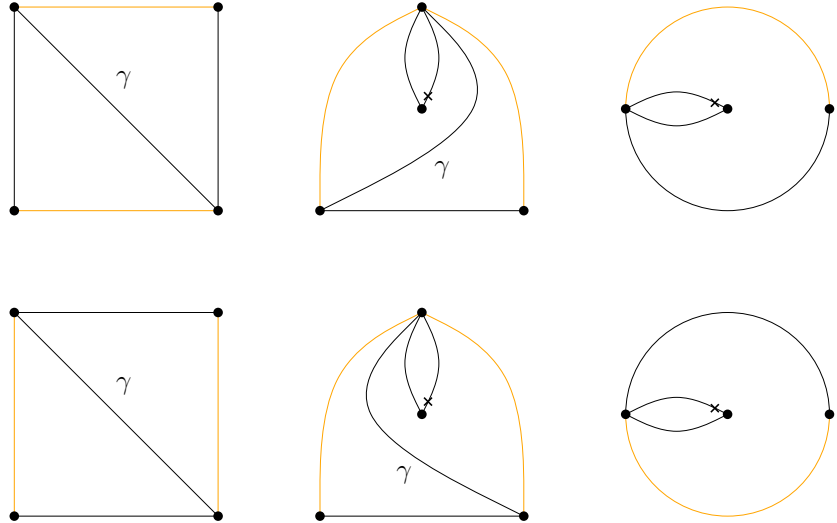}
	\end{center}
	\caption{The only quadrilaterals that can appear in a bipartite triangulation
	of a surface $S$. The edges on the boundary of $S$ are highlighted.
	When the quadrilateral is a digon any of the radii can be the diagonal
	$\gamma$.}
	\label{fig:bipartite_quadrilaterals}
\end{figure}

Let $\mathcal{C}^b_+(B)$ be the subset of $\mathcal{C}_+(B)$ consisting of
$c$-vectors $c_{\gamma,\Gamma}$ such that $\Gamma$ is bipartite.

\begin{prop}
	In types $A_n$ and $D_n$ 
	\[
		\mathcal{C}^b_+(B)=\mathcal{C}_+(B).
	\]
	\label{prop:bipartite}
\end{prop}
\begin{proof}
	Let $c_{\gamma,\Gamma}$ be any element of $\mathcal{C}_+(B)$; we need to show
	that there exists a bipartite triangulation $\Gamma'$ and an arc $\gamma'\in
	\Gamma'$ such that	$c_{\gamma,\Gamma}=c_{\gamma',\Gamma'}$. 

	Let $\Lambda_0$ be the multilamination associated to $B$.
	In view of the observation we just	made we only need to construct a
	quadrilateral like those in Figure \ref{fig:bipartite_quadrilaterals} having
	the same intersections with $\Lambda_0$ that $\Gamma$ does: this will
	automatically determine the bipartite triangulation we are after.

	We concentrate first on type $A_n$.
	The idea is simple: pick a leaf in the support of
	$c_{\gamma,\Gamma}$ and let $\lambda$ be the corresponding elementary
	lamination in $\Lambda_0$. Since $\lambda$ is the ``last'' lamination
	intersecting the quadrilateral enclosing $\gamma$ positively it must belong to
	a triangle in $\Lambda_0$ such that the other two lamination composing it do not
  give rise to positive shear coordinates. Let $p'$ be the only vertex of the
	triangle that is not incident to $\lambda$. We can replace the
	original quadrilateral with one having the two marked points closest to $p'$
	as vertices: all the shear coordinates will be
	unchanged. We can than conclude by applying the same procedure to the other leaf 
	(cf. Figure \ref{fig:example_bipartite_An}). 
\begin{figure}[htpb]
	\begin{center}
		\includegraphics[scale=.6]{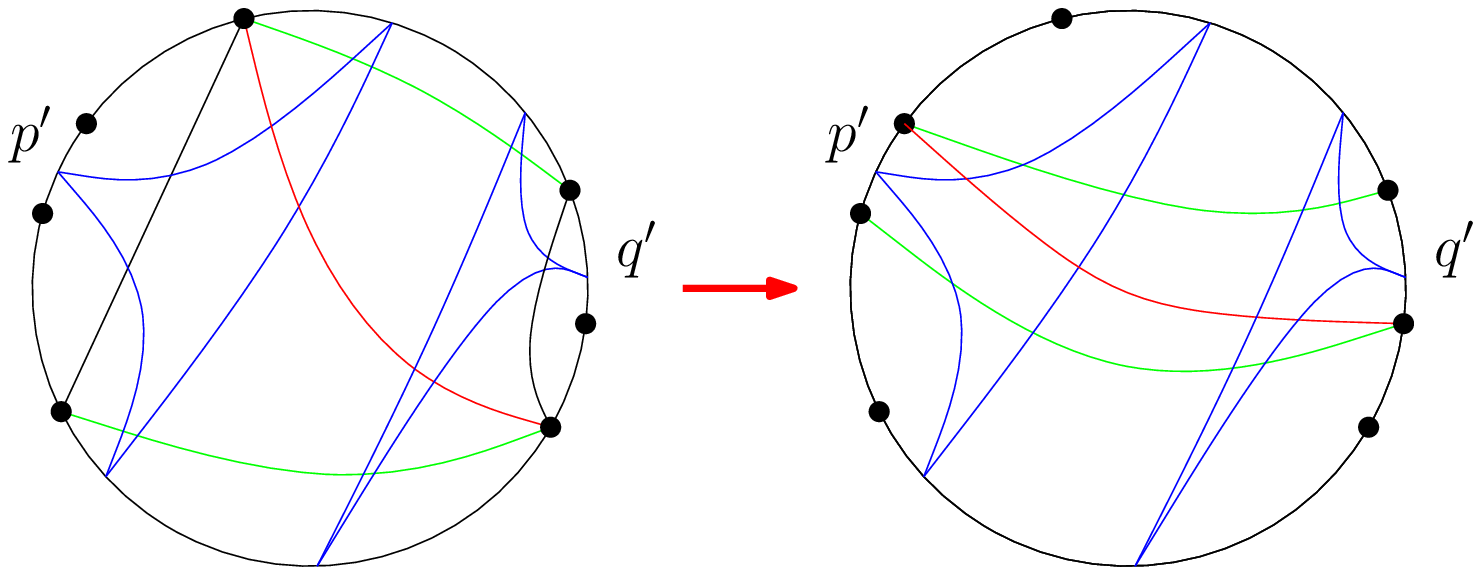}
	\end{center}
	\caption{The reduction of a quadrilateral to a bipartite quadrilateral in type
		$A_n$. The quadrilateral on the right give raise to the same shear
		coordinates produced by the quadrilateral on the left and determines
		uniquely a bipartite triangulation.}
	\label{fig:example_bipartite_An}
\end{figure}

	This is sufficient in type $A_n$ but not in general in type $D_n$: we need to
	deal with folded quadrilaterals as well. The replacement to be performed
	depends both on $\Lambda_0$ and $\gamma$ but it is straightforward from the
	pictures. The general procedure is shown in Figure
	\ref{fig:example_bipartite}. The reduction is in two steps: first we apply the
	same strategy of type $A_n$ to have the correct amount of edges of the
	quadrilateral on the boundary of the surface. Then, if needed, we replace the quadrilateral we
	obtain with one from  Figure \ref{fig:bipartite_quadrilaterals}.
	The precise case analysis is again in 
	Appendix \ref{app:type-Dn}; there we provide, for each possible quadrilateral
	and for each multilamination an explicit replacement.
\end{proof}
\begin{figure}[htpb]
	\begin{center}
		\includegraphics[scale=.6]{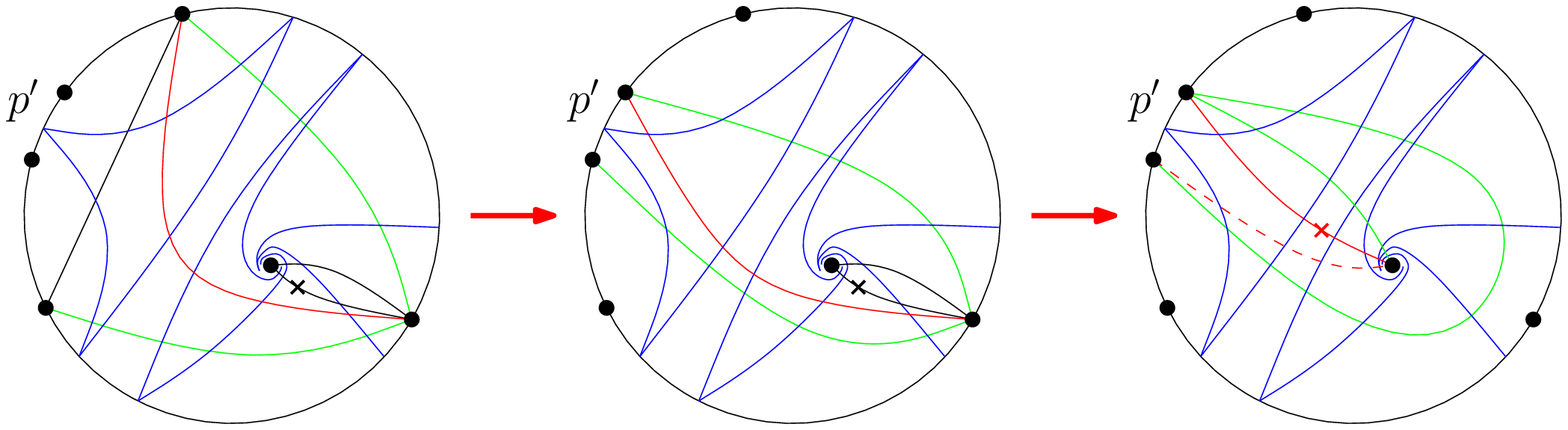}
	\end{center}
	\caption{An example of the reduction of a quadrilateral to a bipartite
		quadrilateral in type $D_n$. }
	\label{fig:example_bipartite}
\end{figure}

In analogy with the definition above let $\mathcal{D}^b(B)$ be the subset of all
the non-initial $d$-vectors corresponding to cluster variables appearing in
bipartite seeds of $\mathcal{A}_\bullet(B)$. Since any arc on $S$ appears in a
bipartite triangulation, in types $A_n$ and $D_n$ we have 
\begin{align}
	\mathcal{D}^b(B)=\mathcal{D}(B).
	\label{eqn:bipartite-D}
\end{align}
\begin{rk}
	The above equality, together with Proposition \ref{prop:bipartite}, prove
	Theorem \ref{thm:bipartite} for cluster algebras of types $A_n$ and $D_n$.
\end{rk}

\begin{prop}
	\label{prop:c=d_simply_laced}
	In types $A_n$ and $D_n$
	\[
		\mathcal{C}_+(B)=\mathcal{D}(B).
	\]
\end{prop}
\begin{proof}
	In view of the above reductions it suffices to show that 
  \[
	   \mathcal{C}^b_+(B)=\mathcal{D}^b(B).
	\]
	As before let $\Gamma_0=\left\{ \gamma_i \right\}_{i\in I}$ be the triangulation
	corresponding to $B$ and $\Lambda_0=\left\{ \lambda_i \right\}_{i\in I}$ the
	associated multilamination.
	In view of Theorem \ref{thm:description_of_d-vectors} and Definition
	\ref{defn:intersection_pairing} all the vectors in $\mathcal{D}^b(B)$ have
	non-negative components. 
	
	Let $\gamma$ be any arc not in $\Gamma_0$ and consider the $d$-vector
	$d_\gamma$; we need to distinguish three cases (cf. Figure
	\ref{fig:d_vectors-to-c_vectors})
	depending on the endpoints of $\gamma$ (call them $p$ and $q$).
	\begin{itemize}
		\item 
			If both $p$ and $q$ are on the boundary of $S$ and they are not adjacent
			then there are two other marked points $r$ and $s$ such that $p'$ is
			contained on the boundary segment $pr$ and $q'$ is contained in the
			boundary segment $qs$. Let $\gamma'$ be the diagonal $rs$ of the
			quadrilateral $prqs$ and complete the quadrilateral to a bipartite triangulation
			$\Gamma'$. We have $d_\gamma=c_{\gamma',\Gamma'}$. Note that if $S$ is of
			type $A_n$ this is the only possible case.
			
		\item
			It both $p$ and $q$ are on the boundary of $S$ and they are adjacent then
			we can assume (up to relabeling) that $q'$ lies on the boundary segment
			$qp$. Let $r$ be such that $p'$ lies on the boundary segment $pr$. Let
			$\gamma'$ be the diagonal $pr$ of the folded quadrilateral having
			vertices $q$, $p$, $r$, and the puncture and having two homotopic radii
			starting at $p$; Let $\Gamma'$ be the bipartite
			triangulation containing this quadrilateral. We have again $d_\gamma=
			c_{\gamma',\Gamma'}$.

		\item
			If one of the endpoints of $\gamma$ (say $q$ to fix ideas) is the puncture
			then let $r$ be the marked point such that $p'$ lies between $p$ and $r$.
			Let $\Gamma'$ be the bipartite triangulation containing the digon with vertices $p$
			and $r$, enclosing the puncture, and such that its radii both start from
			$r$. If $\gamma'$ is the radius with tagging opposite to the tagging of
			$\gamma$ then $d_\gamma=c_{\gamma',\Gamma'}$.

	\end{itemize}

	\begin{figure}[htbp]
		\begin{center}
			\includegraphics[scale=.6]{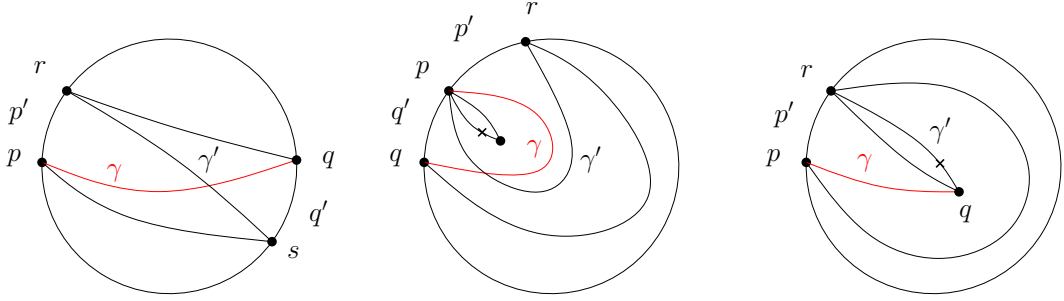}
		\end{center}
		\caption{The three possible cases of Proposition
			\ref{prop:c=d_simply_laced}.}
		\label{fig:d_vectors-to-c_vectors}
	\end{figure}

	Conversely let $c_{\gamma',\Gamma'}$ in $\mathcal{C}^b_+(B)$. The
	quadrilateral of $\Gamma'$ enclosing $\gamma'$ will be exactly one of those
	constructed above (they are all bipartite). Choosing $\gamma$ to be the
	corresponding arc we get $d_\gamma=c_{\gamma',\Gamma'}$.
\end{proof}
We thank Andrei Zelevinsky for providing the idea of using the ``bipartite
belt'' in the above proof.

The following Proposition concludes the proof of Theorem \ref{thm:main} for types $A_n$
and $D_n$.
\begin{prop}
	\label{prop:inverse-AD}
	In types $A_n$ and $D_n$ we have
	\[
		\mathcal{V}(B)\subset\mathcal{D}(B).
	\]
\end{prop}
\begin{proof}
	Let $\Gamma_0=\left\{ \gamma_i \right\}_{i\in I}$ be a triangulation realizing
	$B$ and let $v=\left( v_i \right)_{i\in I}$ be any element in $\mathcal{V}(B)$.

	In type $A_n$ it is clear how to construct an arc $\gamma$ crossing exactly
	one time all the arcs $\gamma_j$ such that $v_j\neq0$: suppose $i$ is a leaf
	in the weighted diagram; the arc $\gamma_i$ corresponding to it belongs to two
	triangles. One of them is such that the nodes corresponding to the other two
	arcs forming it do not belong to the support of the weighted diagram. The arc
	$\gamma$ we are looking for starts from the vertex of this triangle opposed to
	$\gamma_i$.  It crosses then in sequence all the arcs $\gamma_j$ such that
	$v_j\neq0$ and terminates in the vertex opposite to the arc corresponding to
	the other leaf.

	In type $D_n$ the procedure is slightly more involved and depends on the
	initial triangulation $\Gamma_0$ but follows the same basic idea. 
	Suppose at first that $\Gamma_0$ does not contain a digon with two homotopic radii. 
	The same procedure described for type $A_n$ works verbatim for diagrams 
	I, IV, and VIII in Figure \ref{fig:allowed_diagrams_Dn}: 
	the only thing to note is that instead of a leaf we might
	have to take one of the nodes in the left triangle (for IV we cannot use the
	two rightmost leaves). 
	For diagrams II and VI we need a small fix: $\gamma$ starts from the vertex
	opposite to the arc corresponding to the leftmost leaf and ends at the
	puncture; its tagging is the opposite of the tagging of the radii in $\Gamma_0$. 
	For diagrams III and VII we repeat the same argument using the
	leftmost leaf and the leftmost node with multiplicity $2$.
	Diagrams like V cannot be embedded in a $X(B)$ if $\Gamma_0$ does not have a digon
	with two homotopic radii in it.

	If $\Gamma_0$ contains a digon with two homotopic radii then the only diagrams that
	can arise are I, II, III, IV, and V. For V the procedure is the same as
	the one for type $A_n$, we just need to cross both the radii of the digon.
	For diagrams III and IV the procedure is identical to the above. For diagrams
	like II $\gamma$ starts from the vertex opposite to the arc corresponding to
	the leftmost leaf and ends in the vertex of the digon not adjacent to the
	radii. For diagrams like I we need to distinguish two cases: if one of the
	leaves corresponds to a radius then the corresponding endpoint of $\gamma$ is
	the puncture and its tagging is the opposite of the one of that radius.
	Otherwise we proceed as in type $A_n$.
\end{proof}

\section{Types \texorpdfstring{$B_n$}{Bn} and \texorpdfstring{$C_n$}{Cn}: the folding method}
\label{sect:folding-bc}
Building on the results of last section we will now prove Theorem \ref{thm:main}
for  types $B_n$ and $C_n$.  In order to do so we will realize any principal
coefficients cluster algebra of type $B_n$ (respectively $C_n$) as a
subquotient of an appropriate cluster algebra of type $D_{n+1}$ (respectively
$A_{2n-1}$) with principal coefficients.

\subsection{Folding of cluster algebras with trivial coefficients}
The construction, for the coefficient-free case,  was explained in
\cite{Dupont08}. Since we need to generalize it to work with principal
coefficients later on let us begin by recalling in some details its main
features.

Let $B=\left( b_{ij} \right)_{i,j\in I}$ be a skew-symmetrizable integer matrix
and $\sigma$ a permutation of $I$.  

\begin{defn}
	A permutation $\sigma$ is an \emph{automorphism} of $B$ if, for any $i$ and
	$j$ in $I$,
	\begin{equation}
		b_{\sigma(i)\sigma(j)}=b_{ij}.
		\label{eqn:automorphism}
	\end{equation}
	An automorphism of $B$ is said to be \emph{admissible} if, for any $i_1$ and
	$i_2$ in the same $\sigma$-orbit $\overline \imath$ and for any $j$ in $I$,
	\begin{eqnarray}
		b_{i_1,j}b_{i_2j}\ge 0
		\label{eqn:no_2_paths}\\
		b_{i_1,i_2}= 0.
		\label{eqn:no_1_paths}
	\end{eqnarray}
\end{defn}

An easy computation shows that, if $\sigma$ is an admissible automorphism of
$B$ and $k_1$ and $k_2$ are two points in the same $\sigma$-orbit
$\overline{k}$, the mutations $\mu_{k_1}$ and $\mu_{k_2}$ commute; that is 
\[
\mu_{k_1}\circ \mu_{k_2} (B)
=
\mu_{k_2}\circ \mu_{k_1} (B).
\]
Indeed $\mu_{k_2}\left( \mu_{k_1}(b_{ij}) \right)$ is either $-b_{ij}$, if at
least one among $i$ and $j$ is in $\left\{ k_1,k_2 \right\}$, or 
\[
b_{ij} 
+b_{ik_1}[b_{k_1j}]_+ +[-b_{ik_1}]_+  b_{k_1j}
+b_{ik_2}[b_{k_2j}]_+ +[-b_{ik_2}]_+  b_{k_2j}
\]
otherwise. Those expressions are clearly independent on the order in which
$\mu_{k_1}$ and $\mu_{k_2}$ are applied.
It makes therefore sense to define \emph{orbit-mutations} as the compositions
\[
\mu_{\overline{k}}^\sigma:=\prod_{t\in\overline{k}}\mu_t.
\]
Repeating the same reasoning we get
\begin{equation}
	\mu_{\overline{k}}^\sigma(b_{ij})=
	\left\{
  \begin{array}[]{ll}
    -b_{ij} & 
		\mbox{if } i \mbox{ or } j \in \overline{k}\\
		b_{ij} +\sum_{t\in \overline{k}}\left (b_{it}[b_{tj}]_+ +[-b_{it}]_+b_{tj}\right ) &
    \mbox{otherwise}.\\
  \end{array}
  \right.
	\label{eqn:orbit_mutation}
\end{equation}

Note that, given a $\sigma$-orbit $\overline{k}$, the permutation $\sigma$ is
always an automorphism of
$\mu_{\overline{k}}^\sigma(B)$ but it need not be admissible; in particular
condition (\ref{eqn:no_2_paths}) may be violated.

\begin{defn}
	An admissible automorphism $\sigma$ of $B$ is said to be \emph{stable} if, for any finite
	sequence of $\sigma$-orbits $\overline{k_1},\dots\overline{k_\ell}$, it is an
	admissible automorphism of
  \[	
	\mu_{\overline{k_\ell}}^\sigma\circ\dots\circ\mu_{\overline{k_1}}^\sigma(B).
	\]
\end{defn}

\begin{prop}[{\cite[Proposition 2.22]{Dupont08}}]
	If the Cartan counterpart of $B$ is a simply-laced finite type then any
	admissible automorphism of $B$ is stable.
	\label{prop:dynkin_stable}
\end{prop}

\begin{rk}
	We will need the following incarnations of Proposition \ref{prop:dynkin_stable}:
	\begin{enumerate}
		\item 
			$B$ has Cartan counterpart of type $A_{2n-1}$ and, using the standard
			labeling of the nodes of the associated Dynkin diagram,
			\[
			\sigma:=\prod_{i=1}^n(i,2n-i)
			\]
			is an admissible automorphism of $B$.
			
		\item
			$B$ has Cartan counterpart of type $D_{n+1}$ and, again in the standard
			labeling,
			\[
			\sigma=(n,n+1)
			\]
			is an admissible automorphism of $B$.
	\end{enumerate}
	\label{rk:our_cases}
\end{rk}

Given a skew-symmetrizable integer matrix $B$ and a (stable) admissible
automorphism $\sigma$ we can define a \emph{folded} matrix
$\pi(B):=\overline{B}=\left( b_{\overline{\imath}\overline{\jmath}} \right)$, as
$\overline{\imath}$ and $\overline{\jmath}$ vary over all the $\sigma$-orbits, by
setting
\begin{equation}
	b_{\overline{\imath}\overline{\jmath}}:=
	\sum_{s\in\overline{\imath}}b_{sj}.
	\label{eqn:folded_b_matrix}
\end{equation}
In view of condition (\ref{eqn:automorphism}) the value of
$b_{\overline{\imath}\overline{\jmath}}$ does not depend on the choice of a
representative of $\overline{\jmath}$. The folded matrix $\pi(B)$ is itself
skew-symmetrizable (see \cite[Lemma 2.5]{Dupont08}).

The key point here is this: if $\sigma$ is a stable admissible automorphism of $B$
then for any $\sigma$-orbit $\overline{k}$
\[
\pi\left( \mu_{\overline{k}}^\sigma(B) \right)=\mu_{\overline{k}}\left( \pi(B) \right)
\]
thanks to condition (\ref{eqn:no_2_paths}) (see \cite[Theorem 2.24]{Dupont08}).

We will use the following obvious converse stating the existence of
``unfolding'' for the matrices we are interested into.
\begin{prop}
	Let $\overline{B}'$ be any matrix in the same mutation class of a matrix
	$\overline{B}$ obtained by folding from a skew-symmetrizable matrix $B$ with a
	stable admissible automorphism $\sigma$. There exist a matrix $B'$ and a
	sequence of $\sigma$-orbits $\overline{k_1},\dots,\overline{k_\ell}$ such that
	\begin{enumerate}
		\item 
			$\mu_{\overline{k_\ell}}^\sigma\circ\dots\circ\mu_{\overline{k_1}}^\sigma(B)=B'$
		\item
			$\overline{B}'=\overline{B'}$.
	\end{enumerate}
	\label{prop:unfolding}
\end{prop}

The folding map can be extended to a morphism of algebras as
follows. 
Fix an initial $B$-matrix $B$ and a stable admissible automorphism $\sigma$. Let
\[
\left( B,\left\{ x_i \right\}_{i\in I} \right)
\]
be the initial cluster of the coefficient-free cluster algebra $\mathcal{A}_0(B)$.
Write $\mathcal{A}_0^\sigma(B)$ for the subalgebra of
$\mathcal{A}_0(B)$ generated by all the clusters reachable from the
initial one by a sequence of orbit mutations.

Let $\mathcal{A}_0(\overline{B})$ be the coefficient-free cluster algebra with
initial $B$-matrix $\pi(B)=\overline{B}$ and initial cluster variables 
$\left\{ x_{\overline{\imath}}\right\}_{\overline{\imath}\in I/\sigma}$.
The assignment 
\[
\pi(x_i):=x_{\overline{\imath}}
\]
extends to a surjective map
\[
\pi:\mathcal{A}_0^\sigma(B)\longrightarrow\mathcal{A}_0(\overline{B}).
\]
The algebra $\mathcal{A}_0(\overline{B})$ is the quotient of
$\mathcal{A}_0^\sigma(B)$  by the ideal generated by the relations
\[
x_i=x_{\sigma(i)}.
\]
Moreover, and this is the key point in the construction, the map $\pi$ preserves
the cluster structure: seeds of $\mathcal{A}_0^\sigma(B)$ are mapped to seeds of
$\mathcal{A}_0(\overline{B})$.

Combining the above observation with Remark \ref{rk:our_cases} we get the
following statement.
\begin{prop}
	\label{prop:b_and_c_are_folded}
	Any matrix of cluster type $B_n$ (respectively $C_n$) is the image $\pi(B)$ of
	a matrix $B$ of cluster type $D_{n+1}$ (respectively $A_{2n-1}$) with
	automorphism $\sigma$ from Remark \ref{rk:our_cases}. The
	coefficient-free cluster algebra $\mathcal{A}_0(\overline{B})$ is the quotient
	of a subalgebra of $\mathcal{A}_0(B)$ by an ideal preserving the cluster
	structure. In particular any exchange matrix of $\mathcal{A}_0(\overline{B})$
	is the folding of some exchange matrix of $\mathcal{A}_0(B)$.
\end{prop}

\subsection{Proof of Propositions \ref{prop:dynkin-Bn} and \ref{prop:dynkin-Cn}.}
\label{sect:folding-bc-dynkin}
The results just summarized are enough to describe the sets $\mathcal{X}(B_n)$
and $\mathcal{X}(C_n)$.

\begin{proof}[Proof of Proposition \ref{prop:dynkin-Bn}]
	In view of Proposition \ref{prop:b_and_c_are_folded} any element of
	$\mathcal{X}(B_n)$ can be obtained by folding an element of
	$\mathcal{X}(D_{n+1})$. On the other hand not every diagram from Figure
	\ref{fig:dynkin-Dn} can be folded: we know that any chordless cycle in such a
	diagram corresponds to a cyclically oriented chordless cycle in the quiver
	$Q(B)$ associated to it (see \cite{Fomin03a,Barot06}). By definition of
	admissible automorphism all the vertices in the only non-trivial orbit of
	$\sigma$ must be not adjacent and must be connected to all the other adjacent
	vertices in the same way.  This forces us to conclude that diagrams (c) and
	(d) cannot be folded. 

	The diagrams of Figure \ref{fig:dynkin-Bn} are thus the folding of diagrams
	(a) and (b) from Figure \ref{fig:dynkin-Dn}.
\end{proof}

\begin{proof}[Proof of Proposition \ref{prop:dynkin-Cn}]
	In view of Proposition \ref{prop:b_and_c_are_folded} diagrams in
	$\mathcal{X}(C_n)$ are obtained by folding elements of $\mathcal{X}(A_{2n-1})$.
	The only requirement a diagram must satisfy to be folded is to be symmetric
	with respect to the only fixed point of $\sigma$ from Remark
	\ref{rk:our_cases}.
\end{proof}

\subsection{Folding of \texorpdfstring{$c$}{c}-vectors}
In order to consider $c$-vectors we need to extend the above construction to
cluster algebras with principal coefficients. We take inspiration from the
following example.

\begin{ex}
	Let $\mathcal{A}_\bullet(B)$ be the cluster algebra of type $D_4$ with principal
	coefficients at the initial cluster given by
	\[	
	B=
	\left( 
	\begin{array}[]{cccc}
		0 & -1 & 0 & 0  \\
		1 & 0 & -1 & -1 \\
		0 & 1 & 0 & 0 \\
		0 & 1 & 0 & 0 \\
	\end{array}
	\right).
	\]
	$B$ is invariant under permutation $\sigma=(34)$ and has $b_{34}=0$.
	Moreover the mutations in directions $3$ and $4$ commute; that is 
	\[
	\mu_3\circ\mu_4(B)=\mu_4\circ\mu_3(B).
	\]
	Let $\mathcal{A}_\bullet^{\sigma }(B)$ be the subalgebra of all the clusters
	reachable from the initial one by any sequence of the mutations $\mu_1$,
	$\mu_2$, and $\mu_3\circ\mu_4$. All the $B$-matrices in it have
	$b_{34}=0$.

	The permutation $\sigma $ acts on the set of clusters of $\mathcal{A}_\bullet(B)$
	by relabeling:
	\[
	\sigma(x_i):=x_{\sigma(i)}
	\qquad
	\mbox{and}
	\qquad
	\sigma(y_i):=y_{\sigma(i)}
	\]
	Let $\mathcal{I}$ be the ring ideal of
	$\mathcal{A}_\bullet^{\sigma }(B)$ generated by the relations
	\[
	x_3=x_4
	\qquad
	\mbox{and}
	\qquad
	y_3=y_4	
	\]
	
	The quotient $\mathcal{A}_\bullet^{\sigma }(B)/\mathcal{I}$ is a cluster algebra
	of type $B_3$ with principal coefficients at the initial cluster given by
	$\pi(B)$. Under the projection map clusters of $\mathcal{A}_\bullet^{\sigma }(B)$
	are mapped to clusters of $\mathcal{A}_\bullet^{\sigma }(B)/\mathcal{I}$. 
	Moreover exchange relations in the quotient come from
	exchange relations of $\mathcal{A}_\bullet(B)$.
\end{ex}

For any skew-symmetrizable integer matrix $B$ endowed with a stable admissible
automorphism $\sigma$ let $\mathcal{A}_\bullet(B)$ and
$\mathcal{A}_\bullet(\overline{B})$ be the cluster algebras with principal
coefficients respectively at $B$ and $\overline{B}=\pi(B)$.  Let
$\mathcal{A}_\bullet^\sigma(B)$ be the subalgebra of $\mathcal{A}_\bullet(B)$
generated by all clusters reachable from the initial one using orbit-mutations.

In view of the above example it is natural to define folding for a $c$-vector
$c=\left( c_i \right)_{i\in I}$ of $\mathcal{A}_\bullet^\sigma(B)$ componentwise as
follows:
\begin{equation}
	\label{eqn:folding-c}
	c_{\overline{\imath}}:=\sum_{s\in \overline{\imath}}c_s.
\end{equation}
However the correctness of this definition is not so obvious because the
tropicalization map (\ref{eqn:tropicalization-y}) and folding are not compatible
in general. Let us clarify the condition required to guarantee that 
(\ref{eqn:folding-c}) is well-posed.

Note that if $C$ and $C'$ are two coefficient matrices of
$\mathcal{A}_\bullet^\sigma(B)$ connected by a single orbit-mutation
$\mu_{\overline{k}}$ then it follows directly from having assumed $\sigma$ to
be a stable admissible automorphism of $B$ that:
\begin{equation}
	c'_{ij}=
	\left\{
 	\begin{array}[]{ll}
 	 -c_{ij} & 
		\mbox{if } j \in \overline{k}\\
		c_{ij} +\sum_{t\in \overline{k}}\left (c_{it}[b_{tj}]_+ +[-c_{it}]_+b_{tj}\right ) &
 	 \mbox{otherwise}.\\
 	\end{array}
 	\right.
	\label{eqn:orbit_mutation-c}
\end{equation}

From (\ref{eqn:orbit_mutation-c}) we get an important observation: all the
$C$-matrices in $\mathcal{A}_\bullet^\sigma(B)$ are such that 
\begin{equation}
	c_{\sigma(i)\sigma(j)}=c_{ij}.
	\label{eqn:c-vector-invariant}
\end{equation}
Indeed the property holds for the initial $C$-matrix and we can use the
admissibility of $\sigma$ to propagate it.

We introduce the folded $C$-matrix $\overline{C}$ for a $C$-matrix $C=\left(
c_{ij} \right)_{ij \in I}$ of $\mathcal{A}_\bullet^\sigma(B)$ as
\begin{align}
	c_{\overline{\imath}\overline{\jmath}}:=\sum_{s\in
	\overline{\imath}}c_{sj}.
	\label{eqn:folding-c-matrix}
\end{align}
Note that (\ref{eqn:folding-c-matrix}) is independent of the choice of a
representative of $\overline{\jmath}$ due to the symmetry
(\ref{eqn:c-vector-invariant}).

\begin{prop}
	\label{prop:folding-c}
	Let $B$ be any skew-symmetrizable integer matrix and let $\sigma$ be a stable
	admissible automorphism of $B$. 
	The matrix $\overline{C}$ satisfies the recursion relation
\begin{equation}
	c'_{\overline{\imath}\overline{\jmath}}=
	\left\{
 	\begin{array}[]{ll}
		-c_{\overline{\imath}\overline{\jmath}} & 
		\mbox{if } \overline{\jmath} = \overline{k}\\
		c_{\overline{\imath}\overline{\jmath}}
		+c_{\overline{\imath}\overline{k}}[b_{\overline{k}\overline{\jmath}}]_+ 
		+[-c_{\overline{\imath}\overline{k}}]_+b_{\overline{k}\overline{\jmath}} &
 	 \mbox{otherwise}\\
 	\end{array}
 	\right.
	\label{eqn:orbit_mutation-c-folded}
\end{equation}
	if and only if the following condition holds: for any $i$ and $j$ the sign of
	$c_{sj}$ is independent of the choice of representative
	$s\in\overline{\imath}$.
\end{prop}
\begin{proof}
	It suffices to establish the proposition for a single mutation; if
	$\overline{\jmath}=\overline{k}$ our claim is trivial so we can assume
	$j\not\in \overline{k}$. On the one hand we can rewrite
	(\ref{eqn:orbit_mutation-c-folded}) as
	\[
		c'_{\overline{\imath}\overline{\jmath}}
		=
		c_{\overline{\imath}\overline{\jmath}} +
		c_{\overline{\imath}\overline{k}}[b_{\overline{k}\overline{\jmath}}]_+ +
		[-c_{\overline{\imath}\overline{k}}]_+b_{\overline{k}\overline{\jmath}}
		=
		\sum_{s\in \overline{\imath}} c_{sj}+
		\sum_{s\in \overline{\imath}} c_{sk}\left[ \sum_{t\in\overline{k}}b_{tj}\right]_+ +
		\left [ -\sum_{s\in \overline{\imath}} c_{sk} \right ]_+\sum_{t\in\overline{k}}b_{tj}
		.
	\]

	On the other hand 
	\[
		c'_{\overline{\imath}\overline{\jmath}}
		=
		\sum_{s\in \overline{\imath}} c'_{sj}
		=
		\sum_{s\in \overline{\imath}} \left( c_{sj} +\sum_{t\in \overline{k}}\left
		(c_{st}[b_{tj}]_+ +[-c_{st}]_+b_{tj}\right ) \right).
	\]
	Therefore the recursion \ref{eqn:orbit_mutation-c-folded} is satisfied if and
	only if
	\[
		\sum_{s\in \overline{\imath}} c_{sk}\left[\sum_{t\in\overline{k}}b_{tj}\right]_+
		=
		\sum_{s\in \overline{\imath}} \sum_{t\in \overline{k}} c_{st}[b_{tj}]_+
	\]
	and
	\[
		\left [ -\sum_{s\in \overline{\imath}} c_{sk} \right]_+\sum_{t\in\overline{k}}b_{tj}
		=
		\sum_{s\in \overline{\imath}} \sum_{t\in \overline{k}} [-c_{st}]_+b_{tj}.
	\]
	The first condition is guaranteed by the admissibility of $\sigma$; indeed we
	get 	
	\[
		\sum_{s\in \overline{\imath}} \sum_{t\in\overline{k}}c_{sk}\left[b_{tj}\right]_+
		=
		\sum_{s\in \overline{\imath}} \sum_{t\in \overline{k}} c_{st}[b_{tj}]_+
	\]
	which is true by a simple change of summation index using
	(\ref{eqn:c-vector-invariant}).

	Similarly the second condition is equivalent to 
	\[
		\sum_{s\in \overline{\imath}}\sum_{t\in\overline{k}} \left [ -c_{sk} \right]_+b_{tj}
		=
		\sum_{s\in \overline{\imath}} \sum_{t\in \overline{k}} [-c_{st}]_+b_{tj}.
	\]
	if and only if the sign of $c_{sk}$ is independent on the choice of
	representative $s\in\overline{\imath}$.
\end{proof}

In our situation the condition of Proposition \ref{prop:folding-c} is satisfied
by the sign-coherence property of $c$-vectors established in Lemma
\ref{lemma:c-vectors_are_sign_coerent} or, more generally for skew-symmetric
$B$-matrices, explained in Section \ref{sect:back}. Thus definition
(\ref{eqn:folding-c}) is well-posed in our case.

It is worth noticing at this point that the folding map (\ref{eqn:folding-c})
sends the identity matrix to the identity matrix: as one might expect the image
of the initial cluster of $\mathcal{A}_\bullet(B)$ is the initial cluster of
$\mathcal{A}_\bullet\left( \overline{B} \right)$.

\begin{cor}
	\label{cor:folding}
	Let $\overline{B}$ be any skew-symmetrizable  matrix of cluster type
	$B_n$ (respectively $C_n$). There exists a matrix $B$ of
	cluster type $D_{n+1}$ (respectively $A_{2n-1}$) such that
	the cluster algebra with principal coefficients $\mathcal{A}_\bullet\left(
	\overline{B} \right)$ is a subquotient of the cluster algebra
	with principal coefficients $\mathcal{A}_\bullet(B)$. In particular any
	$c$-vector of $\mathcal{A}_\bullet\left( \overline{B} \right)$ is the 
	folding of some	$c$-vector of $\mathcal{A}_\bullet(B)$.
\end{cor}

\subsection{Folding of \texorpdfstring{$d$}{d}-vectors}
Our next goal is to produce a folding rule for $d$-vectors. From the above
example it is natural to fold the vector $d=\left(
d_i \right)_{i\in I}$ componentwise in this way:
\begin{equation}
	\label{eqn:fold-d}
	d_{\overline{\imath}}:=\sum_{s\in\overline{\imath}}d_s.
\end{equation}
Once again the correctness of the above definition is not obvious because, in general, folding is
not compatible with the tropicalization map (\ref{eqn:tropicalization-x}).

Recall the definition of $D$-matrix given in Section \ref{sect:back.vectors}.

\begin{lem}
	\label{lemma:d-matrix-invariance}
	If $\sigma$ is a stable admissible automorphism of $B$ then the entries in any
	$D$-matrix of $\mathcal{A}^\sigma(B)$ satisfy
	\begin{align}
		d_{\sigma(i)\sigma(j)}=d_{ij}.
		\label{eqn:d-matrix-invariance}
	\end{align}
\end{lem}
\begin{proof}
	The property holds for the $D$-matrix of the initial cluster. Suppose that $D$
	and $D'$ correspond to clusters obtained from one another by a single orbit
	mutation $\mu_{\overline{k}}$ and that the property holds for $D$. 
	The only non trivial case we need to consider is when $j$ is in
	$\overline{k}$. By (\ref{eqn:d-recursion}) we have 
	\[
	d'_{ij}
	=
	-d_{ij}+\max\left( \sum_{t\in I}d_{it}[b_{tk}]_+,\sum_{t\in I}d_{it}[-b_{tk}]_+ \right).
	\]
	Using both induction hypotheses and the fact that $\sigma$ is stable admissible
	we get
	\[
		d'_{ij}
		=
		-d_{\sigma(i)\sigma(j)}
		+\max\left(
		\sum_{t\in I}-d_{\sigma(i)\sigma(t)}\left[b_{\sigma(t)\sigma(k)} \right]_+
		, 
		\sum_{t\in I}-d_{\sigma(i)\sigma(t)}\left[-b_{\sigma(t)\sigma(k)} \right]_+
		\right)
	\]
	and we can conclude changing the summation index.
\end{proof}

We define the folding $\overline{D}$ of the $D$-matrix $D=\left( d_{ij}
\right)_{ij\in I}$ as we did for $C$-matrices:
\[
	d_{\overline{\imath}\overline{\jmath}}:=
	\sum_{s\in\overline{\imath}}d_{sj}.
\]
Thank to the above lemma this definition is independent of the representative
$j$.
\begin{prop}
	\label{prop:condition-to-fold-d}
	The matrix $\overline{D}$ satisfies the recursion 
	\begin{align}
		d'_{\overline{\imath}\overline{\jmath}} &=
 		\begin{cases}
 		\displaystyle
		-d_{\overline{\imath}\overline{k}}
  	+
		\max  \left( 
		\sum_{\overline{\ell}\in I/\sigma}d_{\overline{\imath}\overline{\ell}}
		{[b_{\overline{\ell}\overline{k}}]_+},
 		\sum_{\overline{\ell}\in I/\sigma}d_{\overline{\imath}\overline{\ell}}  
		{[-b_{\overline{\ell}\overline{k}}]_+},
		\right)
 		& \overline{\jmath}=\overline{k}\\
 		d_{\overline{\imath}\overline{\jmath}}
 		& \overline{\jmath}\neq \overline{k}.
 		\end{cases}
	 \end{align}
	if and only if for any $\sigma$-orbit
	$\overline{\imath}$
	the sign of 
	\[
		\sum_{t\in I}d_{st}b_{tk}
	\]
	is independent of the representative $s\in\overline{\imath}$.
\end{prop}
\begin{proof}
	We proceed again by induction. It suffices to show that the property holds
	for a single mutation. Fix a $\sigma$-orbit $\overline{\imath}$.
	The only non-trivial case is when
	$\overline{\jmath}=\overline{k}$. On the one hand we have
	\[
		d'_{\overline{\imath}\overline{\jmath}}
		=
		\sum_{s\in\overline{\imath}}
		\left( 
		-d_{sj}+ 
		\max\left( 
		\sum_{t\in I}d_{st}[b_{tk}]_+,
		\sum_{t\in I}d_{st}[-b_{tk}]_+
		\right)
		\right).
	\]
	On the other hand, for the recursion to be satisfied, we must have 
	\[
		d'_{\overline{\imath}\overline{\jmath}}
		=
		-\sum_{s\in\overline{\imath}}d_{sj} +
		\max\left(
		\sum_{t\in I}\sum_{s\in\overline{\imath}}d_{st}[b_{tk}]_+, 
		\sum_{t\in I}\sum_{s\in\overline{\imath}}d_{st}[-b_{tk}]_+ 
		\right).
	\]
	We need therefore to have
	\begin{align}
		\sum_{s\in\overline{\imath}}
		\left( 
		\max\left( 
		\sum_{t\in I}d_{st}[b_{tk}]_+,
		\sum_{t\in I}d_{st}[-b_{tk}]_+
		\right)
		\right)
		&=\\
		\max\left(
		\sum_{s\in\overline{\imath}} \sum_{t\in I}d_{st}[b_{tk}]_+,
		\sum_{s\in\overline{\imath}} \sum_{t\in I}d_{st}[-b_{tk}]_+
		\right)
	\end{align}
	which holds if and only if
	the sign of 
	\[
		\sum_{t\in I}d_{st}[b_{tk}]_+-\sum_{t\in I}d_{st}[-b_{tk}]_+=\sum_{t\in
		I}d_{st}b_{tk}
	\]
	is independent of the choice of the representative $s\in\overline{\imath}$.
\end{proof}

\begin{rk}
	 Anna Felikson and Pavel Tumarkin found a case of cluster affine type $D$
	 where the condition of previous proposition does not hold \cite{Felikson12};
	 we thank them for showing us their example.
\end{rk}

For our purposes it is enough to show that the condition of Proposition
\ref{prop:condition-to-fold-d} holds in the cases of Remark \ref{rk:our_cases}.
Using Lemma \ref{lemma:d-matrix-invariance} and the fact that $\sigma$ is stable
admissible, it is equivalent to ask the sign of
\[
  \sum_{t\in I}d_{it}b_{tr}
\]
to be independent of the representative $r\in\overline{k}$. We get therefore
that the condition is satisfied whenever $k$ is fixed by $\sigma$.

We prefer to work with this third equivalent formulation: the sign of 
\[
  \sum_{t\in I}d_{i\sigma^m(t)}b_{tk}
\]
is independent of $m\in\mathbb{Z}$.
\begin{lem}
	\label{lem:folding-d-tupe_D}
	The condition of Proposition \ref{prop:condition-to-fold-d} holds for $B$ of
	cluster type $D_{n+1}$ endowed with the automorphism $\sigma$ of Remark
	\ref{rk:our_cases}.
\end{lem}
\begin{proof}
	There is only one non-trivial $\sigma$-orbit; in view of previous observations we
	can assume it is the orbit of $k$. This forces $t\not\in\overline{k}$ to be
	fixed by $\sigma$. Moreover, since $\sigma$ is a stable admissible automorphism
	$b_{tk}=0$ if $t\in\overline{k}$. Therefore 
	\[
		\sum_{t\in I}d_{i\sigma^m(t)}b_{tk}
		=
		\sum_{t\in I\setminus\overline{k}}d_{it}b_{tk}
	\]
	which is manifestly independent of $m$.
\end{proof}

\begin{lem}
	\label{lem:folding-d-tupe_A}
	The condition of Proposition \ref{prop:condition-to-fold-d} holds for $B$ of
	cluster type $A_{2n-1}$ endowed with the automorphism $\sigma$ of Remark
	\ref{rk:our_cases}.
\end{lem}
\begin{proof}
	Note at first that rows of a $D$-matrix associated to a $B$-matrix $B'$ in
	$\mathcal{A}^\sigma_0(B)$ are again $d$-vectors: they are the $d$-vectors of
	$\mathcal{A}^\sigma_0(B')$ in the $D$-matrix associated to $B$. This follows
	directly from the surface realization (see Theorem
	\ref{thm:description_of_d-vectors}). In particular, in this case, they are
	sign-coherent and their support is either a string (if they are positive) or a
	single vertex (if they are negative).

	As before we can assume that $k$ is not fixed by $\sigma$. If the support of
	the row $i$ does not contain neighbours of both $k$ and $\sigma(k)$ then the
	statement is clear. We can therefore assume that there is at least one
	neighbour of each of them in the support of the $i$-th row of $D$.

	Let $t_1$ and $t_2$ be the two neighbours of $k$ and $\sigma(k)$ respectively
	lying on the shortest path from $k$ to $\sigma(k)$. By the symmetry required
	for folding $t_2=\sigma(t_1)$. Moreover if a row of $D$ contains at least one
	neighbour of both $k$ and $\sigma(k)$ then it contains both $t_1$ and $t_2$.
	We claim that, in this situation,  
	\[
		\sum_{t\in I}d_{i\sigma^m(t)}b_{tk}
	\]
	is either $0$ or has the same sign of $b_{t_1k}$. Indeed each row of
	$D$ has at most $2$ neighbours of $k$ in its support and the entries of $B$
	are either $0$ or $\pm1$.

	We can therefore conclude our proof: since $\sigma$ is a stable admissible
	automorphism of $B$ we have:
	\[
		b_{t_1k}=b_{\sigma(t_1)\sigma(k)}.
	\]
\end{proof}

\subsection{Proof of Theorem \ref{thm:main} for types
	\texorpdfstring{$B_n$}{Bn} and \texorpdfstring{$C_n$}{Cn}}
To fix the notation observe that any $B$-matrix of cluster type $B_n$ or $C_n$ 
uniquely determines a $\sigma$-invariant matrix of cluster type respectively
$D_{n+1}$ or $A_{2n-1}$ of which it is the folding. We will therefore denote by
$\overline{B}$ a matrix of cluster type $B_n$ or $C_n$ and by $B$ its unfolding.

Let $\pi\left( \mathcal{V}(B) \right)$ be the image of the set $\mathcal{V}(B)$
under the folding map
\[
\begin{array}{cccc}
	\pi:&\mathcal{V}(B)&\longrightarrow&\mathbb{Z}^{\overline{I}}\\
	&\left( v_i \right)_{i\in I}&\longmapsto&\left( \sum_{s\in \overline{\imath}}v_s
	\right)_{\overline{\imath}\in\overline{I}}
\end{array}
\]
and recall the definition of the sets $\mathcal{W}(B_n)$ and $\mathcal{W}(C_n)$
from Section \ref{sect:sets}.

\begin{prop}
	\label{prop:folding-v}
	For any matrix $\overline{B}$ of cluster type $B_n$ or $C_n$ we have
	\[
		\mathcal{V}(\overline{B})=\pi\left( \mathcal{V}(B) \right).
	\]
\end{prop}
\begin{proof}
	The claim is clear once we observe that the diagrams in $\mathcal{W}(B_n)$ and
	$\mathcal{W}(C_n)$ are obtained precisely by folding diagrams from
	$\mathcal{W}(D_{n+1})$ and $\mathcal{W}(A_{2n-1})$ embedded in $X(B)$.
\end{proof}

We have now the tools we need to deduce Theorem 
\ref{thm:main} for types $B_n$ and $C_n$ from the same result for types $A_n$
and $D_n$.

\begin{prop}
	For any matrix $\overline{B}$ of cluster type $B_n$ or $C_n$ we have
	\[
		\mathcal{C}_+(\overline{B})\subset\mathcal{V}(\overline{B})
	\]
\end{prop}
\begin{proof}
	Combining Corollary \ref{cor:folding}, Proposition \ref{prop:An_Dn-support}
	and Proposition \ref{prop:folding-v} we have 
	\[
		\mathcal{C}_+(\overline{B})\subset
		\pi\left(\mathcal{C}_+(B) \right)=
		\pi\left( \mathcal{V}(B) \right)=
		\mathcal{V}(\overline{B}).
	\]
\end{proof}

\begin{prop}
	For any matrix $\overline{B}$ of cluster type $B_n$ or $C_n$ we have
	\[
		\mathcal{D}(\overline{B})\subset\mathcal{V}(\overline{B})
	\]
\end{prop}
\begin{proof}
	Combining Lemmata \ref{lem:folding-d-tupe_D} and \ref{lem:folding-d-tupe_A}
	with Proposition \ref{prop:c=d_simply_laced}, Proposition \ref{prop:An_Dn-support}
	and Proposition \ref{prop:folding-v} we have 
	\[
		\mathcal{D}(\overline{B})\subset
		\pi\left(\mathcal{D}(B) \right)=
		\pi\left( \mathcal{V}(B) \right)=
		\mathcal{V}(\overline{B}).
	\]
\end{proof}

To conclude we need one last lemma.
\begin{lem}
	\label{lem:folding-bipartite}
	For any matrix $\overline{B}$ of cluster type $B_n$ or $C_n$ we have
	\[
		\mathcal{C}^b_+(\overline{B})=
		\pi\left(\mathcal{C}^b_+(B) \right)
	\]
	and
	\[
		\mathcal{D}^b(\overline{B})=
		\pi\left(\mathcal{D}^b(B) \right).
	\]
\end{lem}
\begin{proof}
	The claim follows directly from the following observation: a matrix
	$\overline{B}$ is bipartite if and only if its unfolding $B$ is bipartite. We
	get equalities (as opposed to inclusions) because any two bipartite
	matrices of cluster type $D_{n+1}$ or $A_{2n-1}$ are connected by orbit
	mutations.
\end{proof}

\begin{prop}
	For any matrix $\overline{B}$ of cluster type $B_n$ or $C_n$ we have
	\[
		\mathcal{V}(\overline{B})\subset\mathcal{C}_+(\overline{B})
	\]
	and
	\[
		\mathcal{V}(\overline{B})\subset\mathcal{D}(\overline{B}).
	\]
\end{prop}
\begin{proof}
	We show only the second condition; the first one is obtained in the same way.
	Using Proposition \ref{prop:folding-v}, Proposition \ref{prop:inverse-AD},
	equation (\ref{eqn:bipartite-D}), and
	Lemma \ref{lem:folding-bipartite} we get
	\[
		\mathcal{V}(\overline{B})=
		\pi\left( \mathcal{V}(B) \right)\subset
		\pi\left( \mathcal{D}(B) \right)=
		\pi\left( \mathcal{D}^b(B) \right)=
		\mathcal{D}^b(\overline{B})\subset
		\mathcal{D}(\overline{B})
		.
	\]
\end{proof}

For completeness we record also the following equalities (of which Theorem
\ref{thm:bipartite} is a direct consequence).
\begin{cor} 
  For any matrix $\overline{B}$ of cluster type $B_n$ or $C_n$ we have
	\[
		\mathcal{C}^b_+(\overline{B})=\mathcal{C}_+(\overline{B}),
	\]
	\[
		\mathcal{D}^b(\overline{B})=\mathcal{D}(\overline{B}),
	\]
	\[
		\mathcal{C}^b_+(\overline{B})=\mathcal{D}^b(\overline{B}).
	\]
\end{cor}

\section{Proof of Theorem \ref{thm:corollaries}}
\label{sect:corollaries}
Here we derive Theorem \ref{thm:corollaries}. The claim
(\ref{thm:corollaries-4}) is a direct consequence of our description of $c$- and
$d$-vectors in Theorem \ref{thm:main}. For simply-laced types 
claims (\ref{thm:corollaries-1}) and (\ref{thm:corollaries-3}) 
follow from Corollaries \ref{cor:schur} and \ref{cor:independent1}.
However, for types $A_n$ and $D_n$, we provide a direct proof using Theorem
\ref{thm:main} without referring to the representation-theoretic results of
Section \ref{sect:back}.

As we did before we deal with types $A_n$ and $D_n$ first; we will use again a
folding argument to deduce the results for types $B_n$ and $C_n$. 

\subsection{Types \texorpdfstring{$A_n$}{An} and \texorpdfstring{$D_n$}{Dn}}
Let $B$ be any skew-symmetric integer matrix of cluster type either $A_n$ or
$D_n$.
Having built an explicit list of all the positive $c$-vectors and non-initial
$d$-vectors for the cluster algebra $\mathcal{A}_\bullet(B)$ with principal
coefficients we can give a combinatorial proof of Theorem \ref{thm:corollaries}.

\begin{prop}
	All $c$-vectors and $d$-vectors of	$\mathcal{A}_\bullet(B)$ are roots in the
	root system associated to the Cartan counterpart of $B$. Each of them is real
	if and only if its support in $X(B)$ is a tree.
	\label{prop:roots_simply_laced}
\end{prop}
\begin{proof}
	It suffices to establish the claim for positive $c$-vectors. 
	We are dealing with a local property: since the support of any $c$-vector
	$c$ of $\mathcal{A}_\bullet(B)$ is a connected sub-diagram of
	$X(B)$ it suffices to show that $c$ is a root in the root system
	associated to its support.

	The claim is clear for type $A_n$ and for cases I, II and III of type
	$D_n$: they are all roots in the corresponding finite type root system.
	
	Applying in sequence the simple reflections corresponding to the outermost
	node with multiplicity $2$ we can reduce case VII to case VI. We can then ``trim
	the branches'' reflecting each time with respect to a leaf of the diagram.
	After these reductions we are left with the four cases in Figure
	\ref{fig:prop-roots}.
	\begin{figure}[htbp]
		\begin{center}
			\includegraphics[scale=.5]{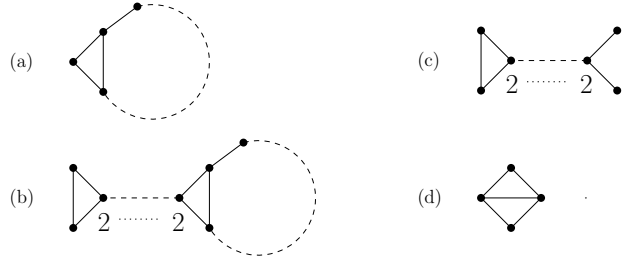}
		\end{center}
		\caption{Reduced $c$-vectors.}
		\label{fig:prop-roots}
	\end{figure}
	They all correspond to imaginary roots. Indeed let
	$c$ be any of these reduced $c$-vectors and let	$A$ be the generalized Cartan
	matrix associated to its support, then all the components of the vector $Ac$
	are non-positive which is exactly the condition of \cite[Lemma 5.3]{Kac90}.
\end{proof}

Let $\langle\cdot,\cdot\rangle$ be the Euler form of the quiver $Q=Q(B)$ associated
to $B$; it is defined on roots as follows:
\[
\langle \sum_{i\in I} c_i \alpha_i,\sum_{i\in I} d_i \alpha_i\rangle:=
\sum_{i\in I}c_i d_i-
\sum_{b_{ij}>0}b_{ij}c_j d_i.
\]

To show that elements of $\mathcal{C}_+(B)=\mathcal{D}(B)$ are Schur roots we will use
the following result of A. Schofield (\cite[Theorem 6.2]{schofield}).
\begin{thm}
	Let $\alpha$ be a positive root that is not a Schur root then $\alpha$ satisfies one of
	the following conditions:
	\begin{enumerate}
		\item 
			$\langle \alpha,\alpha\rangle=0$ and there are a positive (imaginary) root
			$\beta$ and a
			positive integer $k$ such that $\alpha=k\beta$.
			\label{thm:schofield-1}

		\item
			$\alpha$ is the sum of two positive roots, one of them (call it $\beta$) is real
			and satisfies
		  \[	
			\langle \alpha,\beta\rangle>0
			\quad\mbox{ and }\quad
			\langle \beta,\alpha\rangle>0.
			\]
			\label{thm:schofield-2}
	\end{enumerate}
	\label{thm:schofield}
\end{thm}
Let $A=A(B)$ be the Cartan counterpart of $B$. 
As noted in \cite{schofield}, if $\alpha$ is an imaginary root that is not
Schur, there are few possibilities for the positive real root $\beta$ satisfying
(\ref{thm:schofield-2}). Namely, if $w$ is the element of the Weyl group such
that all the components of the vector $Aw(\alpha)$ are non positive then
$w(\beta)$ has to be a negative real root.

\begin{prop}
	All the vectors in  $\mathcal{C}_+(B)=\mathcal{D}(B)$  are Schur roots of
	$\Delta\left( A(B) \right)$.
\end{prop}
\begin{proof}
	We are dealing still with a local property so we can assume that the
	$c$-vector we consider has full support.

	It is well known that if $X(B)$ is a finite type
	Dynkin diagram then any root is a Schur root (every indecomposable $kQ(B)$-module is
	rigid if $Q(B)$ is an orientation of a Dynkin diagram); therefore we need only to concentrate
	on cases IV, V, VI, VII, and VIII of type $D_n$. Let $c$ be
	any of these $c$-vectors, they are all imaginary roots. 
	None of them is an integer multiple
	of a root so case (\ref{thm:schofield-1}) of Theorem \ref{thm:schofield}
	is excluded and we need to show only that we are not in case
	(\ref{thm:schofield-2}).
	
	As noted in Proposition \ref{prop:roots_simply_laced} the elements $w$ of the Weyl group we
	need to apply to $c$ are those ``trimming the branches'';
	since the roots $\beta$ we are looking for change sign when acted on by $w$
	their support must be contained in only one of those appendices; we can
	therefore assume that there is only one appendix in
	the weighted diagram of $c$. 
	Label the nodes	on such an appendix with $\left\{ 1\dots,n-1 \right\}$ starting
	from the leaf; let $n$ be the node the appendix is connected to and let $m$ be the 
	innermost node with multiplicity $1$ in the
	appendix. It is clear that the element $w$  we are looking for is then
	$s_{n-1}\dots s_1$ in cases IV, V, VI, and VIII and
	$s_{n-1}\dots s_1 s_n\dots s_{m+1}$ in case VII. The possible roots $\beta$ are then 
	\[
	\alpha_1+\dots+\alpha_k
	\]
	for $k\in\left\{ 1,\dots,n-1 \right\}$ in cases IV, V, VI, and VIII and
	\begin{eqnarray*}
	\alpha_1+\dots+\alpha_{k+1} & m\leq k\leq n-1\\
	\alpha_1+\dots+\alpha_k & k<m\\
	\alpha_{m+1}+\dots+\alpha_k & m+1\leq k\leq n
	\end{eqnarray*}
	in case VII.
	By direct inspection we get that in all cases, regardless of the orientations,
	one of the two integers 
	$\langle c,\beta\rangle$ and 
	$\langle \beta,c\rangle$ is non-positive.
\end{proof}

\begin{prop}
	\label{prop:counting-AD}
	The cardinality $|\mathcal{C}_+(B)|=|\mathcal{D}(B)|$ depends only on the
	cluster type of $B$; it is equal to $n(n+1)/2$ if $B$ is of cluster type $A_n$
	and $n(n-1)$ if $B$ is of cluster type $D_n$.
\end{prop}
\begin{proof}
	Fix an element $X(B)$ of $\mathcal{X}(B)$.
	We need to count in how many different ways any diagrams from $\mathcal{W}(B)$
	can be embedded in $X(B)$.

	This count, for type $A_n$, was done by Parson (\cite[Lemma 5.8]{Parsons11})
	by noting that any embedding of a string is determined by the positions of its
	endpoints.

	Let us consider type $D_n$; there are four cases to be considered depending on
	which of the four diagrams in Figure \ref{fig:dynkin-Dn} describes $X(B)$. We
	present case (d): it involves all the techniques and it is the most complex
	one. The other cases can be dealt with in a similar fashion.

	The only weighted diagrams that can be embedded in a Dynkin diagram shaped as (d)
	are I, VI, VII, and VIII from Figure \ref{fig:allowed_diagrams_Dn}. 
	An embedding of any of those is uniquely determined by a pair of vertices in
	$X(B)$; for I (with at least two nodes), and VIII they are the two leaves; for
	VII they are the only leaf and the leftmost node with weight 2. For VI and
	strings of length 1 the two vertices of $X(B)$ coincide. 
	
	We are going to reverse this observation to count embeddings. Suppose that the
	central cycle contains $k$ vertices. 
	
	To each pair of vertices $i$ and $j$ not in the central cycle we can associate
	precisely two embeddings: if they belong to different components (say $X'$ and
	$X''$)	we have two strings passing on either side of the central cycle.  If
	$i$ and $j$ belong to the same type-$A_m$ component (say $X'$) and are
	distinct then we have a string connecting $i$ to $j$ completely contained in
	$X'$ and a weighted diagram of type VII or VIII depending on the relative
	position of $i$ and $j$. Finally if $i=j$ then we have a single point and a
	weighted diagram of type VI. They sum up to $(n-k)(n-k+1)$ embeddings.

	If one of the two vertices, say $i$, is in the central cycle and $j$ is in the
	component $X'$ then there are two possibilities: if $i$ is one of the two
	vertices adjacent to $X'$ then there is only one embedding associated to the
	pair $i$ and $j$: the shortest string connecting them. Otherwise there are two
	strings that we can embed into $X(B)$ depending on the side of the central
	cycle we cross.
	Therefore there are $2(k-2)(n-k)+2(n-k)$ embeddings with one vertex in a
	type-$A_m$ component and a vertex in the central cycle.

	Finally if both $i$ and $j$ are in the central cycle we need to distinguish
	three cases: they can coincide (yielding embedding of single nodes), they can
	be adjacent (and produce embedding of strings of length 2). Otherwise they
	produce precisely two embedding of strings. In total there are $k^2-k$
	embeddings induced by pair of vertices in the central cycle.

	Summing up all the contributions we get 
	\[
	(n-k)(n-k+1)+2(k-2)(n-k)+2(n-k)+k^2-k=n^2-n
	\]
	as desired.
\end{proof}

\subsection{Types \texorpdfstring{$B_n$}{Bn} and \texorpdfstring{$C_n$}{Cn}}
To extend the above results to types $B_n$ and $C_n$ we will use the following
general fact on the folding of root systems.

\begin{prop}
	Let $B$ be a skew-symmetrizable integer matrix together with an
	admissible automorphism $\sigma$ and denote by $A=A(B)$ its Cartan counterpart.
	Let $\overline{B}$ be the image of $B$ under the folding map $\pi$ and
	$\overline{A}=A(\overline{B})$ the Cartan counterpart of $\overline{B}$.
	Let $\left\{ \alpha_i \right\}_{i\in I}$ be the simple roots for 
	$\Delta(A)$ and $\left\{ \alpha_{\overline{\imath}}
	\right\}_{\overline{\imath}\in I/\sigma}$ be the simple roots for
	$\Delta(\overline{A})$. 

	Define the linear map $\pi$ from the root lattice of $\Delta(A)$ to the root
	lattice of $\Delta(\overline{A})$ by
	\[
	\pi(\alpha_i):=\alpha_{\overline{\imath}}.
	\]
	Then for any $\alpha\in\Delta(A)$ we have
	$\pi(\alpha)\in\Delta(\overline{A})$.
	\label{prop:folding_of_roots}
\end{prop}
\begin{proof}
	This argument is a refinement of	\cite[Proposition A.7]{tanisaki}; there the
	result is stated only for finite type root systems.
	
	Observe first that the map $\pi$ commutes with ``orbit reflections'' 
	\[
		s_{\overline{\imath}}^\sigma:=
		\prod_{t\in\overline{\imath}} s_t
	\]
	that is for any root $\alpha$ 
	\begin{equation}
			s_{\overline{\imath}}(\pi(\alpha))
			=
			\pi\left( s_{\overline{\imath}}^\sigma(\alpha) \right).
		\label{eqn:sp_commute}
	\end{equation}
	Orbit reflections are well defined because, by admissibility of
	$\sigma$, we have 
	\[a_{i_1i_2}=0\]
	for any pair $i_1\neq i_2$ in the same $\sigma$-orbit $\overline{\imath}$.
	It is sufficient to verify (\ref{eqn:sp_commute}) on simple roots; we have
	\begin{align*}
		s_{\overline{\imath}}\left( \pi(\alpha_j) \right)=
		s_{\overline{\imath}}\left( \alpha_{\overline{\jmath}} \right)=
		\alpha_{\bar{\jmath}}-a_{\bar{\imath}\bar{\jmath}}\alpha_{\bar{\imath}}=
		\pi\left( \alpha_j-\sum_{t\in\bar{\imath}}a_{tj}\alpha_t \right)=
		\pi\left(s_{\overline{\imath}}^\sigma(\alpha_j) \right).
	\end{align*}

	Back to our problem, without loss of generality we can assume $\alpha$ to be a
	positive root; we will proceed by induction on 
	\[
	\mathrm{ht}(\alpha)=\mathrm{ht}\left( \sum_{i\in I}c_i\alpha_i \right):=\sum_{i\in I}c_i
	\]
	If $\mathrm{ht}(\alpha)=1$ then $\alpha=\alpha_i$ for some $i\in I$; thus
	$\pi(\alpha)=\alpha_{\bar{\imath}}$. 
	Suppose now that $\mathrm{ht}(\alpha)>1$. If all the components of the vector
	$\overline{A}\pi(\alpha)$ are negative then $\pi(\alpha)$ is an imaginary
	root (see \cite[Lemma 5.3]{Kac90}).
	Otherwise let $\overline{\imath}$ be such that 
	\begin{equation}
		\left( \overline{A}\pi(\alpha)  \right)_{\overline{\imath}} > 0.
		\label{eqn:positive_reflection}
	\end{equation}
	Set $\alpha':=s_{\overline{\imath}}^\sigma(\alpha)$. Since $\overline{\imath}$ is disconnected, in view of
	(\ref{eqn:positive_reflection}) 
	$\alpha'$ is a positive root and 
	\[
	\mathrm{ht}\left(\alpha' \right) <
	\mathrm{ht}(\alpha).
	\]
	By induction hypothesis then $\pi\left( \alpha' \right)$
	is a positive root in the root system of $\Delta(\overline{A})$ therefore
	so is 
	\[
	\pi(\alpha)=
	s_{\overline{\imath}}\left( s_{\overline{\imath}}(\pi(\alpha))\right)=
	s_{\overline{\imath}}\left( \pi\left(s_{\overline{\imath}}^\sigma(\alpha) \right) \right)=
	s_{\overline{\imath}}\left( \pi(\alpha')\right).
	\]
\end{proof}
Note that the folding of roots agrees with the folding of both $c-$ and
$d-$vectors.
\begin{prop}
	Let $B$ be a skew-symmetrizable integer matrix of cluster type $B_n$
	or $C_n$. All the $c$-vectors and $d$-vectors of $\mathcal{A}_\bullet(B)$ are
	roots in the root system $\Delta(A(B))$.
	\label{prop:folding_c-vector-roots}
\end{prop}
\begin{proof}
	It is enough to consider positive $c$-vectors.
	By Corollary \ref{cor:folding} any element of $\mathcal{C}_+(B)$ is
	the image of some $c$-vector of a cluster algebra of type $D_{n+1}$ or
	$A_{2n-1}$. 
	By Proposition \ref{prop:roots_simply_laced} the latter are roots in the root
	system associated to the unfolding of $B$.  Our claim follows then directly
	from Proposition \ref{prop:folding_of_roots}.
\end{proof}

\begin{prop}
	Any $c$-vector ($d$-vector) of $\mathcal{A}_\bullet(B)$ of type  $B_n$ or
	$C_n$ is a real root if and only if its support is a tree.
\end{prop}
\begin{proof}
	If the support of the vector we are considering is a tree there is nothing to
	show. In all other cases we can ``trim the branches'' and check directly as we
	did in Proposition \ref{prop:roots_simply_laced}.
\end{proof}
\begin{prop}
	For any $B$-matrix of cluster type either $B_n$ or $C_n$ the cardinality of
	$\mathcal{V}(B)$ is equal to $n^2$.
\end{prop}
\begin{proof}
	This claim does not follow directly from folding. Nevertheless it is
	straightforward to apply the same argument of Proposition
	\ref{prop:counting-AD} to perform the counting. 
\end{proof}

\subsection*{Acknowledgments}
We thank Anna Felikson, Bernhard Keller, Robert Marsh, Gregg Musiker, Satoshi
Naito, Idun Reiten,  Pavel Tumarkin, Toshiyuki Tanisaki, Gordana Todorov, and
Jerzy Weyman for useful discussion.  We thank Alfredo N\'ajera Ch\'avez for
letting us use his result \cite{Najera12b} prior to the publication, Hugh
Thomas for his guidance on representation theory, and the anonymous referee for
the useful suggestions.  We also thank MSRI,
Berkeley for financial support, for the computer facilities used when dealing
with exceptional types, and for providing the ideal environment where the
last stage of this work was done.

We would like to dedicate this paper to the memory of Andrei Zelevinsky.
We are grateful to him for sharing his insights and ideas on the subject, and
suggesting the coincidence of the $c$- and $d$-vectors to us at the early stage 
of this work before the appearance of the paper \cite{Najera12}.
The second author is especially grateful to him for his guidance, support and
constant spur during the course of his Ph.D. studies.


\appendix
\section{Type \texorpdfstring{$D_n$}{Dn} analysis}
\label{app:type-Dn}
We provide here the detailed case analysis required to prove both Propositions
\ref{prop:bipartite} and \ref{prop:An_Dn-support} in type $D_n$ at the same time.

As explained above, any multilamination
corresponding to an initial triangulation decomposes the surface $S$ into pieces
(see Figure \ref{fig:example_glueing} for an example). Any quadrilateral can
intersect positively at most laminations contained in three different pieces. We
need therefore to consider, for any quadrilateral in Figure
\ref{fig:shear-coordinates}, all the possible ways of inscribing it in a surface
with at most three pieces. These configurations are listed in the leftmost column of
the following tables.

For each of them we distinguish five sub-cases (the other five columns)
depending on how the multilamination looks around the puncture (cf. Figure
\ref{fig:wheel-decomposition}). 

For any given configuration and choice of lamination around the puncture 
we provide a bipartite quadrilateral giving rise
to the same $c$-vector as the original triangulation and we record which
template $c$-vector from Figure \ref{fig:allowed_diagrams_Dn} we get. 
	\begin{figure}[H]
		\begin{center}
  		\includegraphics[scale=1.05,angle=90]{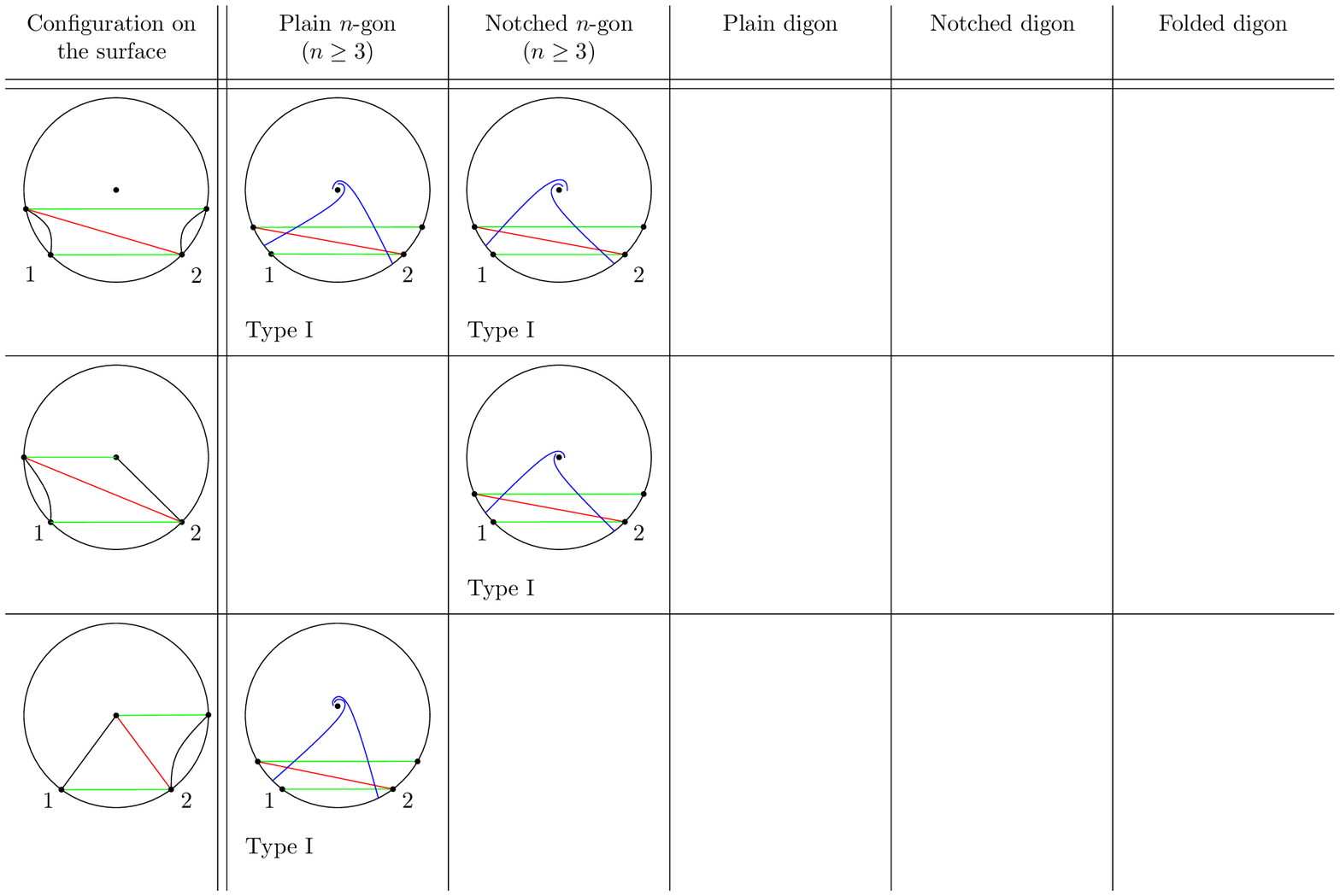}
		\end{center}
    	\label{fig:type-D_n-1}
	\end{figure}
 
	\begin{figure}[H]
		\begin{center}
  		\includegraphics[scale=1.05,angle=90]{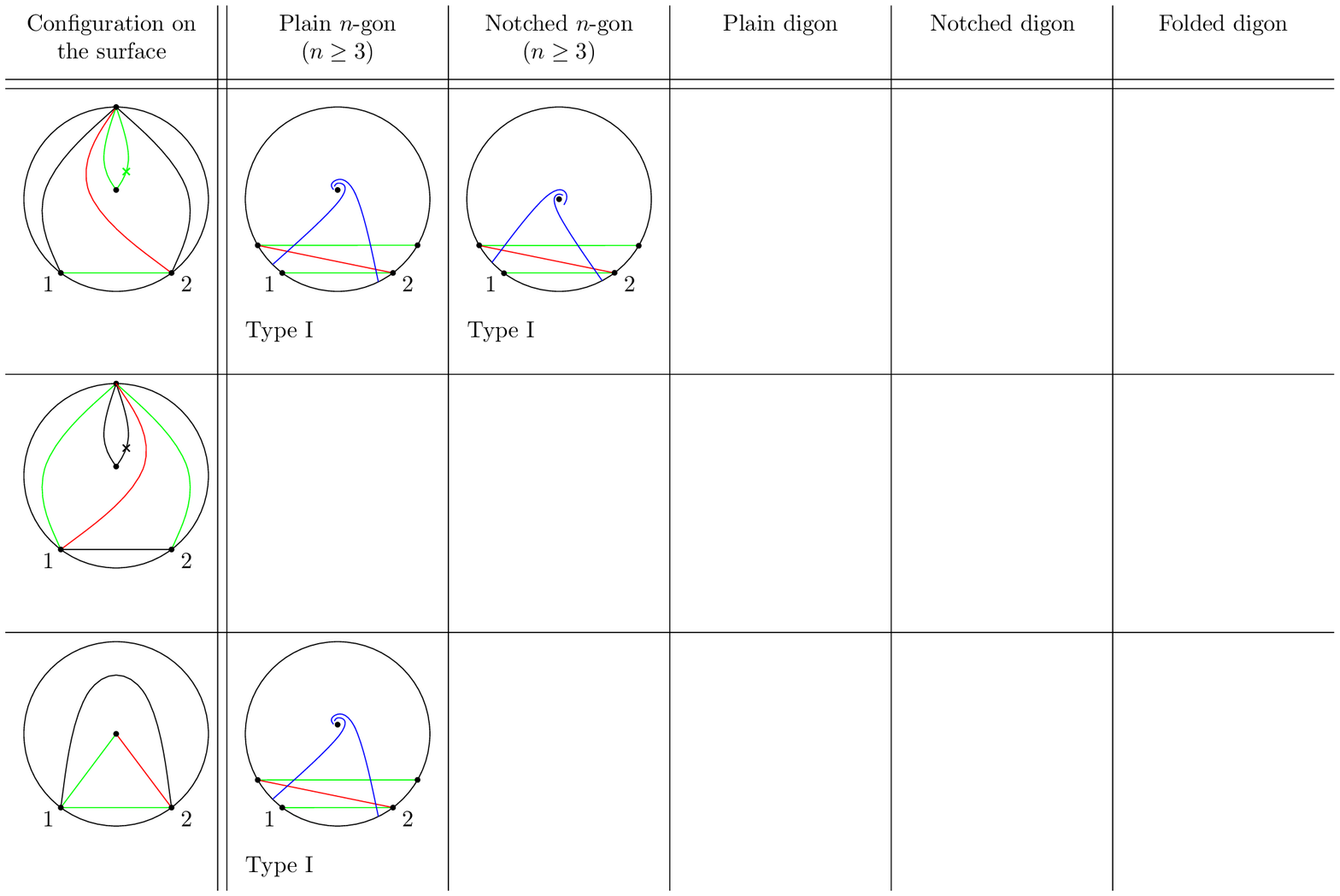}
		\end{center}
    	\label{fig:type-D_n-2}
	\end{figure}
 
	\begin{figure}[H]
		\begin{center}
  		\includegraphics[scale=1.05,angle=90]{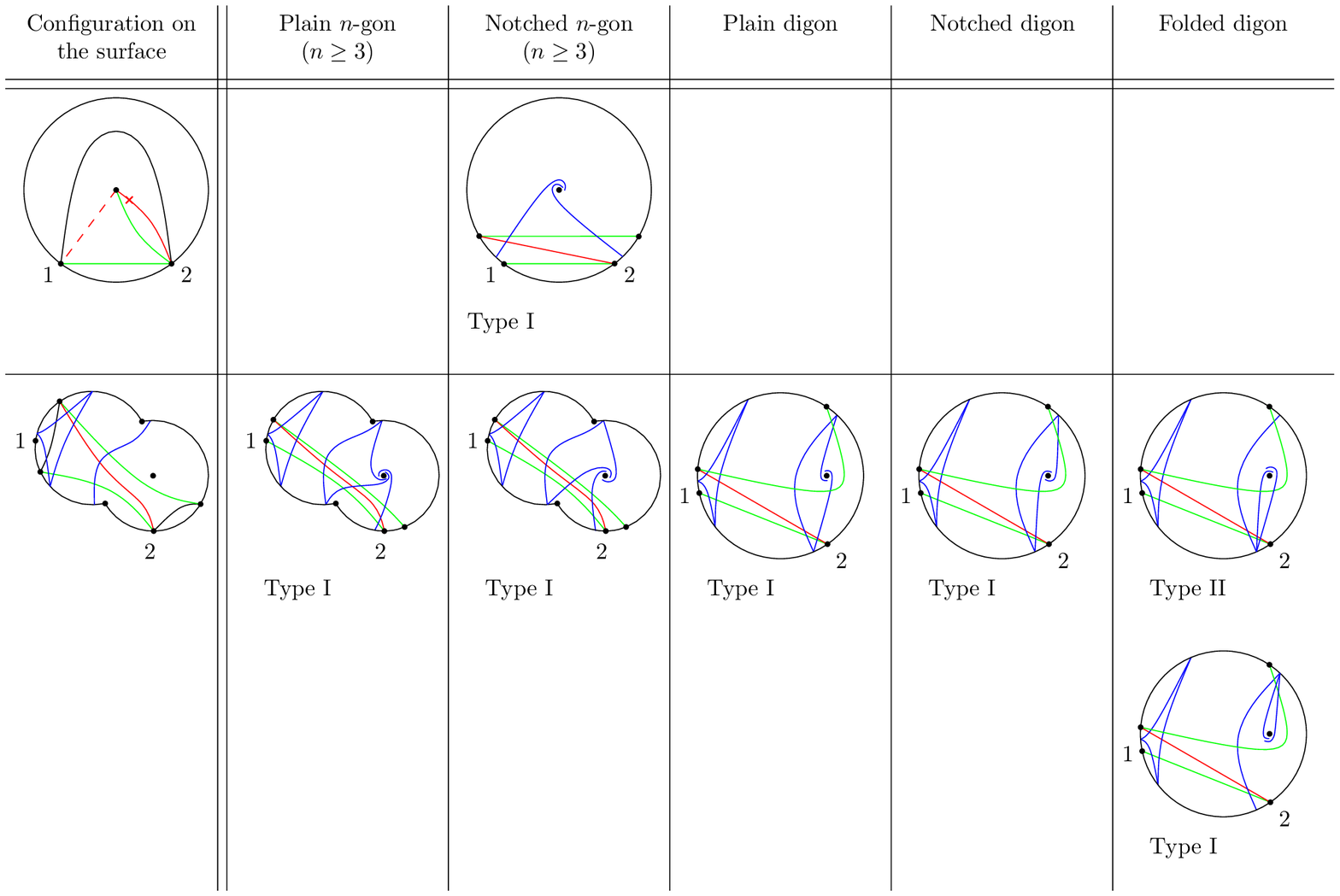}
		\end{center}
    	\label{fig:type-D_n-3}
	\end{figure}
 
	\begin{figure}[H]
		\begin{center}
  		\includegraphics[scale=1.05,angle=90]{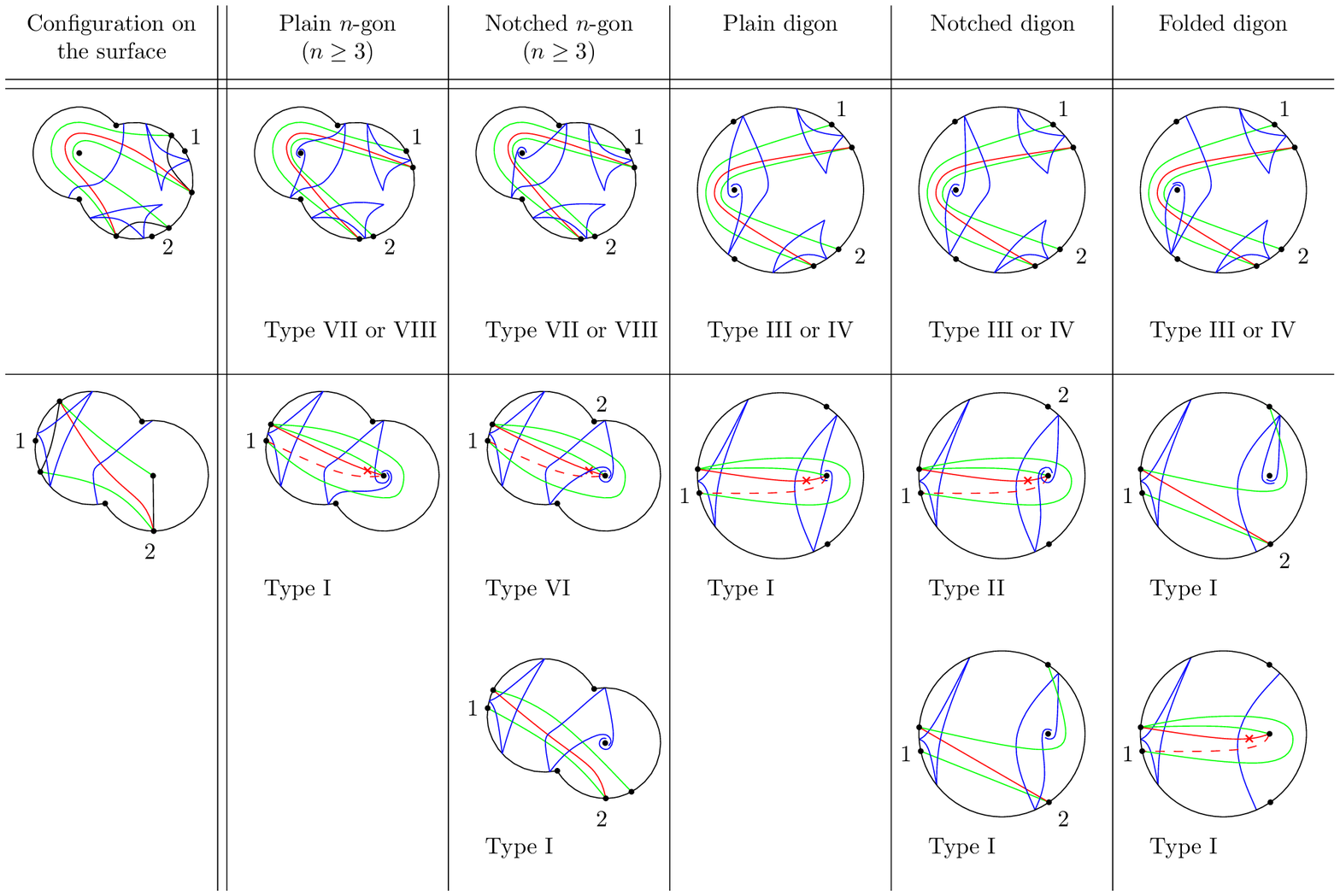}
		\end{center}
    	\label{fig:type-D_n-4}
	\end{figure}
 
	\begin{figure}[H]
		\begin{center}
  		\includegraphics[scale=1.05,angle=90]{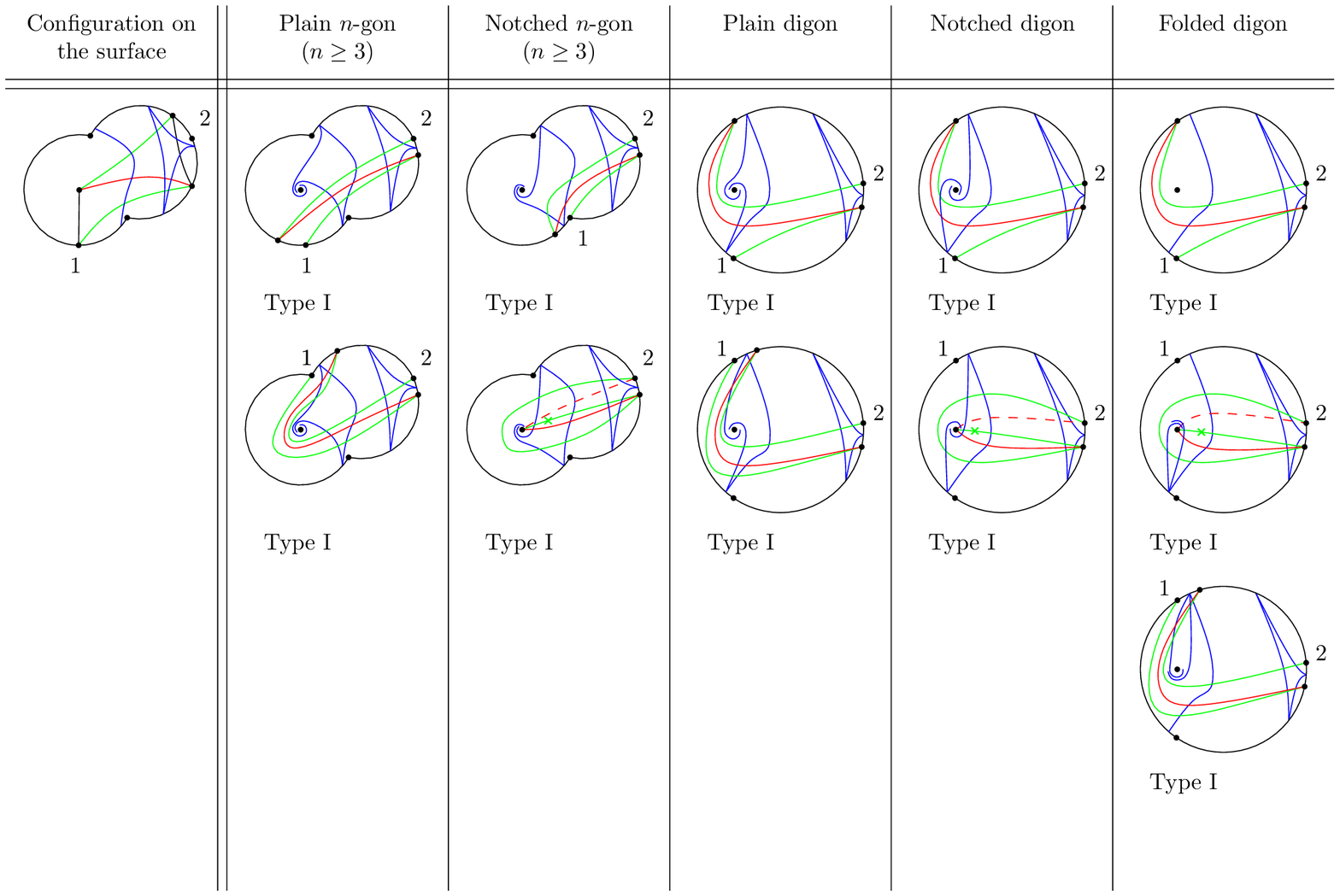}
		\end{center}
    	\label{fig:type-D_n-5}
	\end{figure}
 
	\begin{figure}[H]
		\begin{center}
  		\includegraphics[scale=1.05,angle=90]{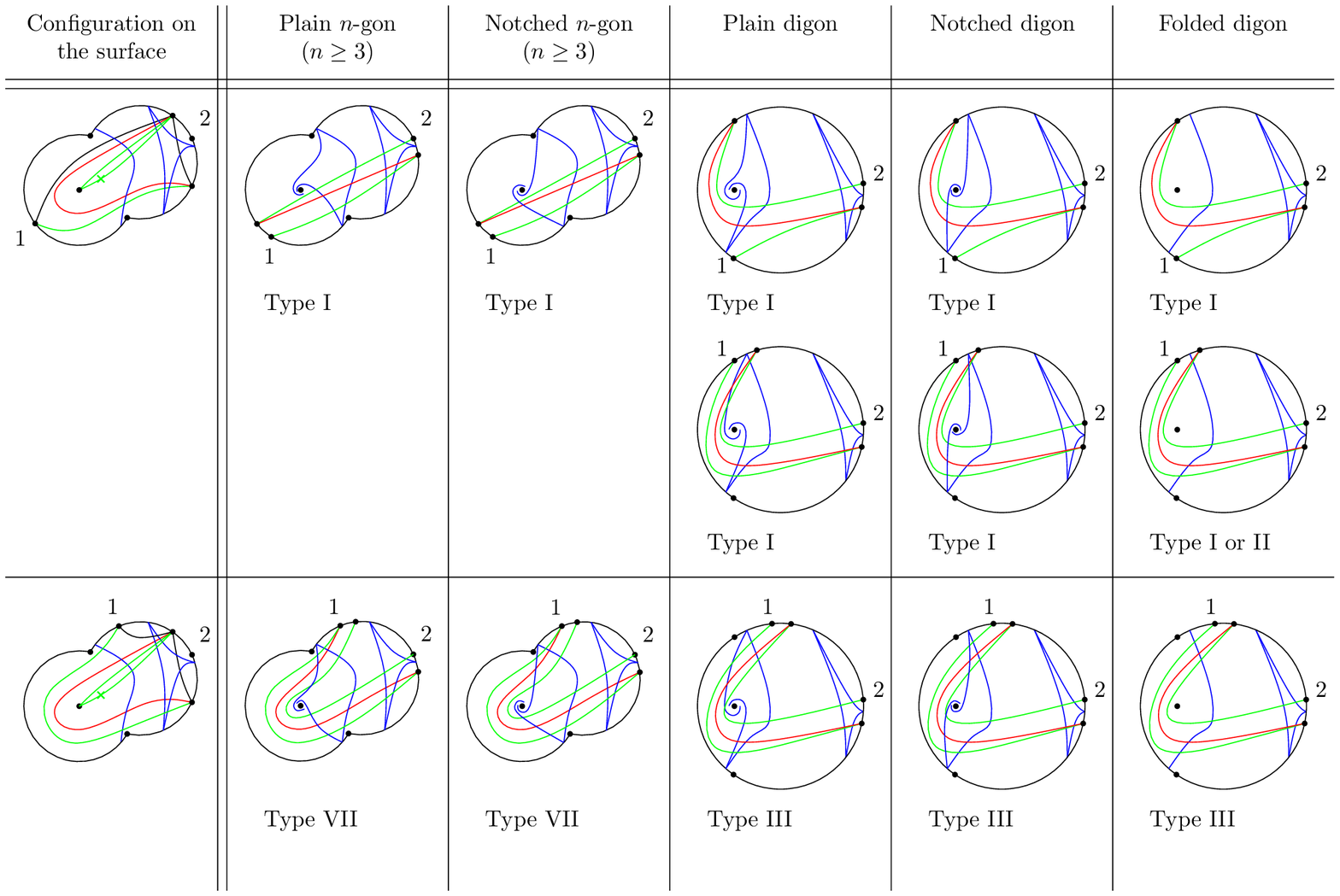}
		\end{center}
    	\label{fig:type-D_n-6}
	\end{figure}
 
	\begin{figure}[H]
		\begin{center}
  		\includegraphics[scale=1.05,angle=90]{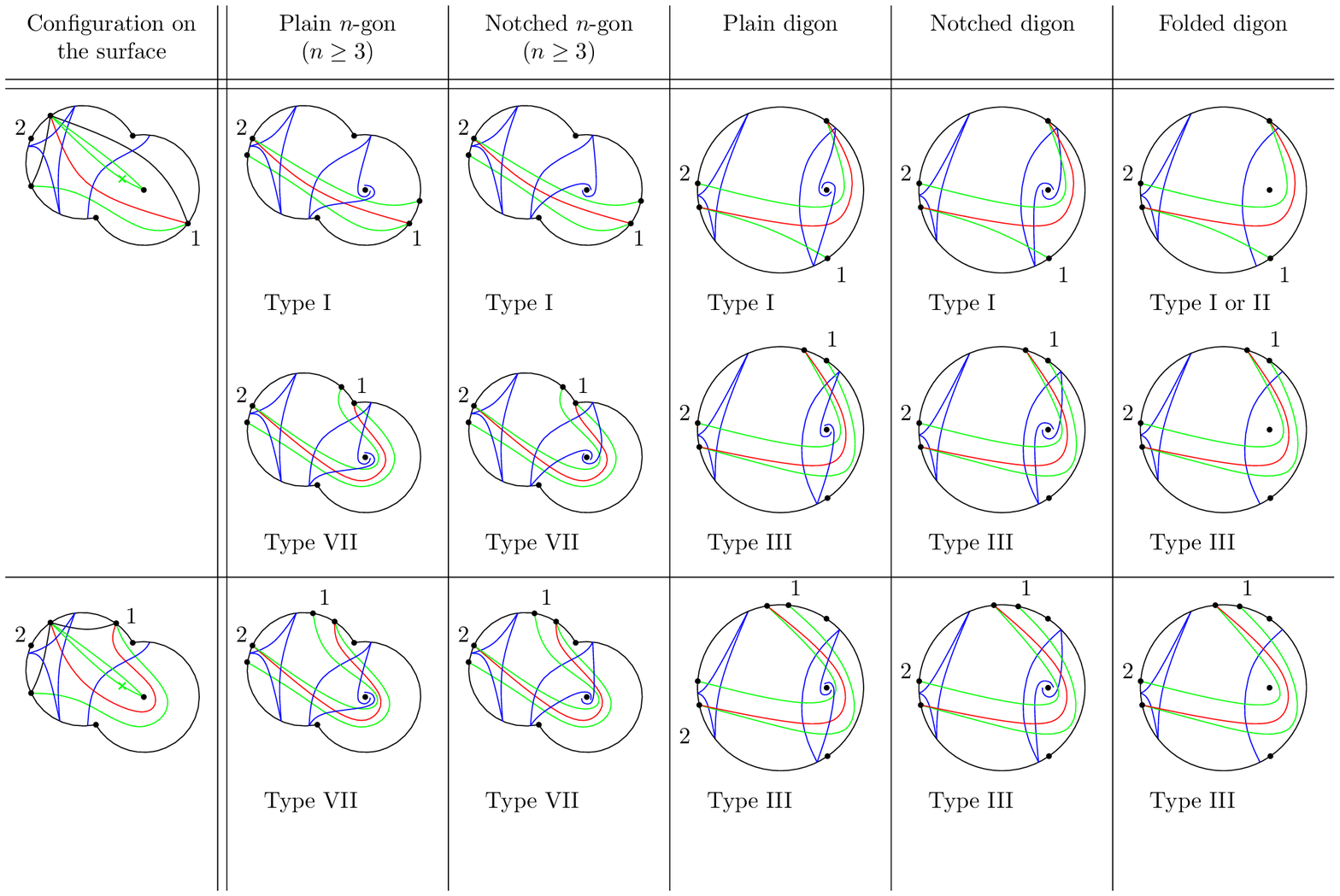}
		\end{center}
    	\label{fig:type-D_n-7}
	\end{figure}
 
	\begin{figure}[H]
		\begin{center}
  		\includegraphics[scale=1.05,angle=90]{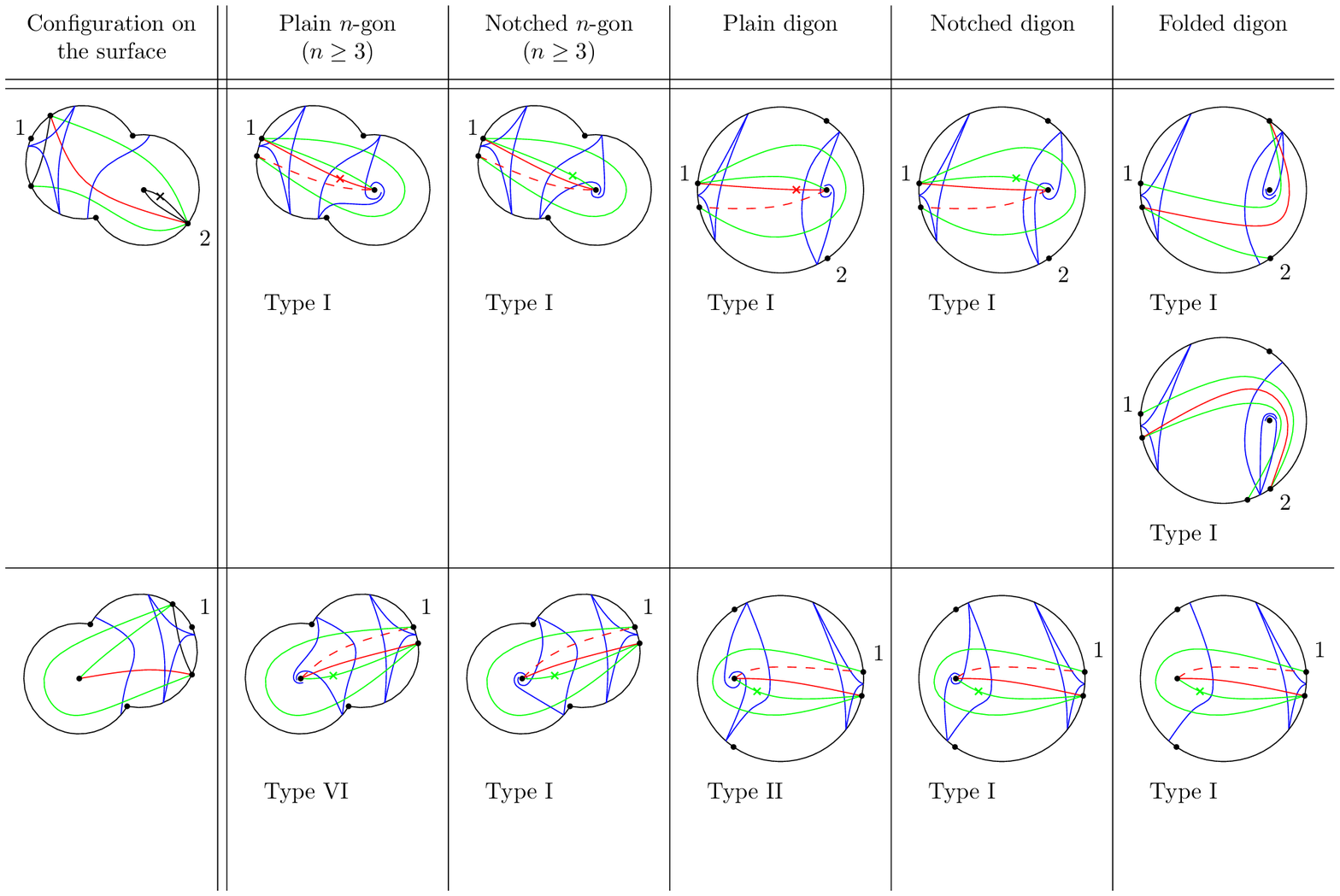}
		\end{center}
    	\label{fig:type-D_n-8}
	\end{figure}
 
	\begin{figure}[H]
		\begin{center}
  		\includegraphics[scale=1.05,angle=90]{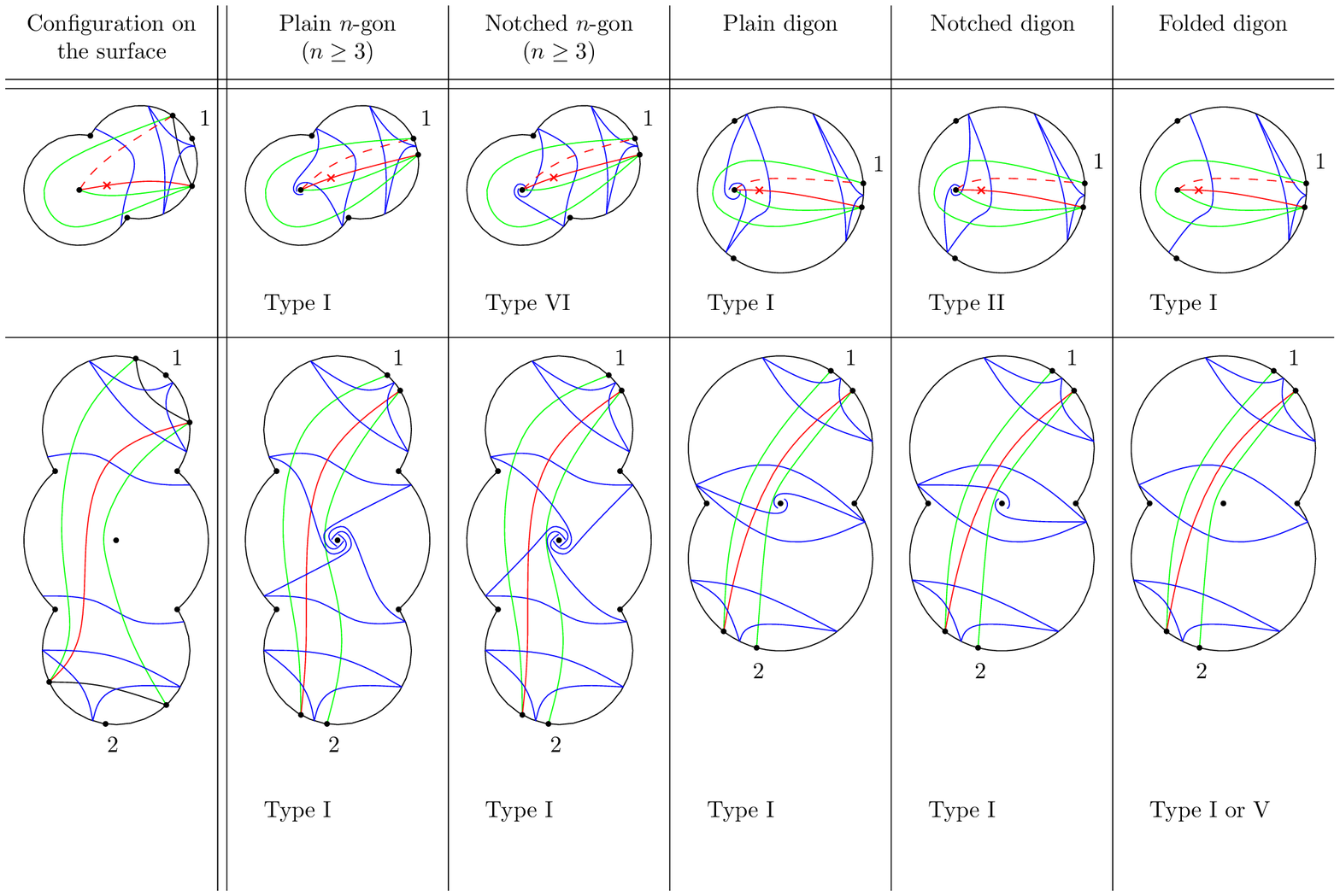}
		\end{center}
    	\label{fig:type-D_n-9}
	\end{figure}
\newpage

\section{The sets \texorpdfstring{$\mathcal{X}(Z)$}{X(Z)} and
	\texorpdfstring{$\mathcal{W}(Z)$}{W(Z)} for the exceptional types.}
\label{app:exceptional}
For exceptional types we obtain a description of
$\mathcal{X}(Z)$ by direct inspection using 
\cite{Keller08c,Musiker10}. Similarly we obtain 
$\mathcal{W}(Z)$ and
check Theorems \ref{thm:main}, \ref{thm:corollaries}, and \ref{thm:bipartite}. 


\subsection{Type \texorpdfstring{$G_2$}{G2}}
\begin{figure}[htpb]
	\begin{center}
		\includegraphics[scale=\scalingconstant]{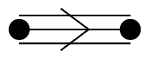}
	\caption{The only diagram in $\mathcal{X}(G_2)$.}
	\label{fig:dynkin-G2}
	\vspace{\baselineskip}
		\includegraphics[scale=\scalingconstant]{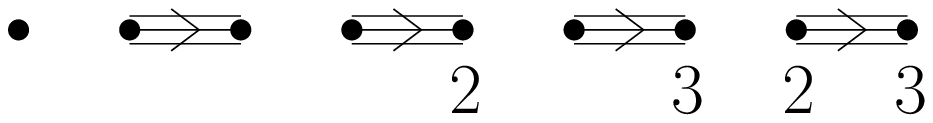}
	\end{center}
	\caption{The set $\mathcal{W}(G_2)$.}
	\label{fig:allowed_diagrams_G2}
\end{figure}

\subsection{Type \texorpdfstring{$F_4$}{F4}}
\begin{figure}[H] 
	\begin{center}
		\includegraphics[scale=\scalingconstant]{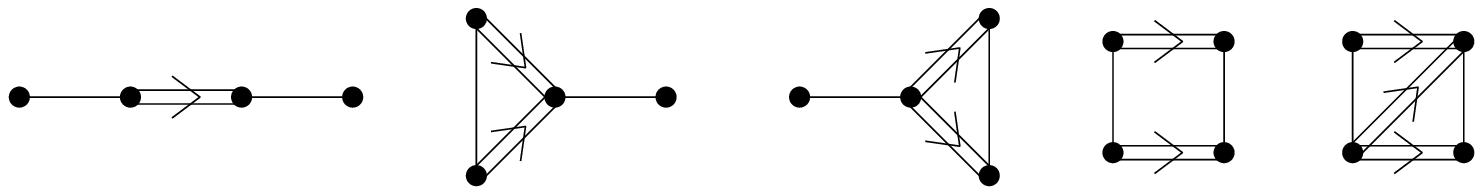}
	\caption{Diagrams in $\mathcal{X}(F_4)$.}
	\label{fig:dynkin-F4}
	\vspace{\baselineskip}
		\includegraphics[scale=\scalingconstant]{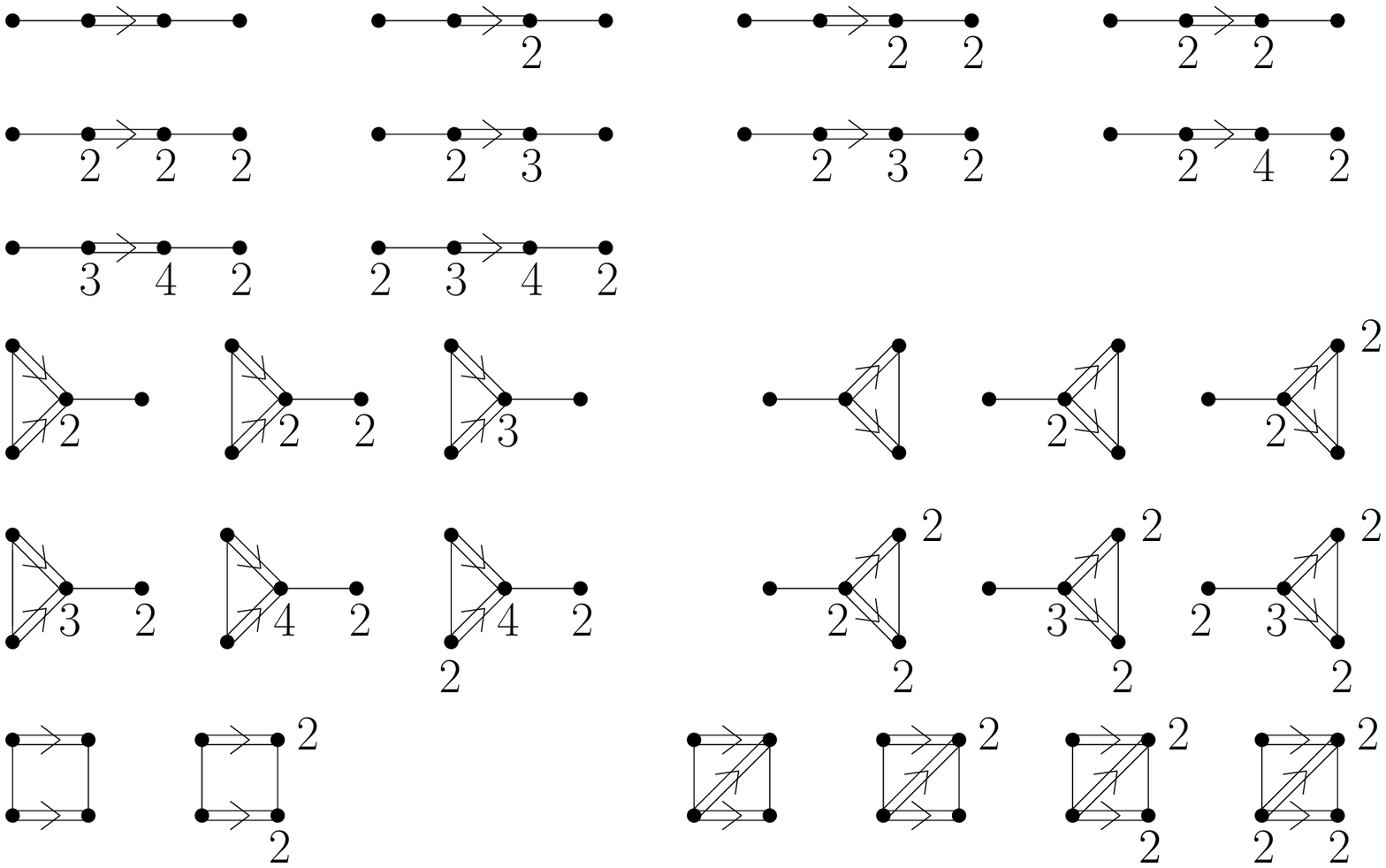}
	\end{center}
	\caption{The set $\mathcal{W}(F_4)$ consists of the above weighed diagrams
		together with all the elements of $\mathcal{W}(B_3)$ and $\mathcal{W}(C_3)$.}
	\label{fig:allowed_diagrams_F4}
\end{figure}
\subsection{Type \texorpdfstring{$E_6$}{E6}}
\begin{figure}[H]
	\begin{center}
		\includegraphics[scale=\scalingconstant]{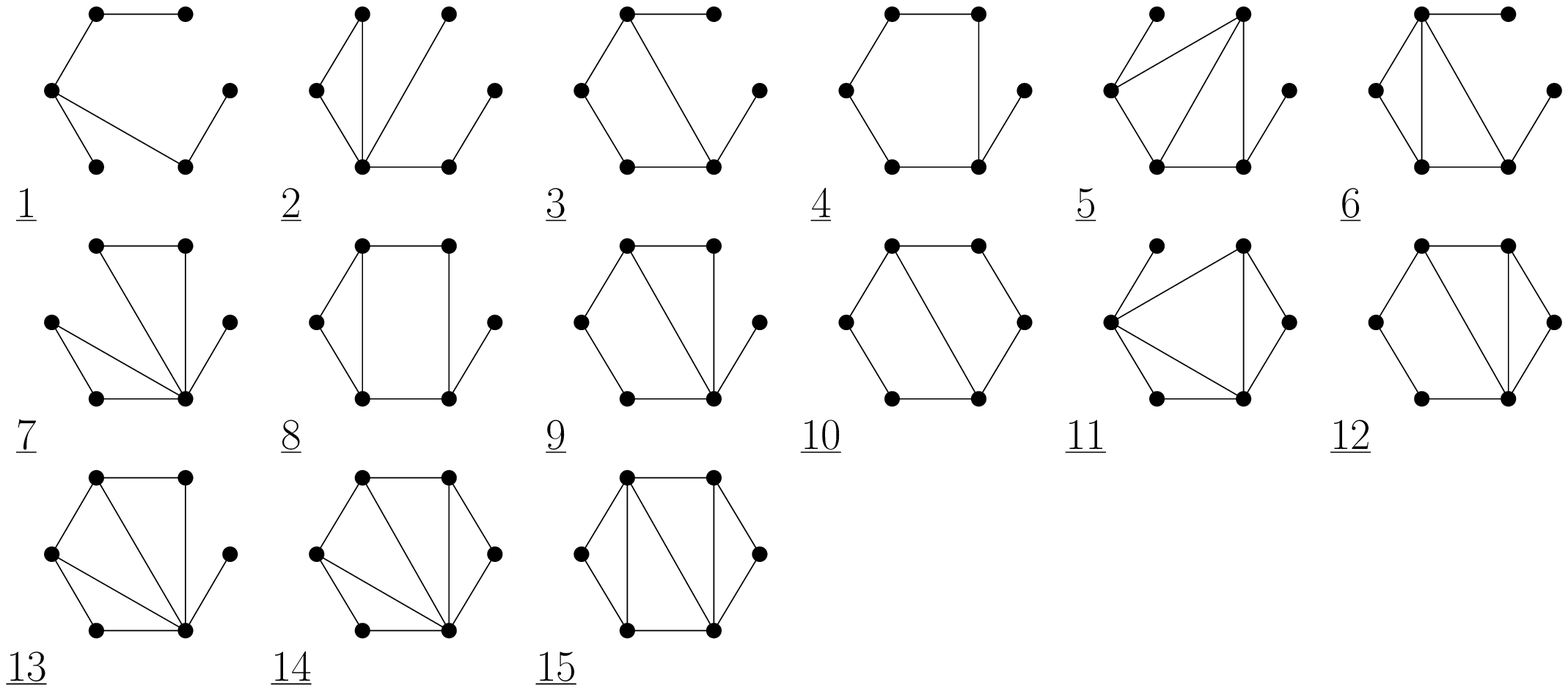}
	\caption{Diagrams in $\mathcal{X}(E_6)$.}
	\label{fig:dynkin-E6}
	\vspace{\baselineskip}
		\includegraphics[scale=\scalingconstant]{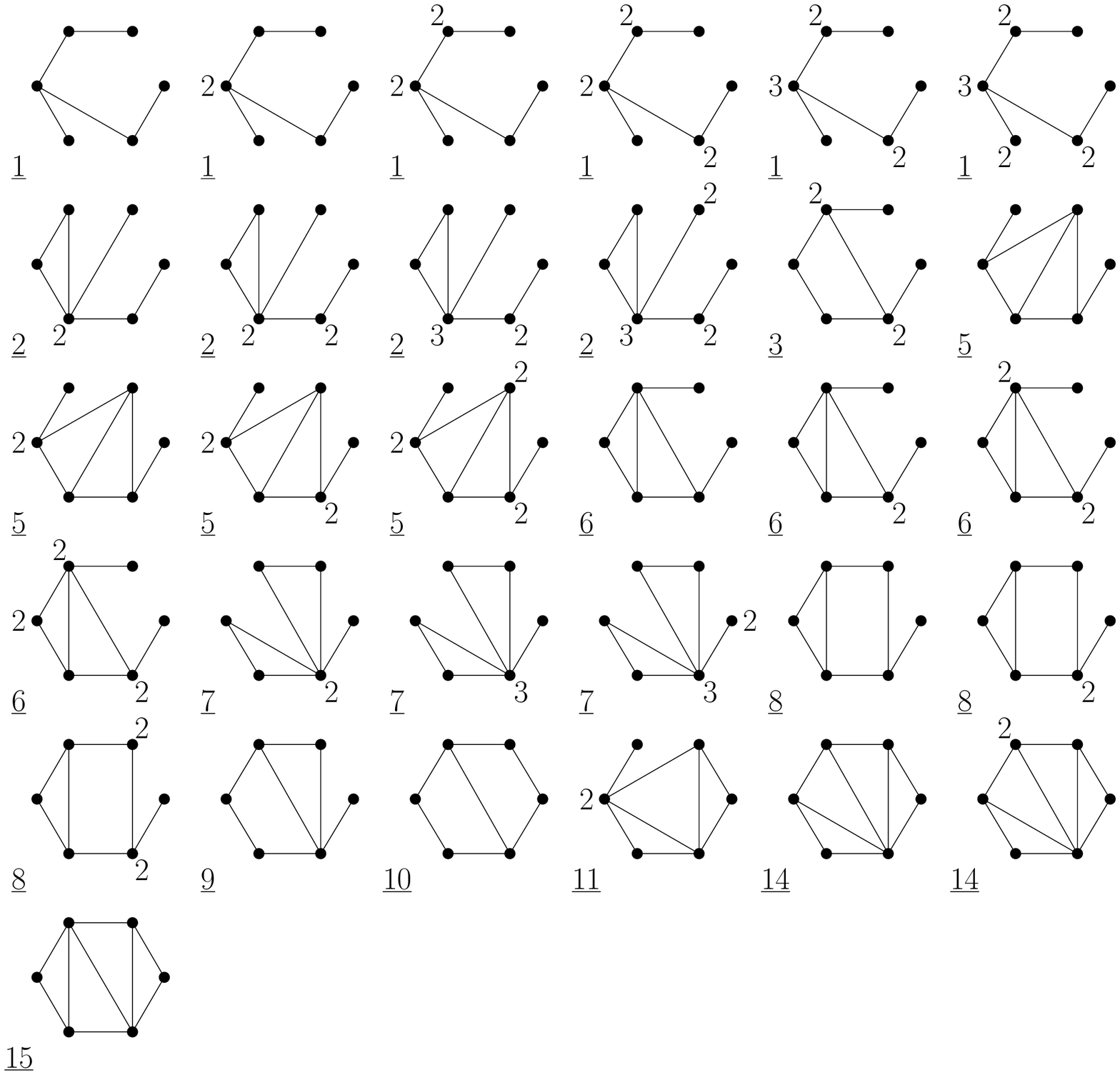}
	\caption{The set $\mathcal{W}(E_6)$ consists of the above weighted diagrams together
	with all the elements of $\mathcal{W}(A_5)$ and $\mathcal{W}(D_5)$.}
		\label{fig:allowed_diagrams_E6}
	\end{center}
\end{figure}

\subsection{Type \texorpdfstring{$E_7$}{E7}}
\begin{figure}[H]
	\begin{center}
		\includegraphics[scale=\scalingconstant]{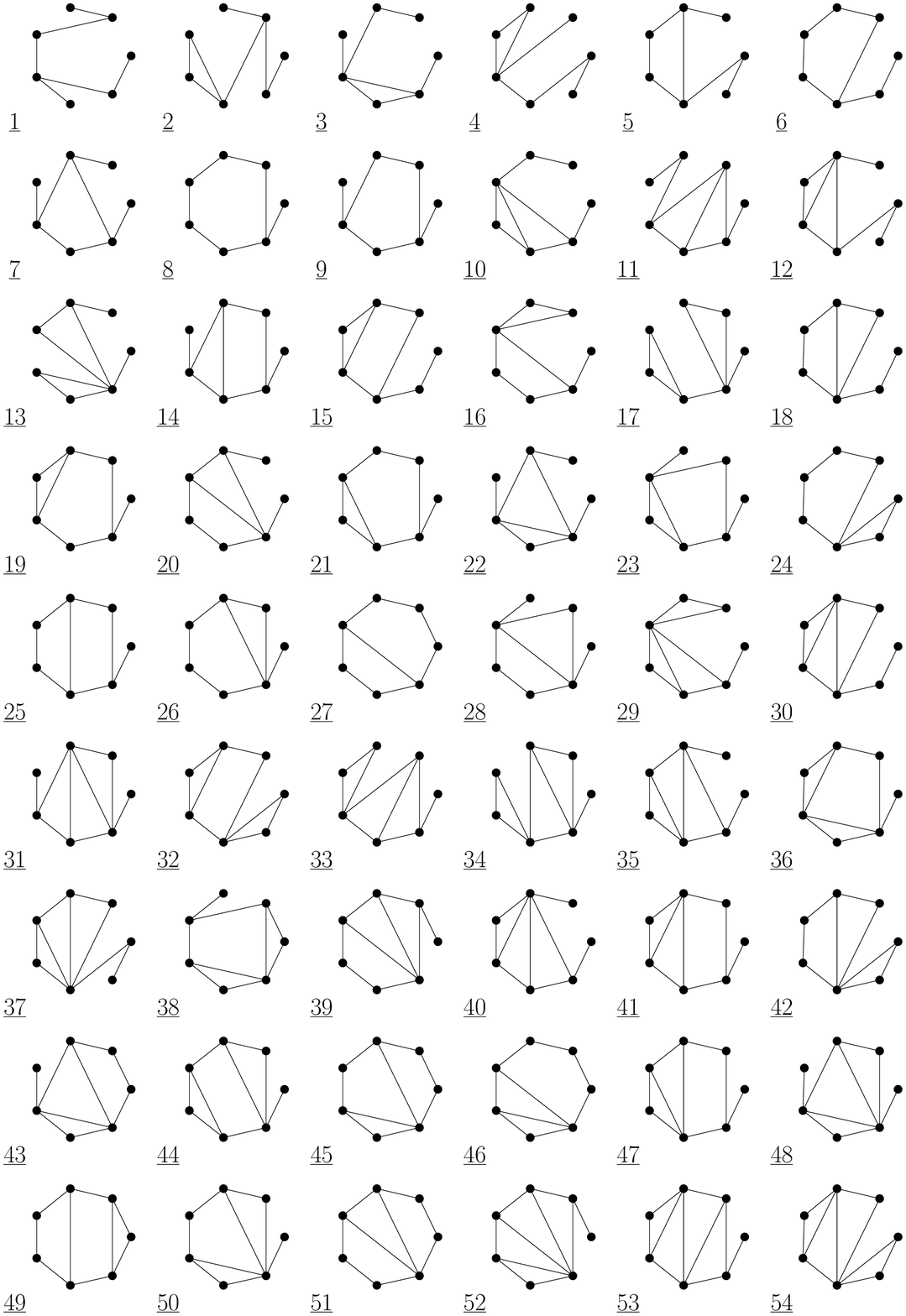}
	\end{center}
	\caption{Diagrams in $\mathcal{X}(E_7)$.}
	\label{fig:dynkin-E7-first}
\end{figure}
\begin{figure}[H]
	\begin{center}
		\includegraphics[scale=\scalingconstant]{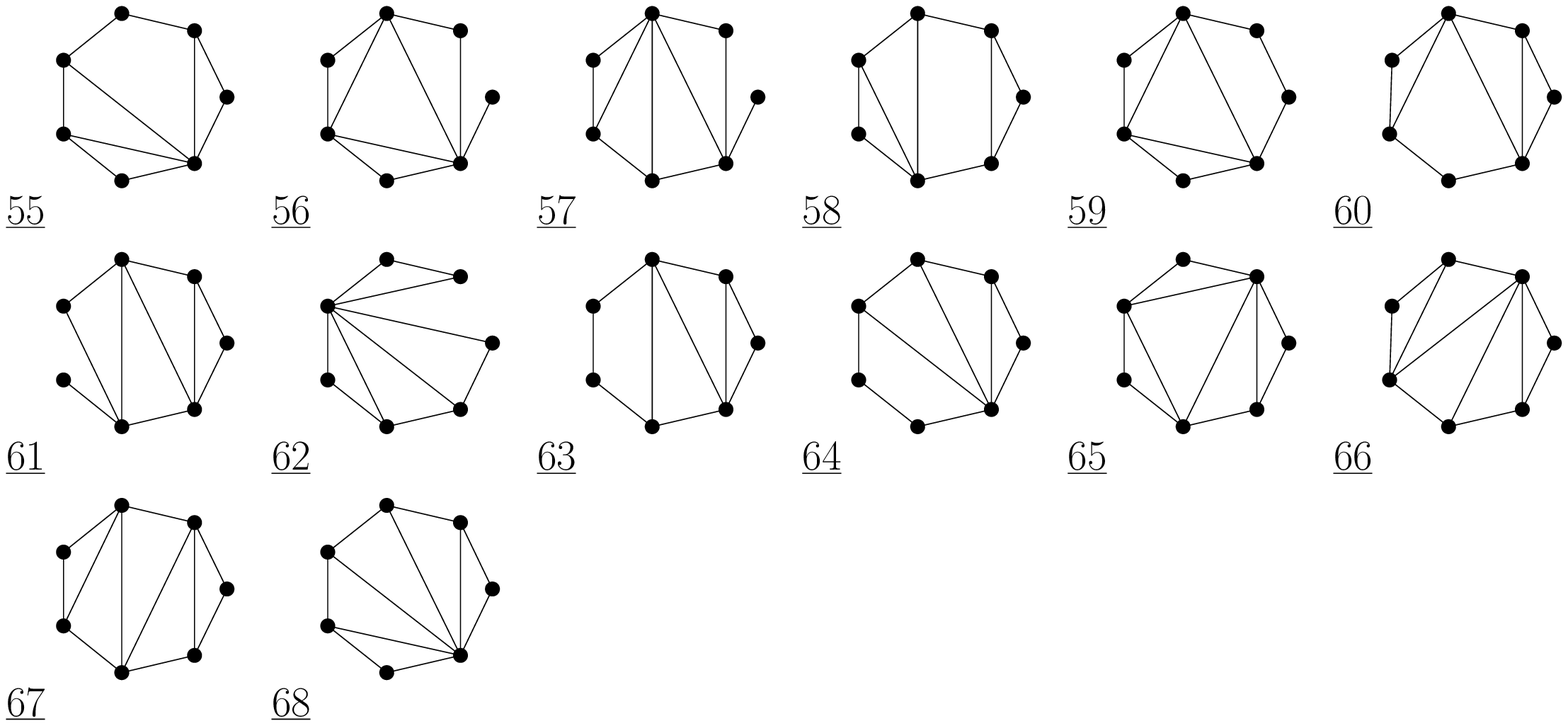}
	\caption{Diagrams in $\mathcal{X}(E_7)$ (continued).}
	\label{fig:dynkin-E7-last}
	\vspace{\baselineskip}
 	\includegraphics[scale=\scalingconstant]{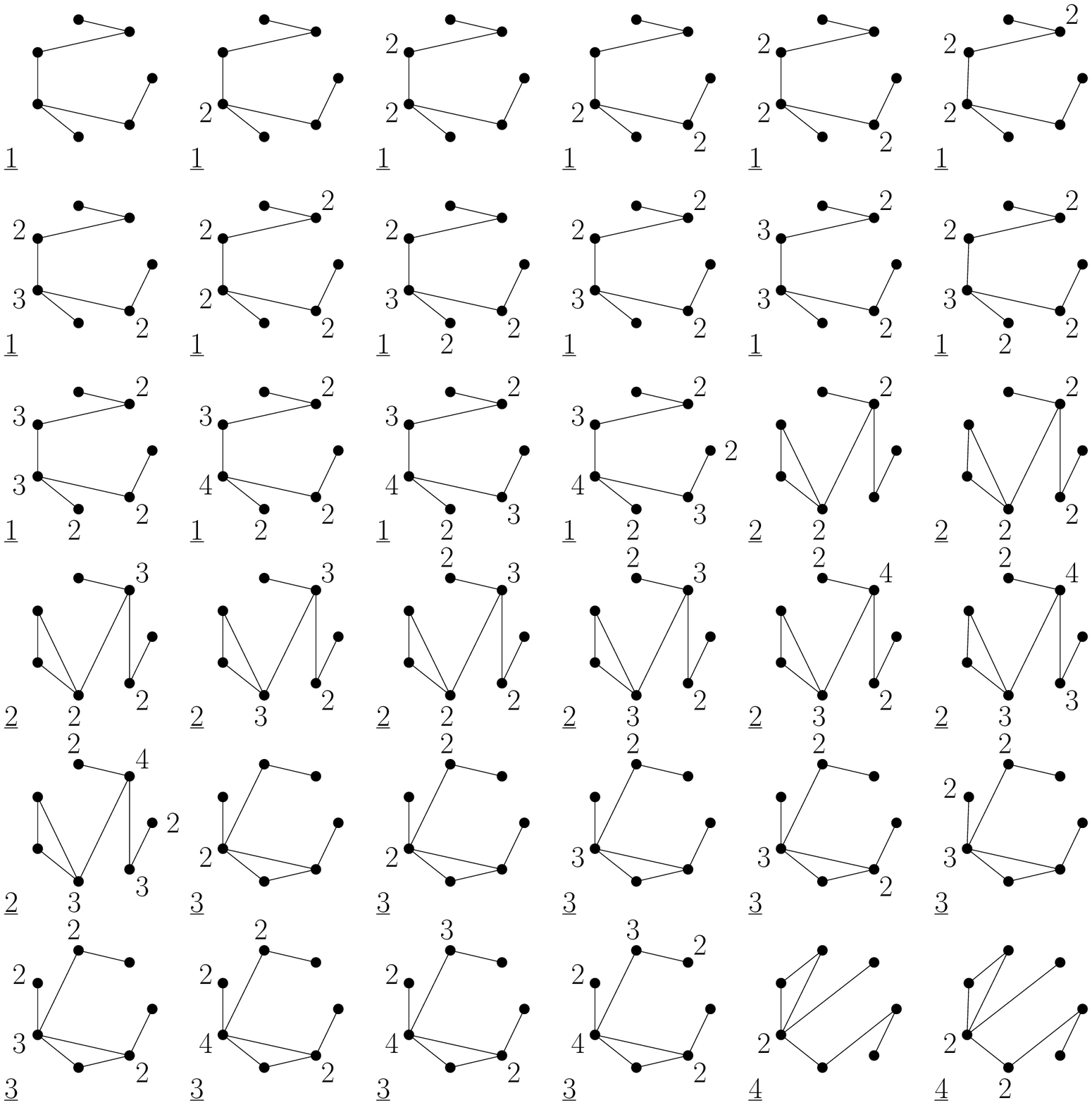}
	\caption{The set $\mathcal{W}(E_7)$ consists of the above weighted diagrams
	and all the elements of $\mathcal{W}(A_6)$, $\mathcal{W}(D_6)$,
	and $\mathcal{W}(E_6)$.}
		\label{fig:allowed_diagrams_E7-first}
	\end{center}
\end{figure}
\begin{figure}[H]
	\begin{center}
		\includegraphics[scale=\scalingconstant]{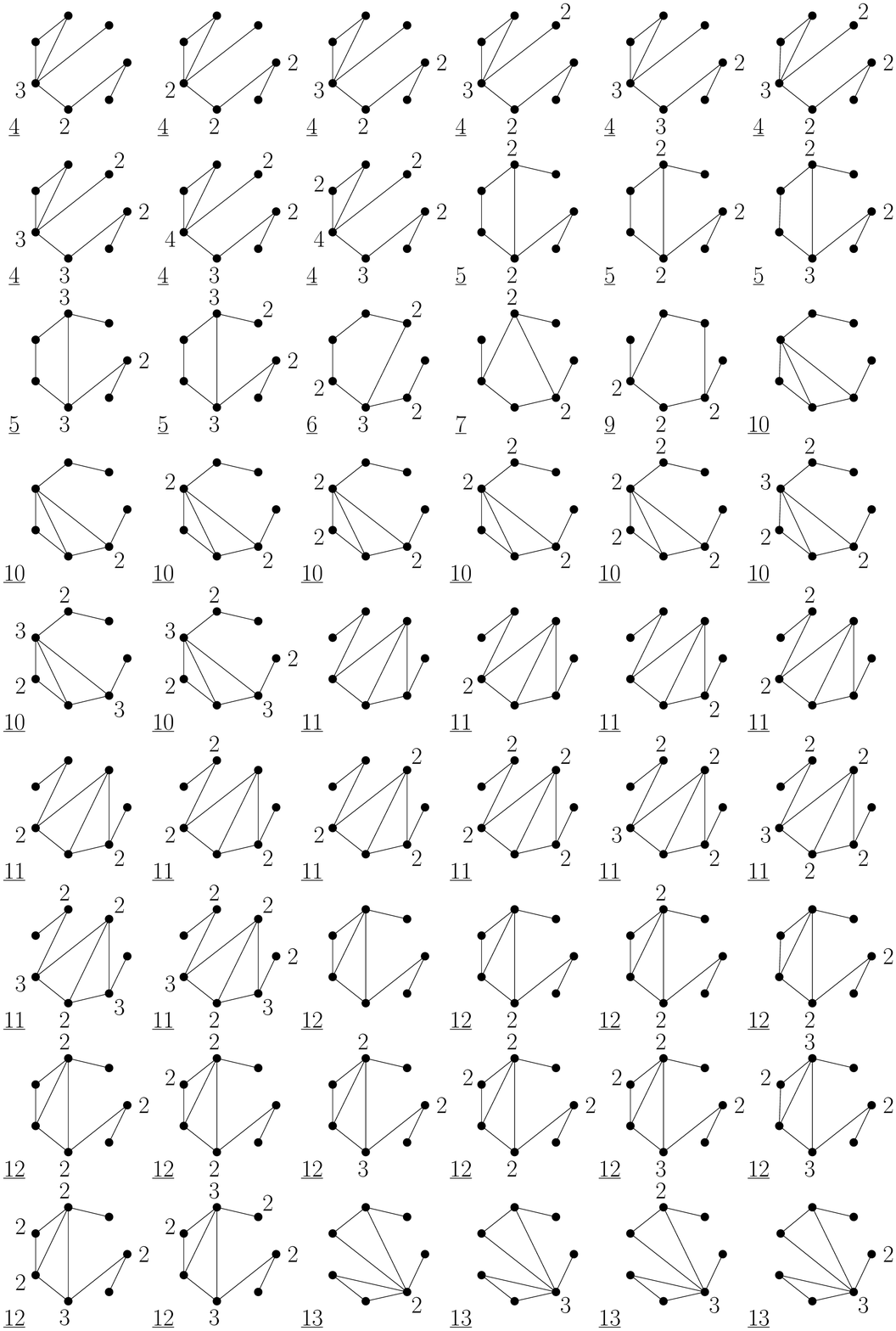}
	\caption{The set $\mathcal{W}(E_7)$ (continued).}
	\end{center}
\end{figure}
\begin{figure}[H]
	\begin{center}
	 	\includegraphics[scale=\scalingconstant]{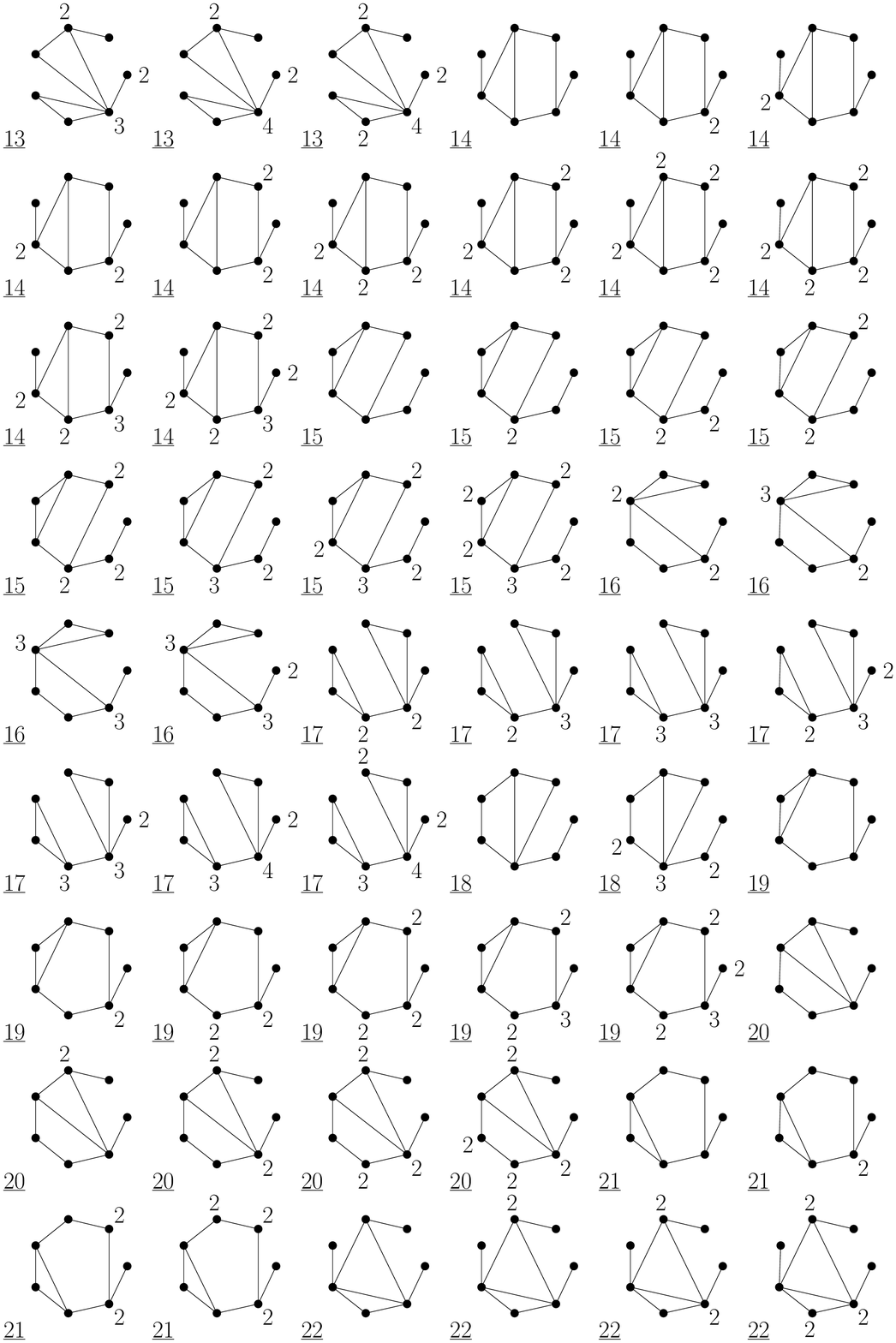}
	\caption{The set $\mathcal{W}(E_7)$ (continued).}
	\end{center}
\end{figure}
\begin{figure}[H]
	\begin{center}
	 	\includegraphics[scale=\scalingconstant]{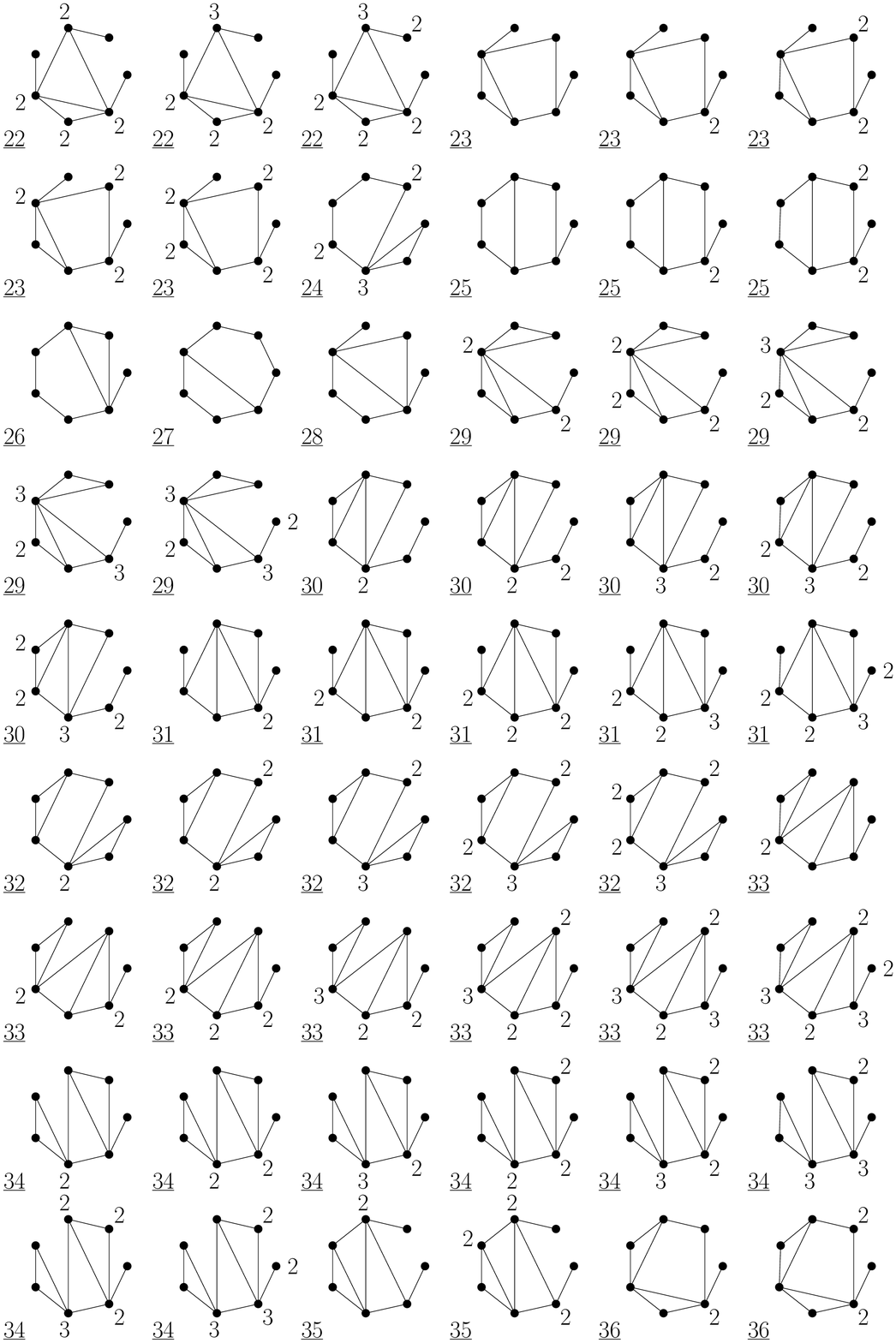}
	\caption{The set $\mathcal{W}(E_7)$ (continued).}
	\end{center}
\end{figure}
\begin{figure}[H]
	\begin{center}
	 	\includegraphics[scale=\scalingconstant]{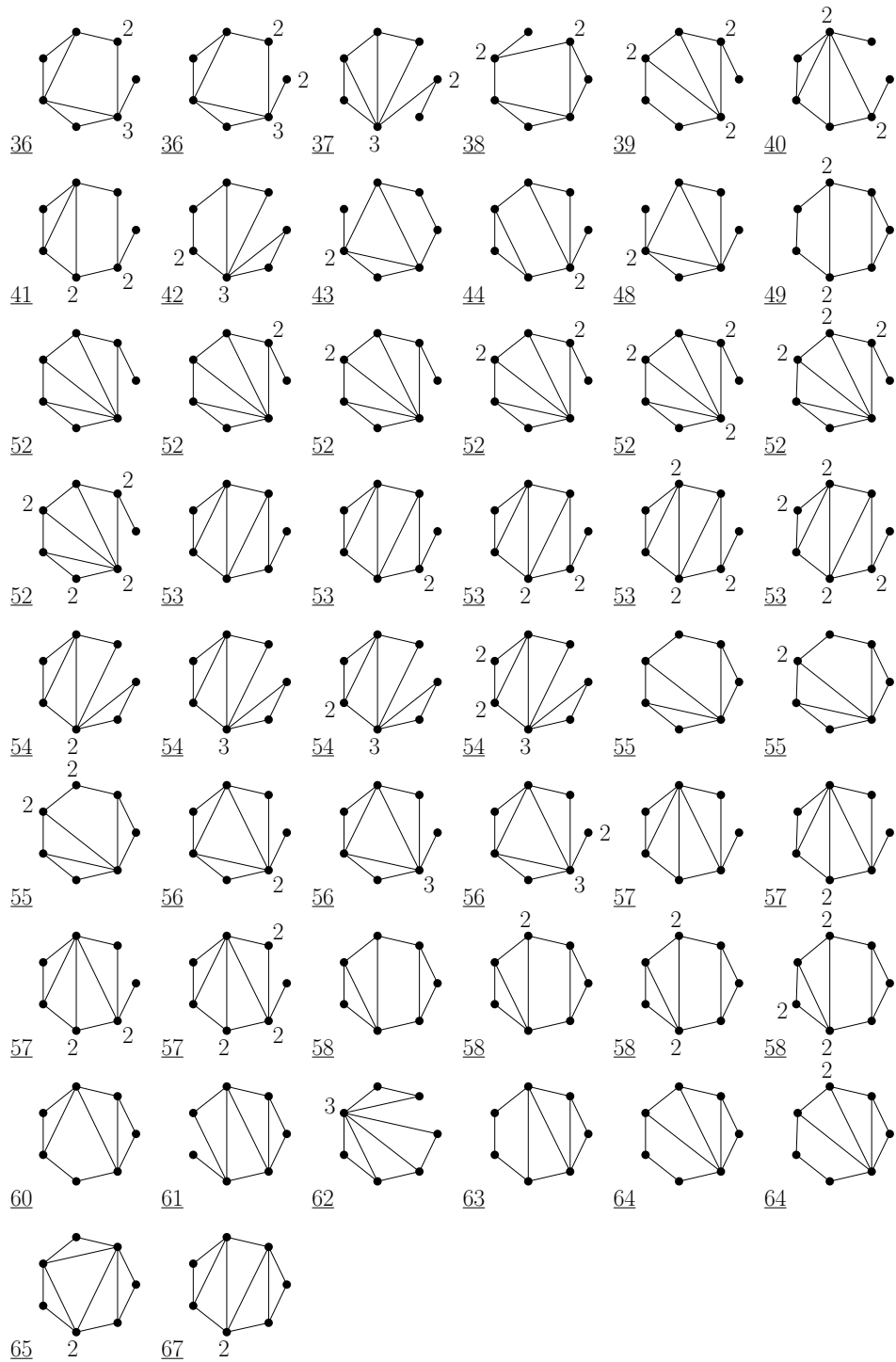}
	\caption{The set $\mathcal{W}(E_7)$ (continued).}
   	\label{fig:allowed_diagrams_E7-last}
	\end{center}
\end{figure}
\newpage

\subsection{Type \texorpdfstring{$E_8$}{E8}}
\begin{figure}[H]
	\begin{center}
		\includegraphics[scale=\scalingconstant]{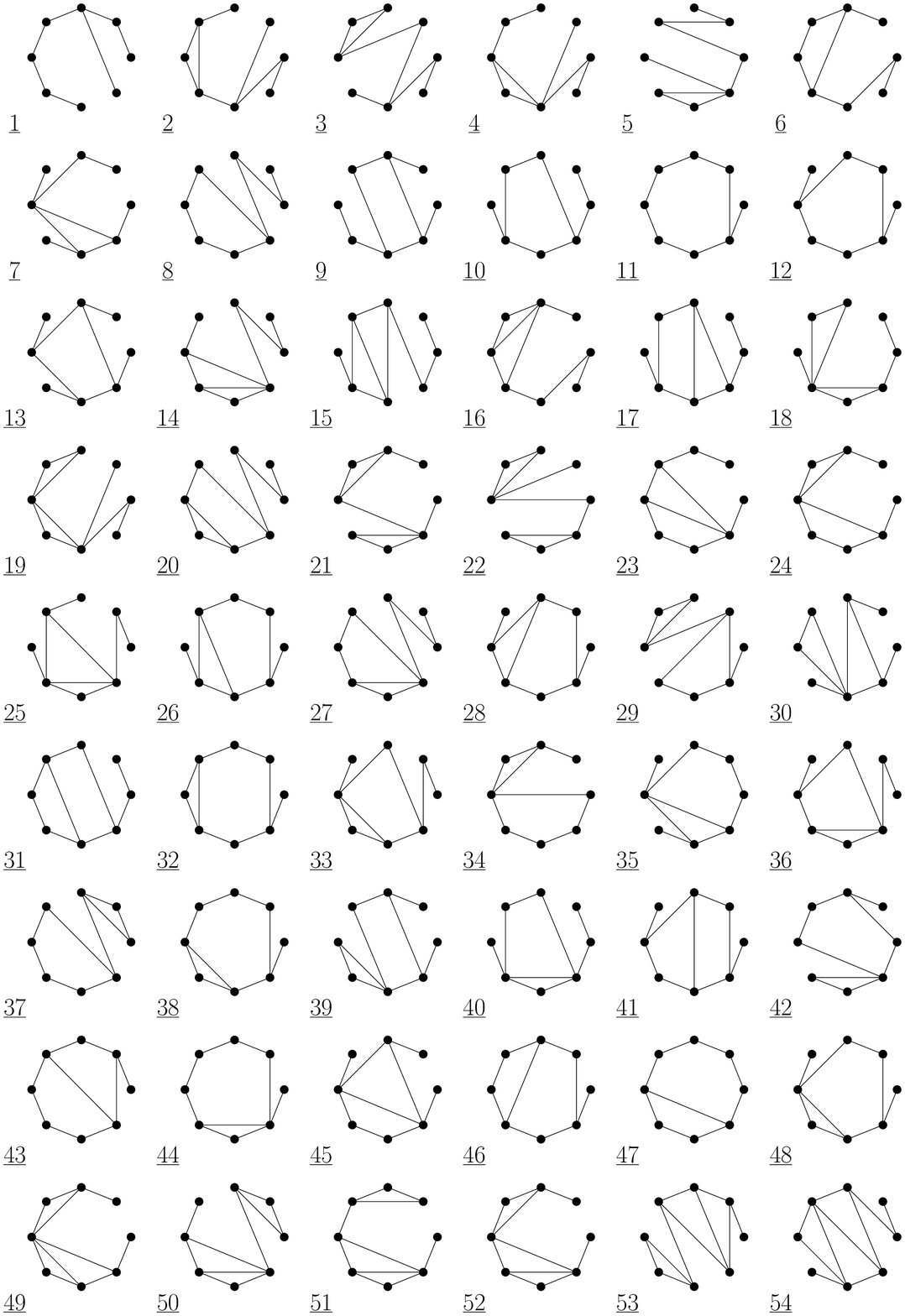}
	\end{center}
	\caption{Diagrams in $\mathcal{X}(E_8)$.}
	\label{fig:dynkin-E8-first}
\end{figure}
\begin{figure}[H]
	\begin{center}
		\includegraphics[scale=\scalingconstant]{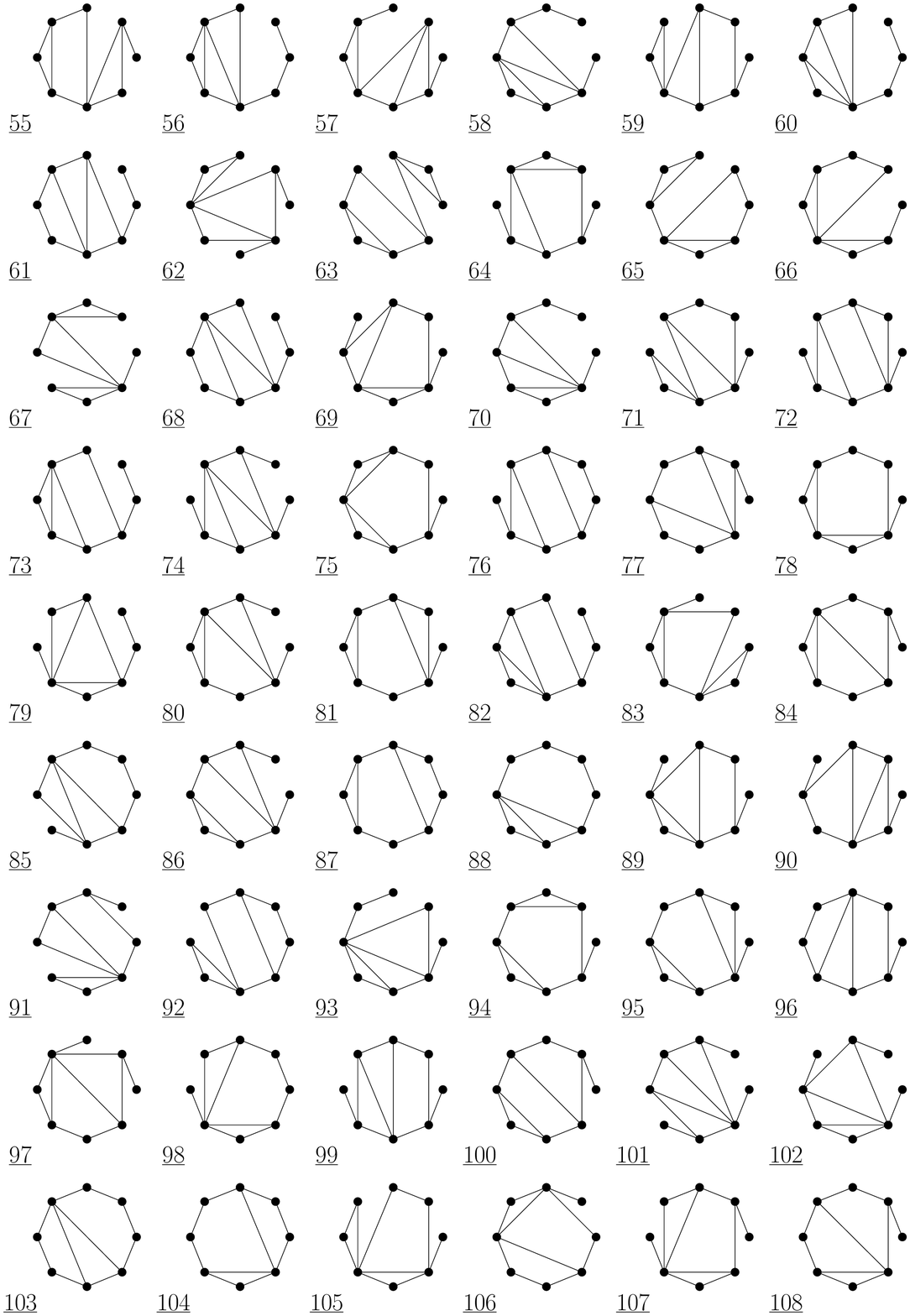}
	\end{center}
	\caption{Diagrams in $\mathcal{X}(E_8)$ (continued).}
\end{figure}
\begin{figure}[H]
	\begin{center}
		\includegraphics[scale=\scalingconstant]{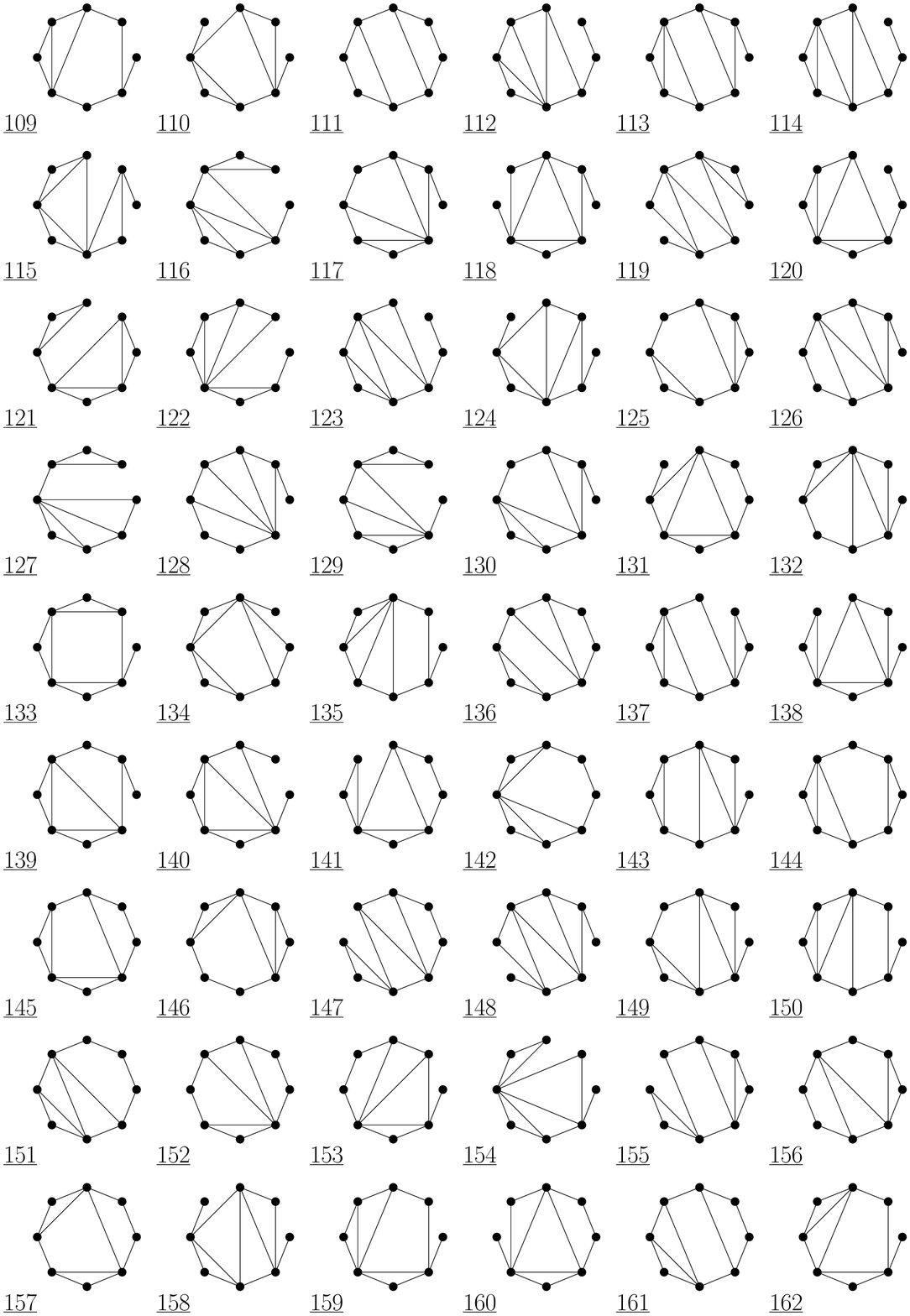}
	\end{center}
	\caption{Diagrams in $\mathcal{X}(E_8)$ (continued).}
\end{figure}
\begin{figure}[H]
	\begin{center}
		\includegraphics[scale=\scalingconstant]{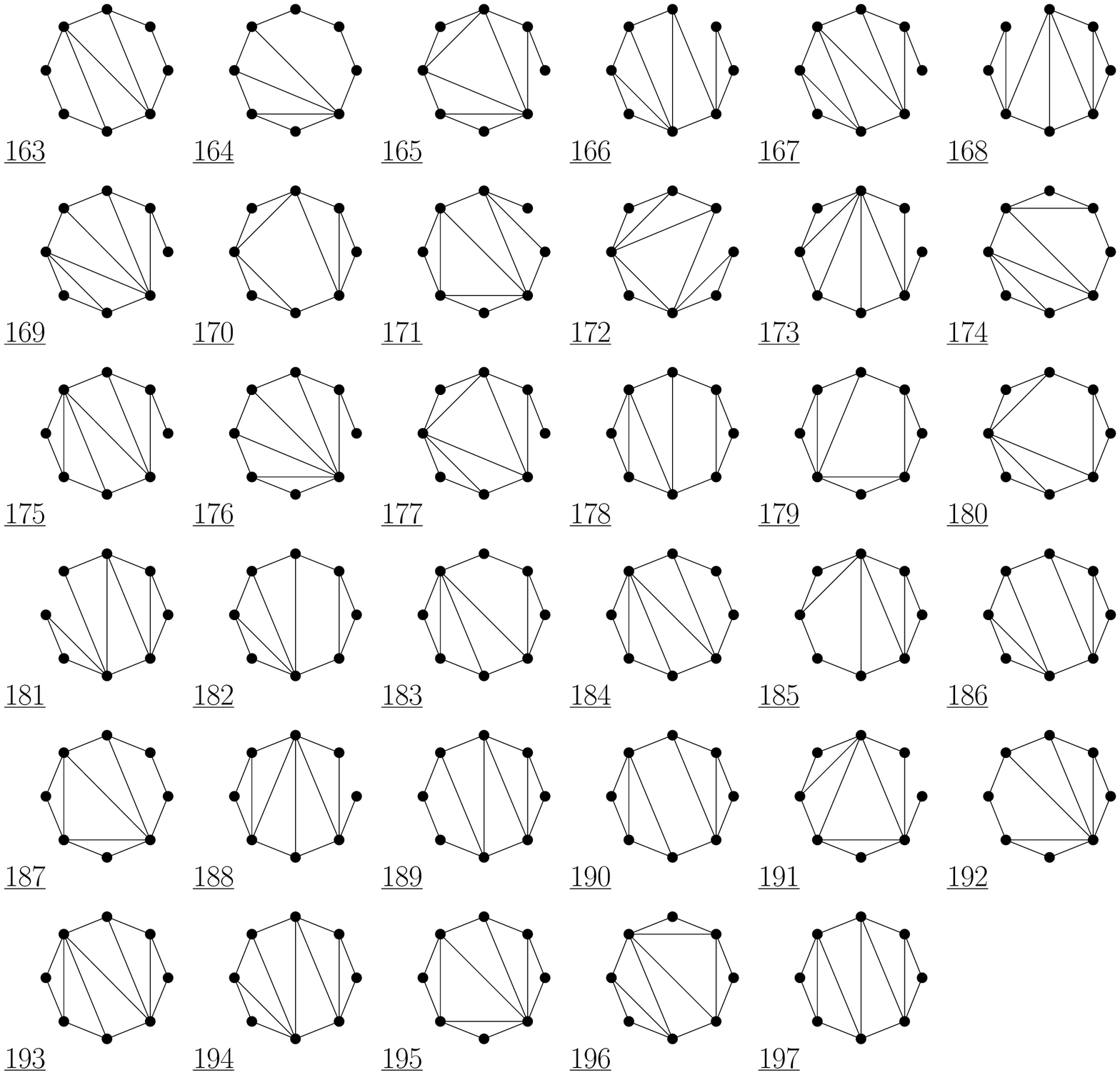}
		\caption{Diagrams in $\mathcal{X}(E_8)$ (continued).}
		\label{fig:dynkin-E8-last}
	 	\includegraphics[scale=\scalingconstant]{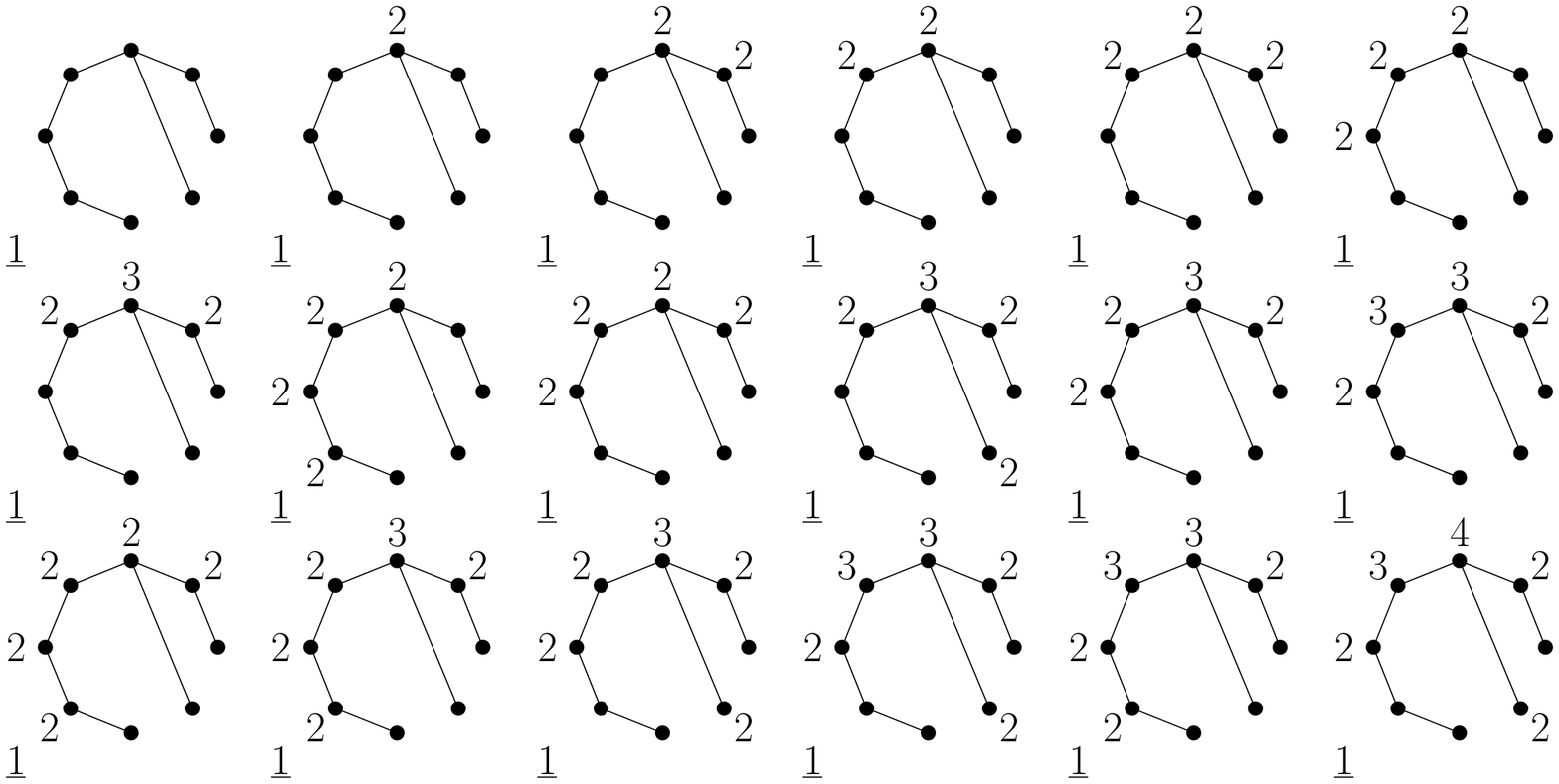}
		\caption{The set $\mathcal{W}(E_8)$ consists of the above weighted diagrams
			and all the elements of $\mathcal{W}(A_7)$, $\mathcal{W}(D_7)$,
			and $\mathcal{W}(E_7)$.}
		\label{fig:allowed_diagrams_E8-first}
	\end{center}
\end{figure}
\begin{figure}[H]
	\begin{center}
		\includegraphics[scale=\scalingconstant]{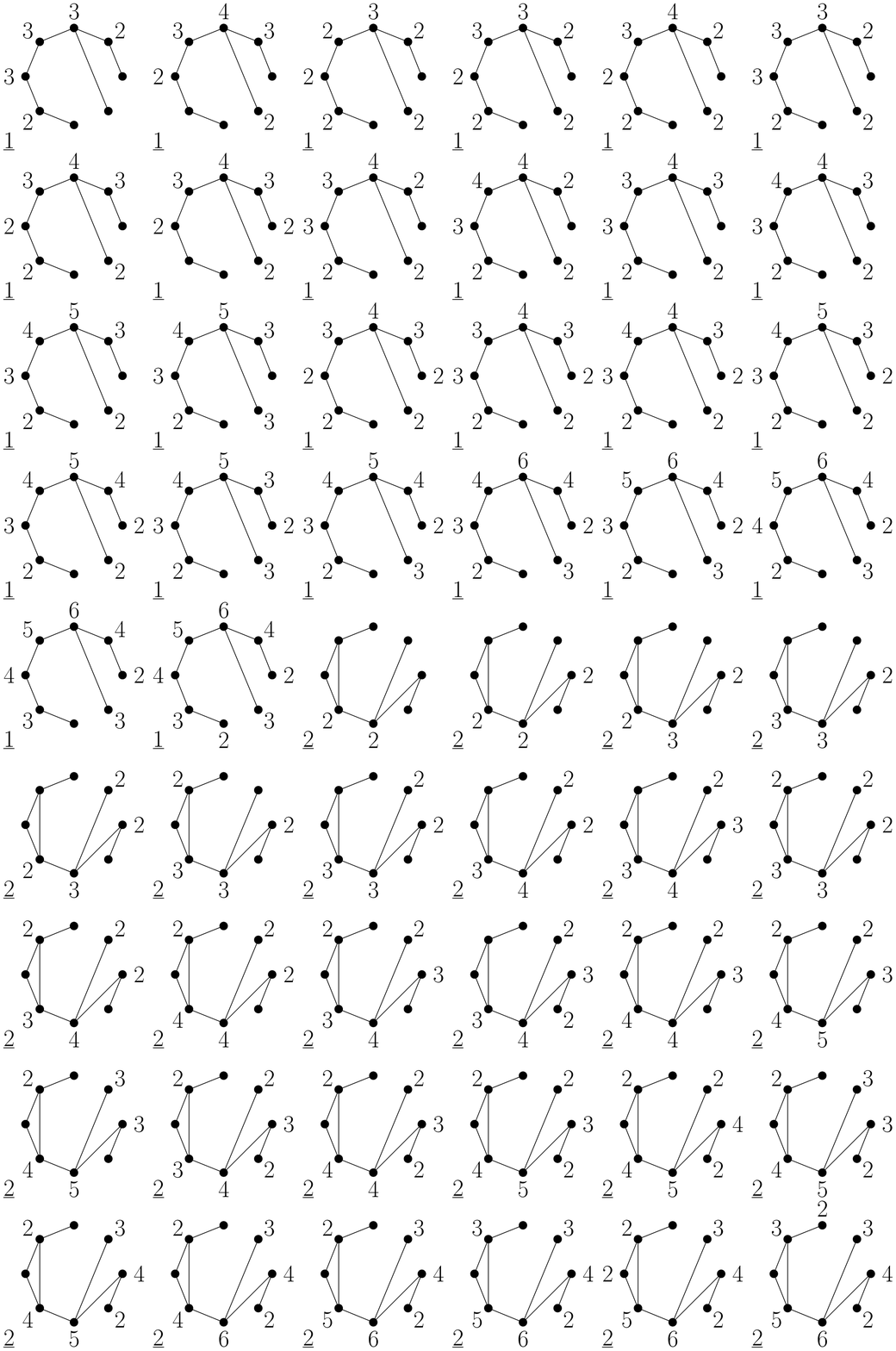}
	\end{center}
	\caption{The set $\mathcal{W}(E_8)$ (continued).}
\end{figure}
\begin{figure}[H]
	\begin{center}
 		\includegraphics[scale=\scalingconstant]{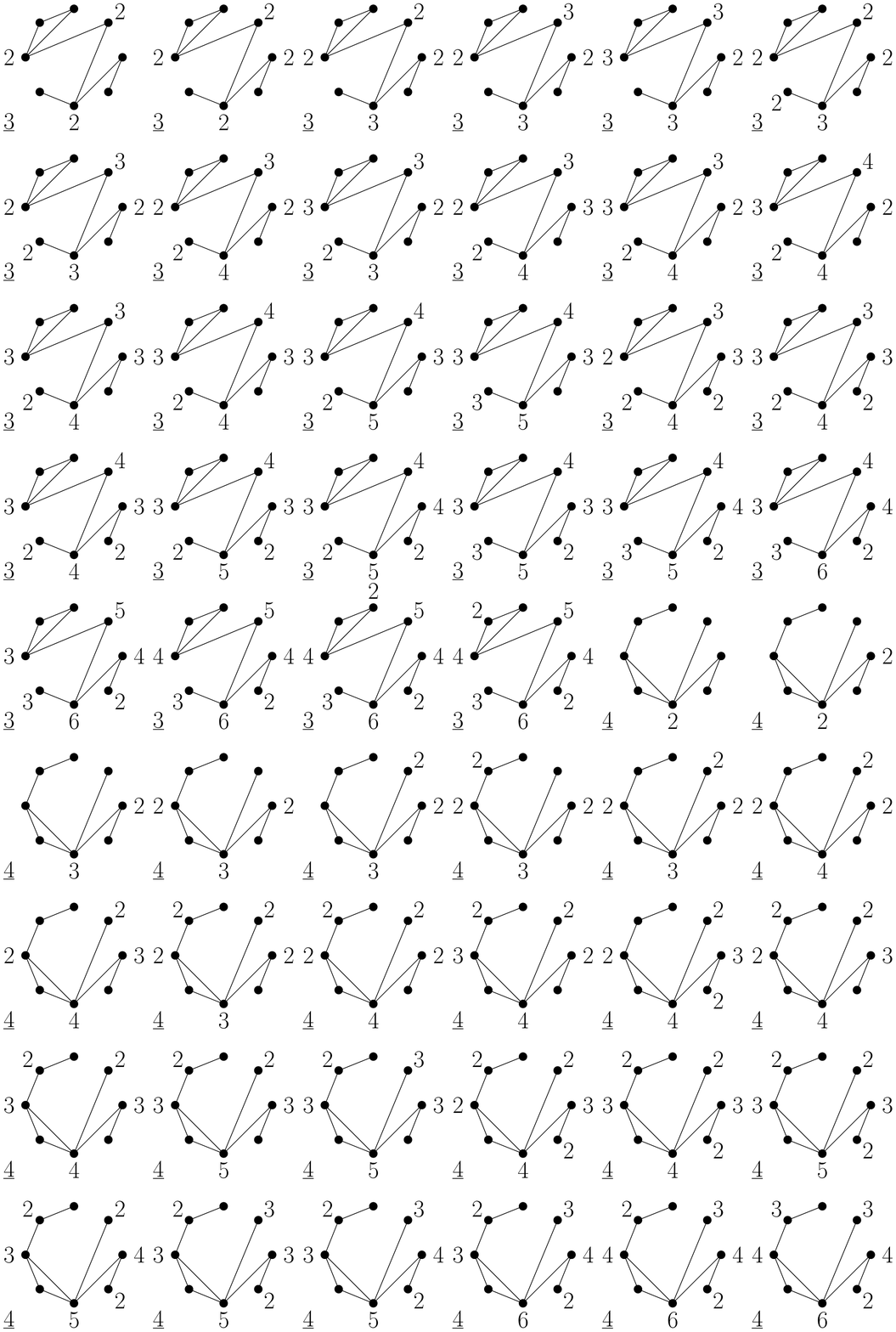}
	\end{center}
	\caption{The set $\mathcal{W}(E_8)$ (continued).}
\end{figure}
\begin{figure}[H]
	\begin{center}
 		\includegraphics[scale=\scalingconstant]{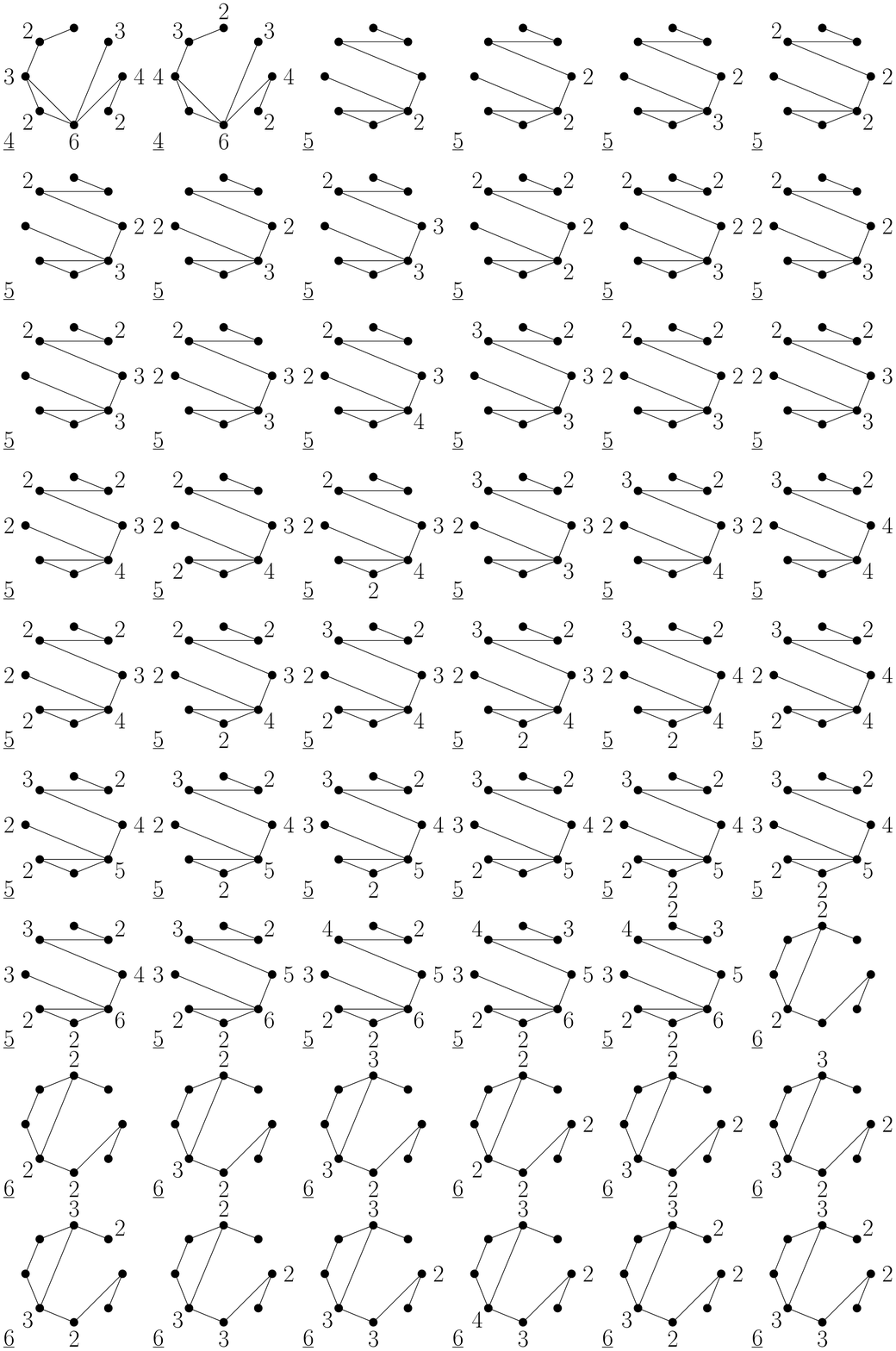}
	\end{center}
	\caption{The set $\mathcal{W}(E_8)$ (continued).}
\end{figure}
\begin{figure}[H]
	\begin{center}
 		\includegraphics[scale=\scalingconstant]{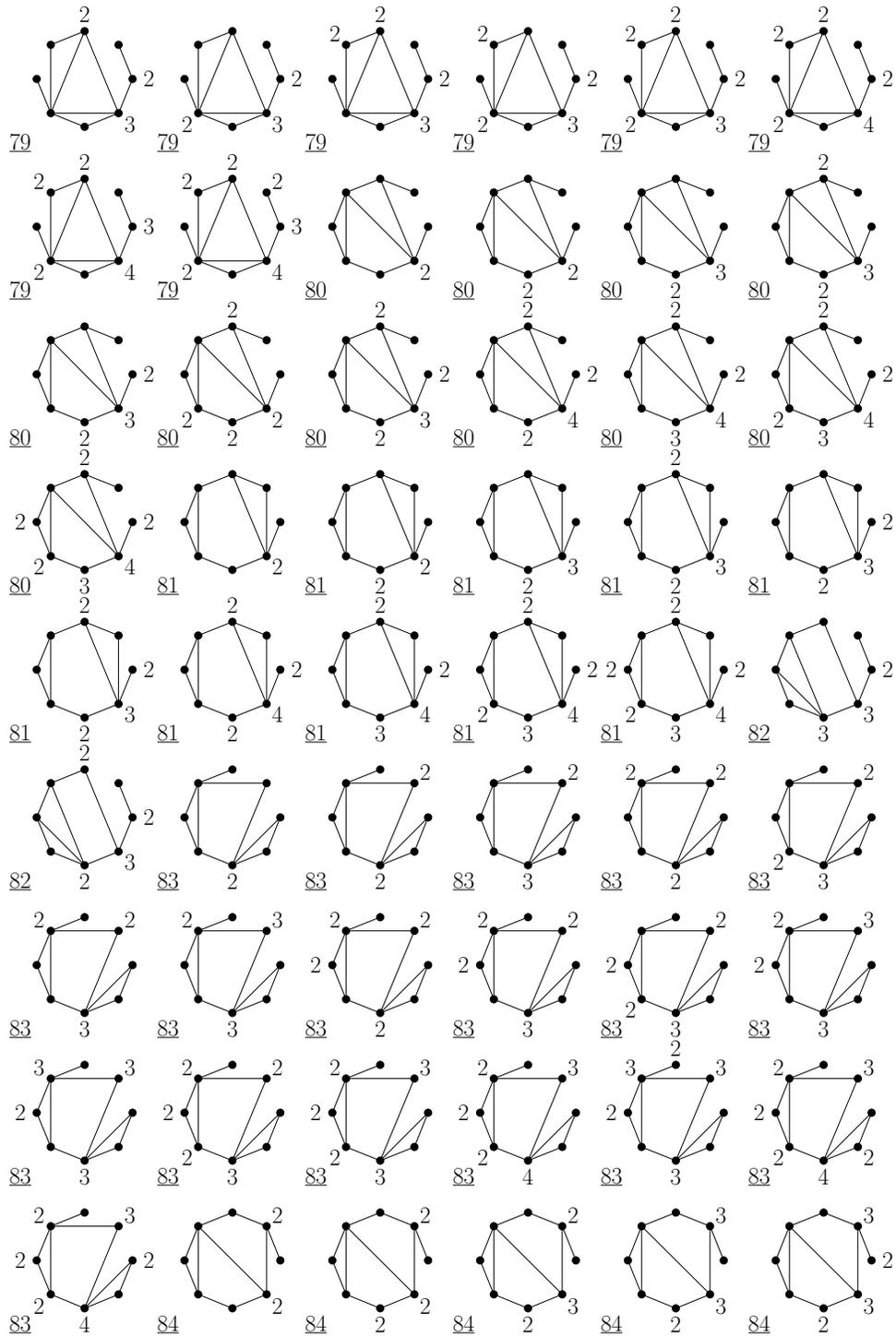}
	\end{center}
	\caption{The set $\mathcal{W}(E_8)$ (continued).}
\end{figure}
\begin{figure}[H]
	\begin{center}
		\includegraphics[scale=\scalingconstant]{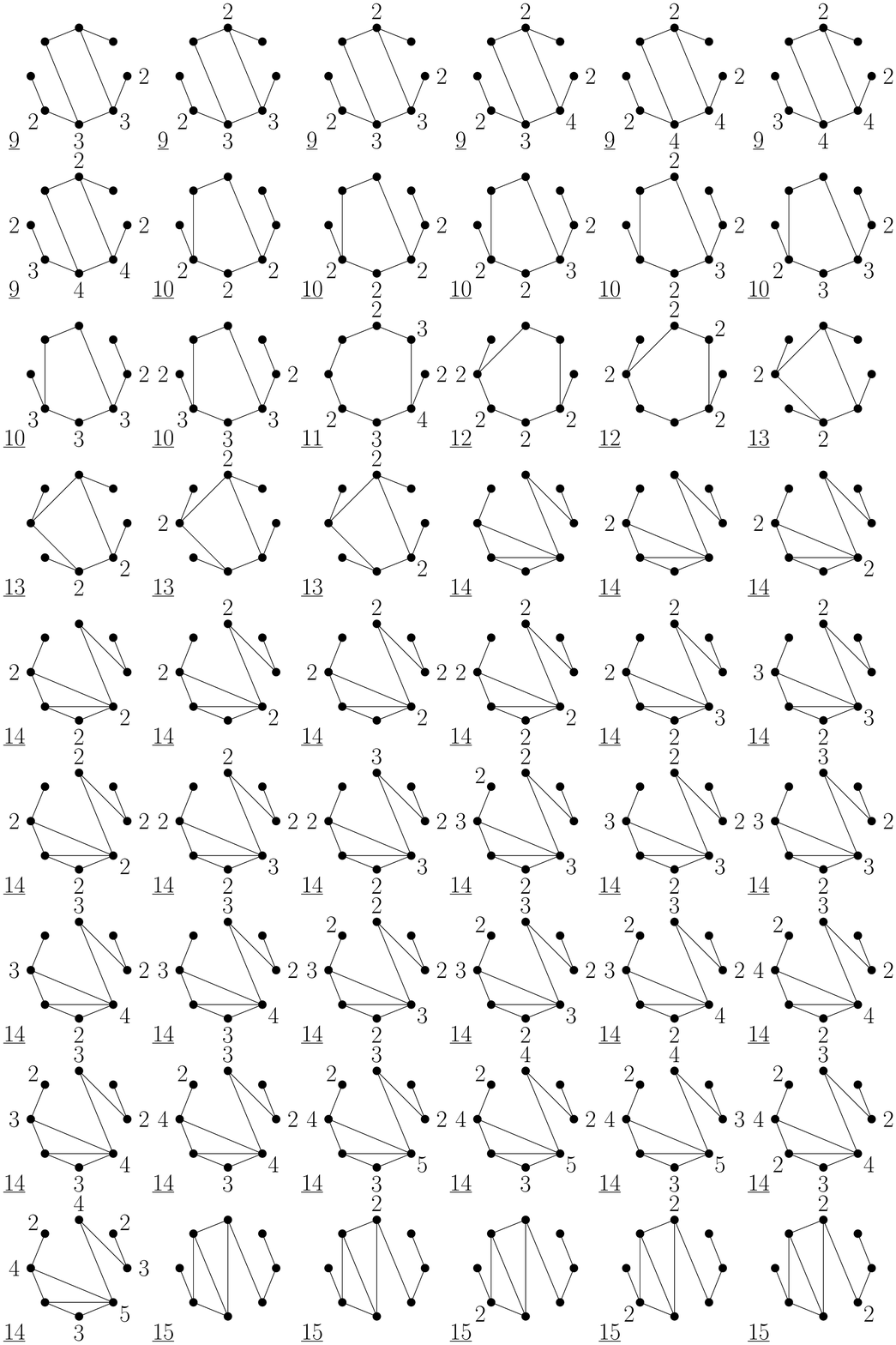}
	\end{center}
	\caption{The set $\mathcal{W}(E_8)$ (continued).}
\end{figure}
\begin{figure}[H]
	\begin{center}
		\includegraphics[scale=\scalingconstant]{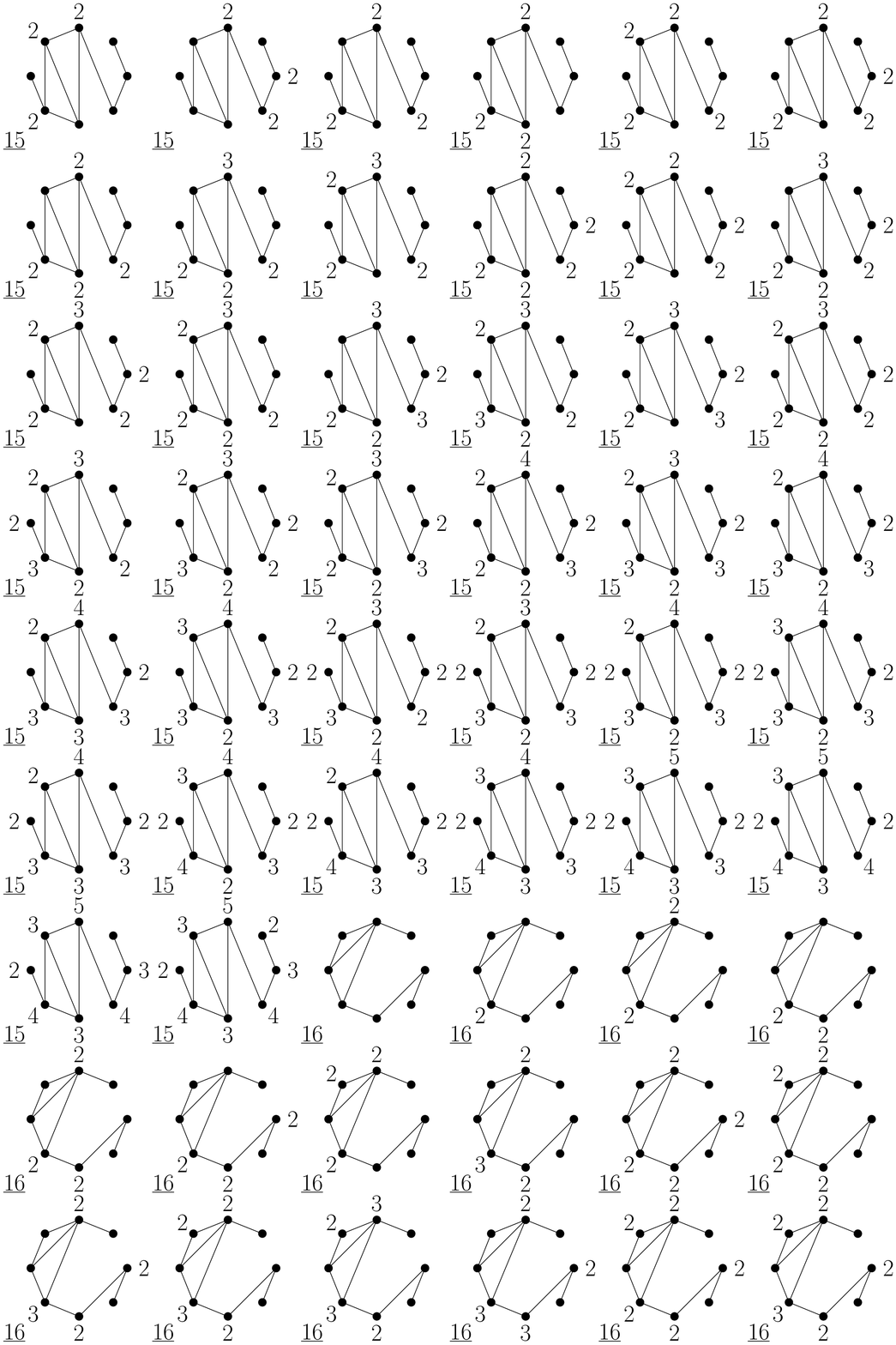}
	\end{center}
	\caption{The set $\mathcal{W}(E_8)$ (continued).}
\end{figure}
\begin{figure}[H]
	\begin{center}
		\includegraphics[scale=\scalingconstant]{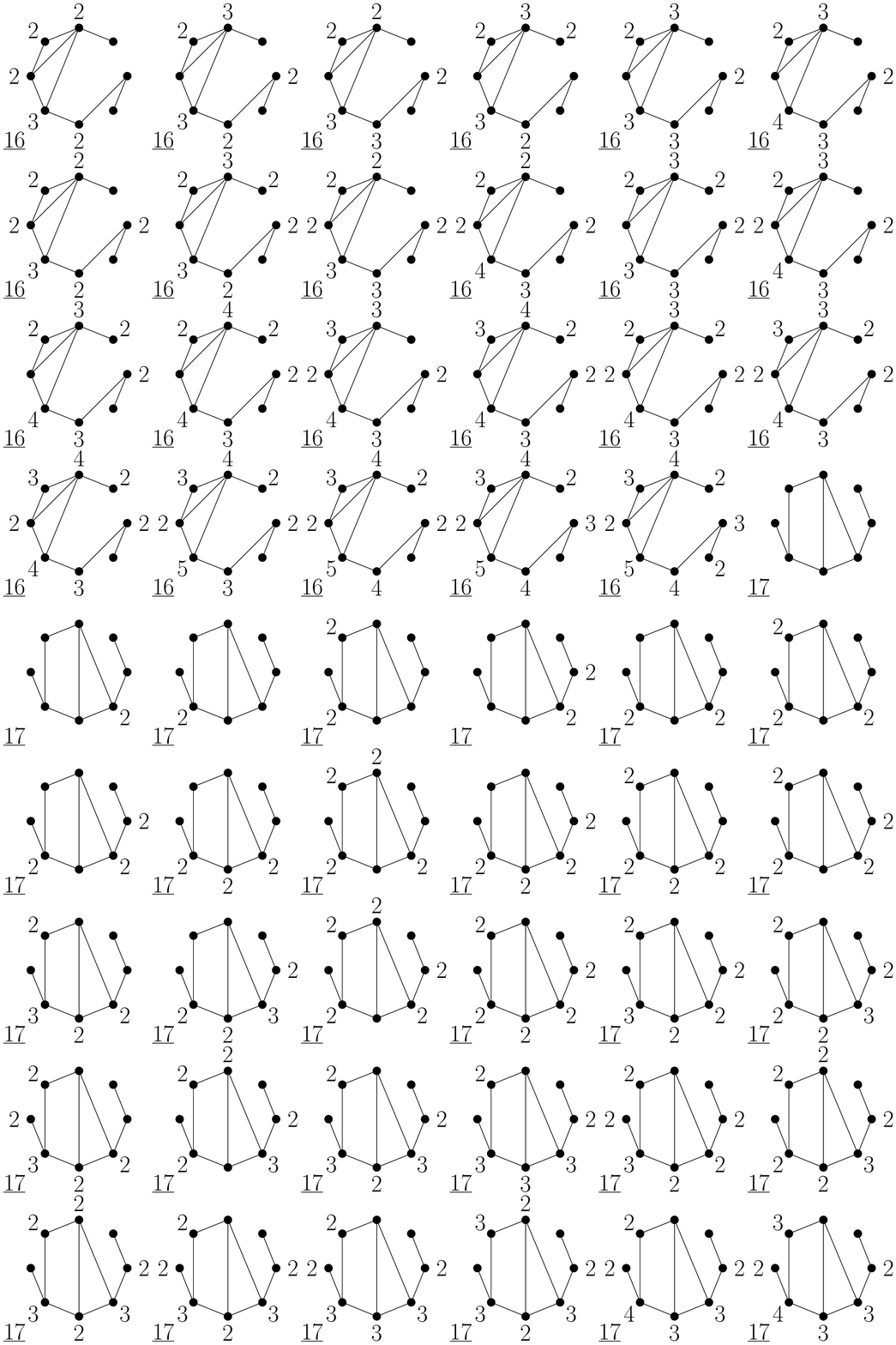}
	\end{center}
	\caption{The set $\mathcal{W}(E_8)$ (continued).}
\end{figure}
\begin{figure}[H]
	\begin{center}
		\includegraphics[scale=\scalingconstant]{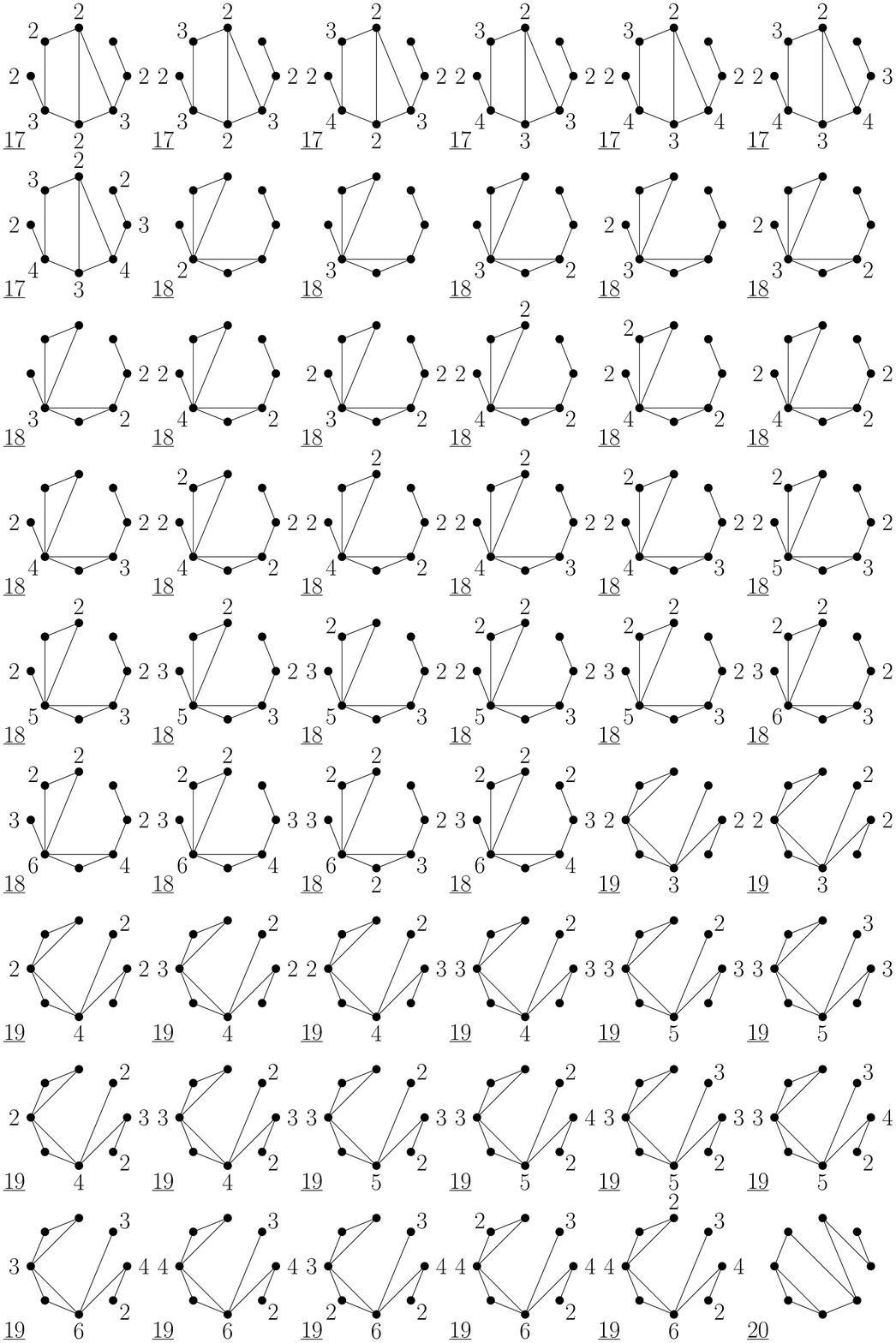}
	\end{center}
	\caption{The set $\mathcal{W}(E_8)$ (continued).}
\end{figure}
\begin{figure}[H]
	\begin{center}
		\includegraphics[scale=\scalingconstant]{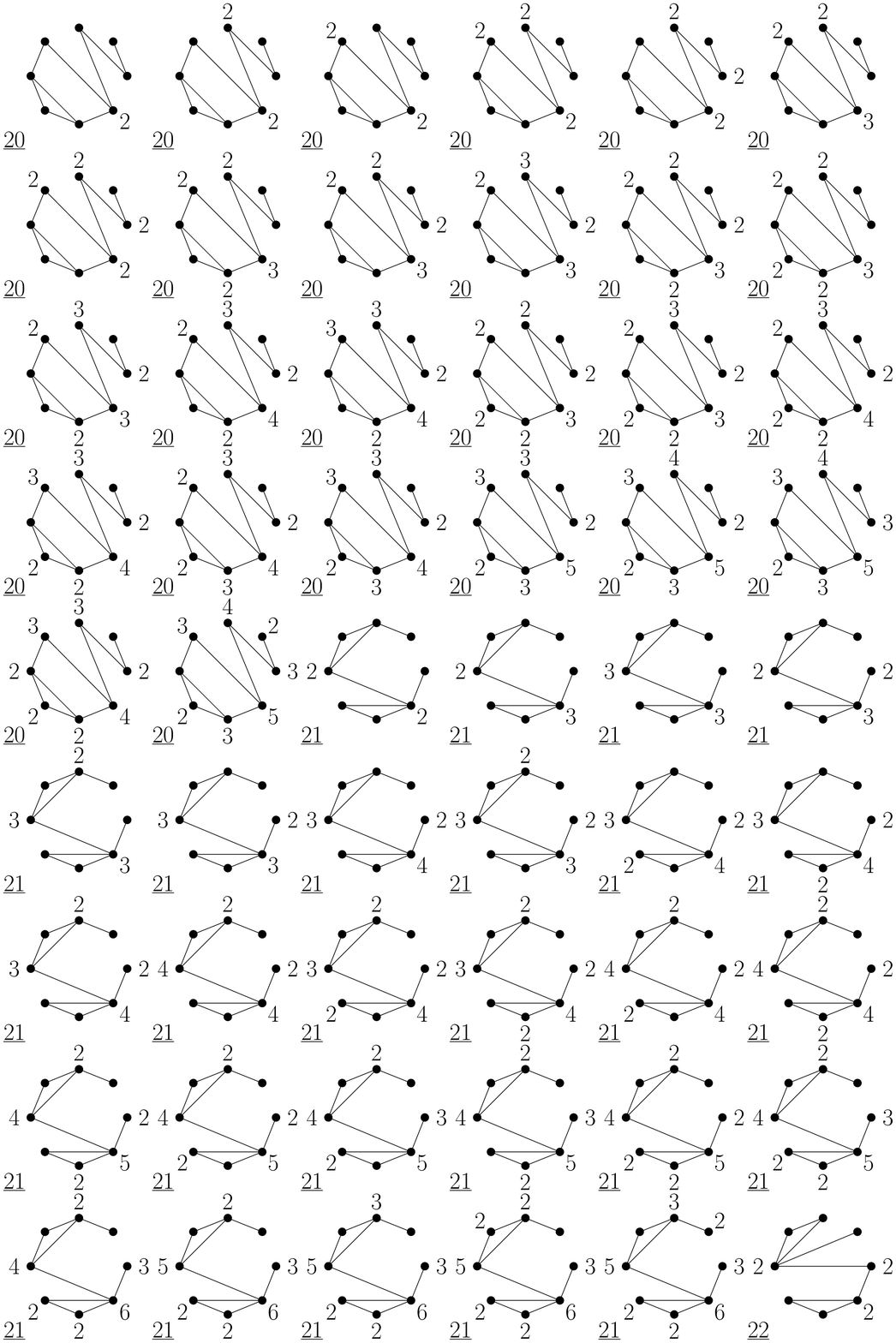}
	\end{center}
	\caption{The set $\mathcal{W}(E_8)$ (continued).}
\end{figure}
\begin{figure}[H]
	\begin{center}
		\includegraphics[scale=\scalingconstant]{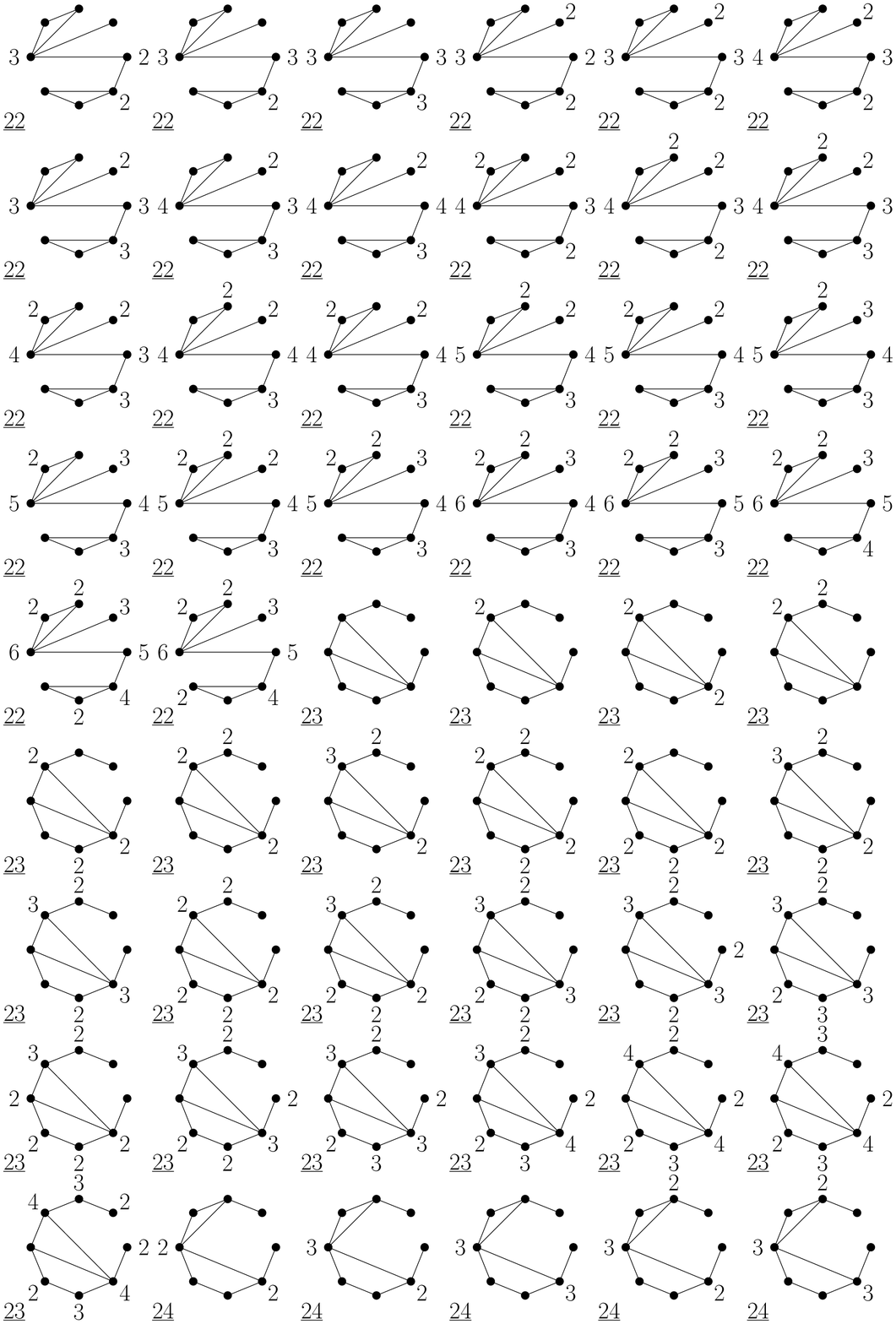}
	\end{center}
	\caption{The set $\mathcal{W}(E_8)$ (continued).}
\end{figure}
\begin{figure}[H]
	\begin{center}
		\includegraphics[scale=\scalingconstant]{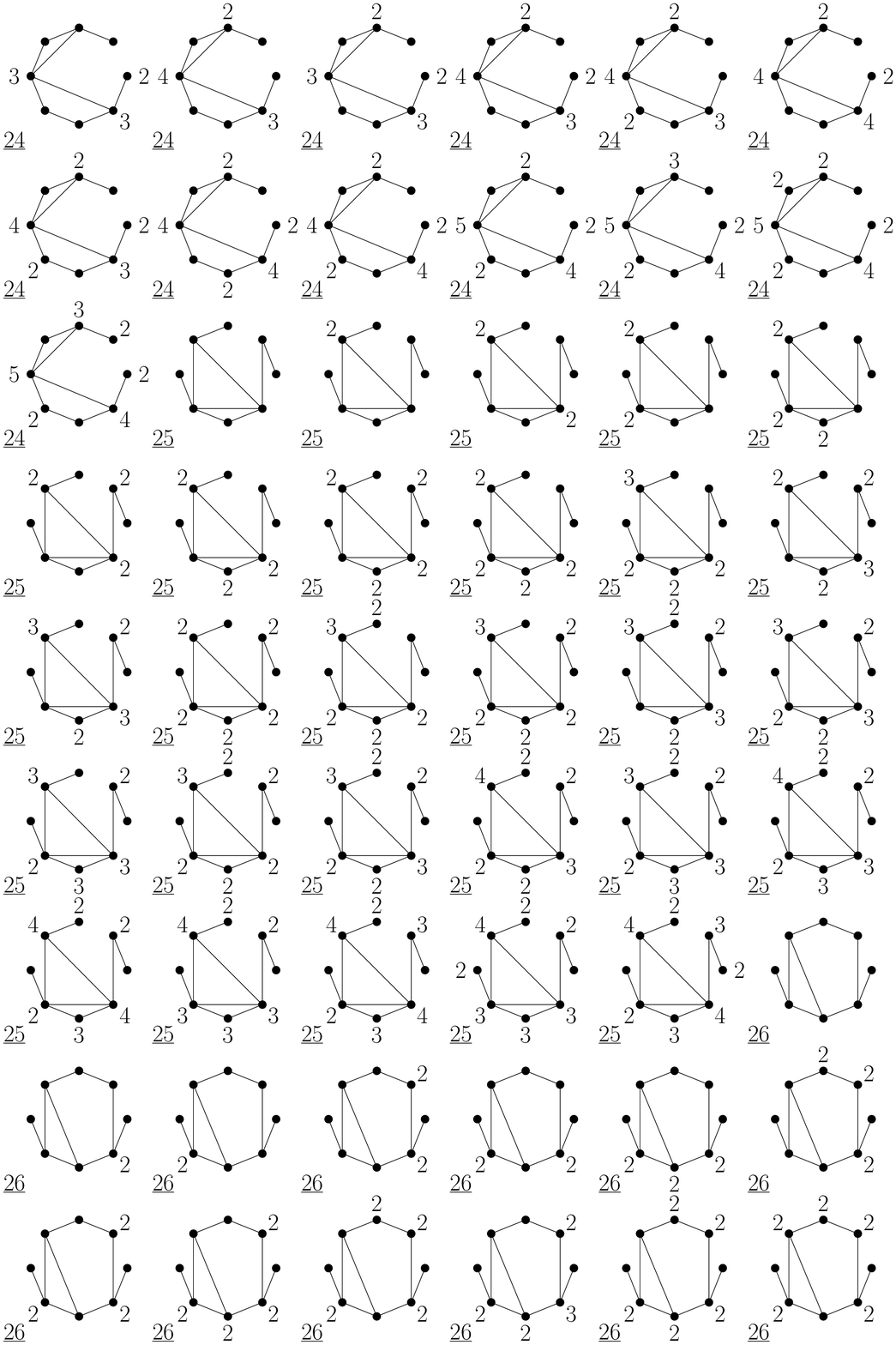}
	\end{center}
	\caption{The set $\mathcal{W}(E_8)$ (continued).}
\end{figure}
\begin{figure}[H]
	\begin{center}
		\includegraphics[scale=\scalingconstant]{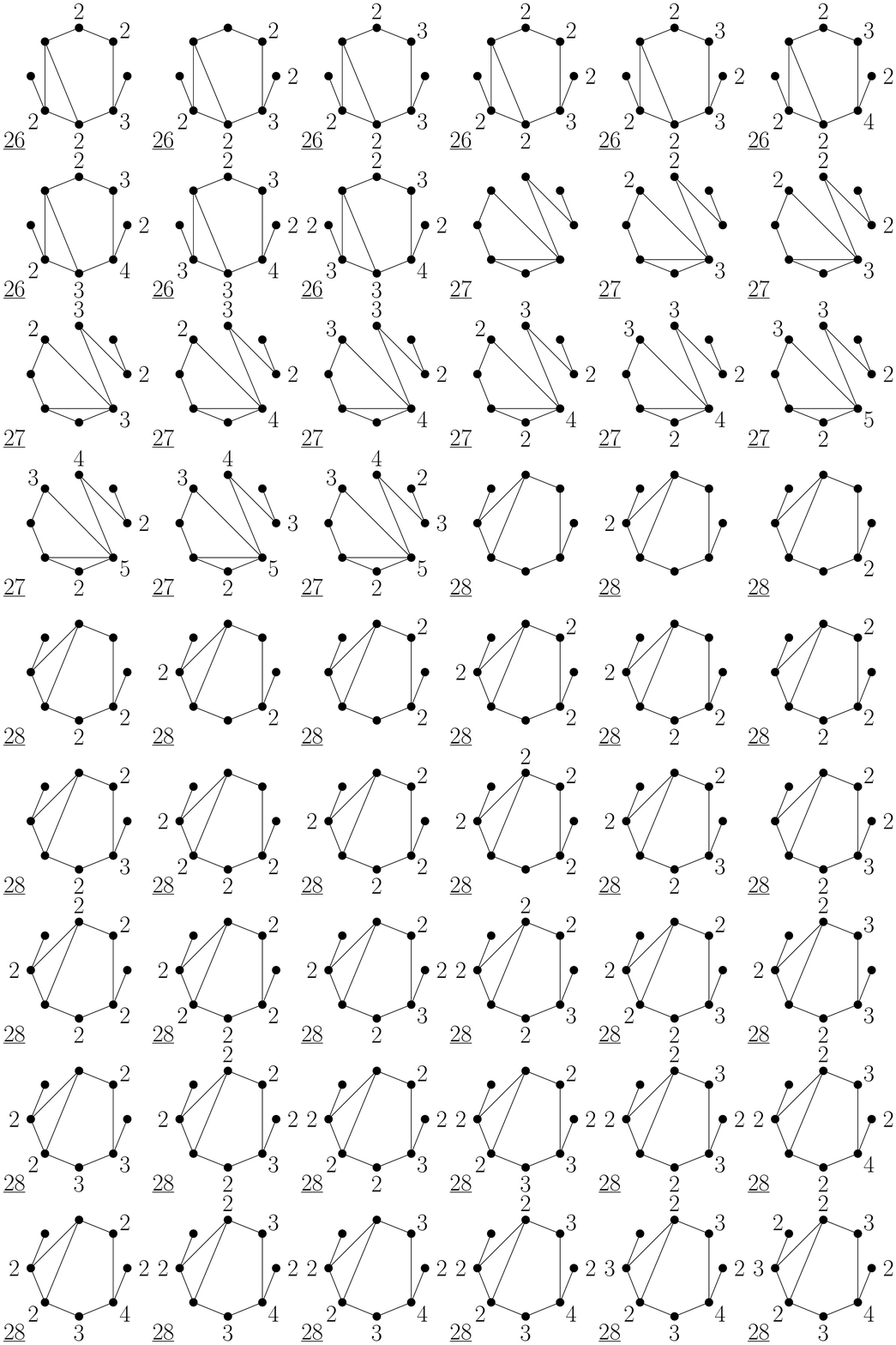}
	\end{center}
	\caption{The set $\mathcal{W}(E_8)$ (continued).}
\end{figure}
\begin{figure}[H]
	\begin{center}
		\includegraphics[scale=\scalingconstant]{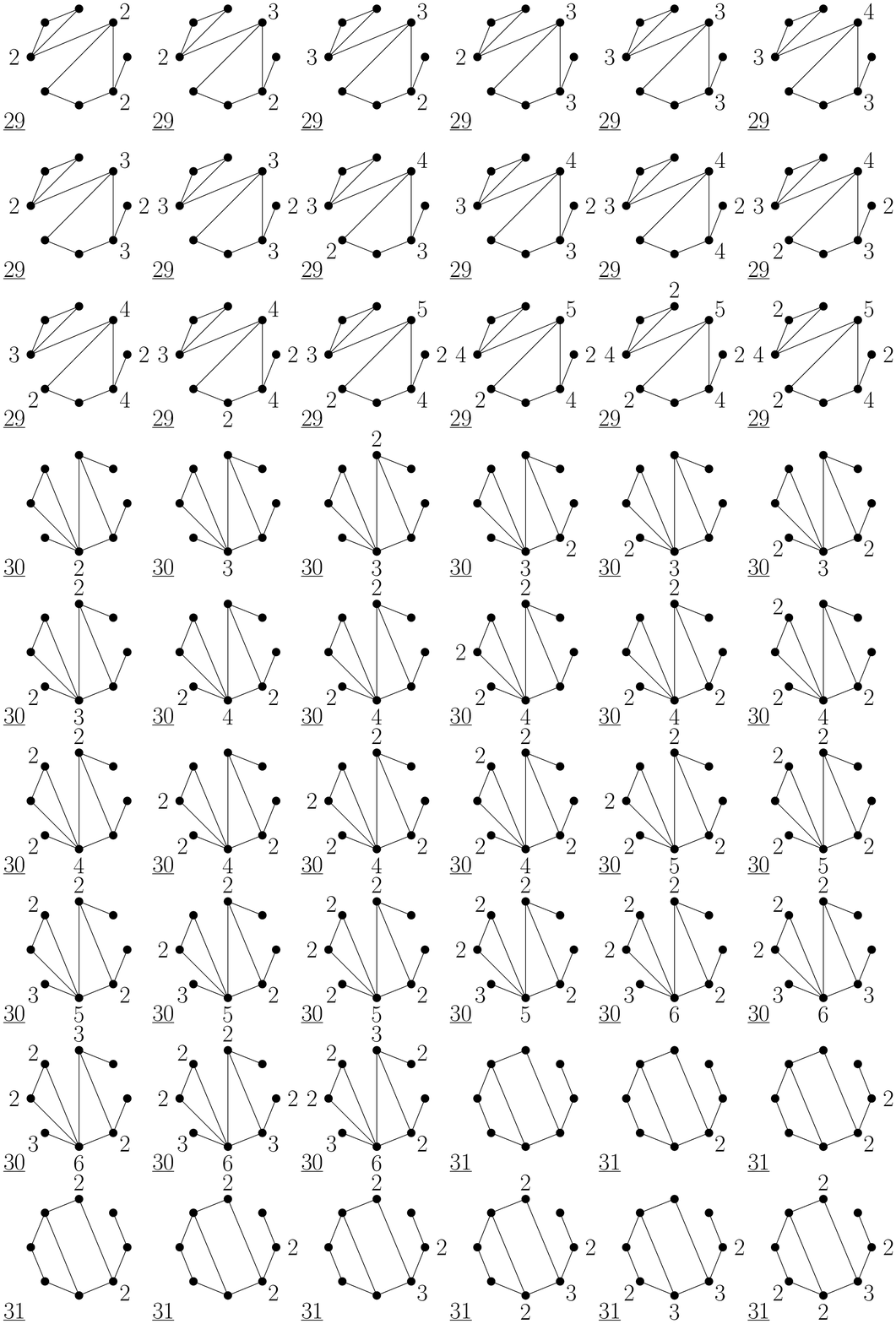}
	\end{center}
	\caption{The set $\mathcal{W}(E_8)$ (continued).}
\end{figure}
\begin{figure}[H]
	\begin{center}
		\includegraphics[scale=\scalingconstant]{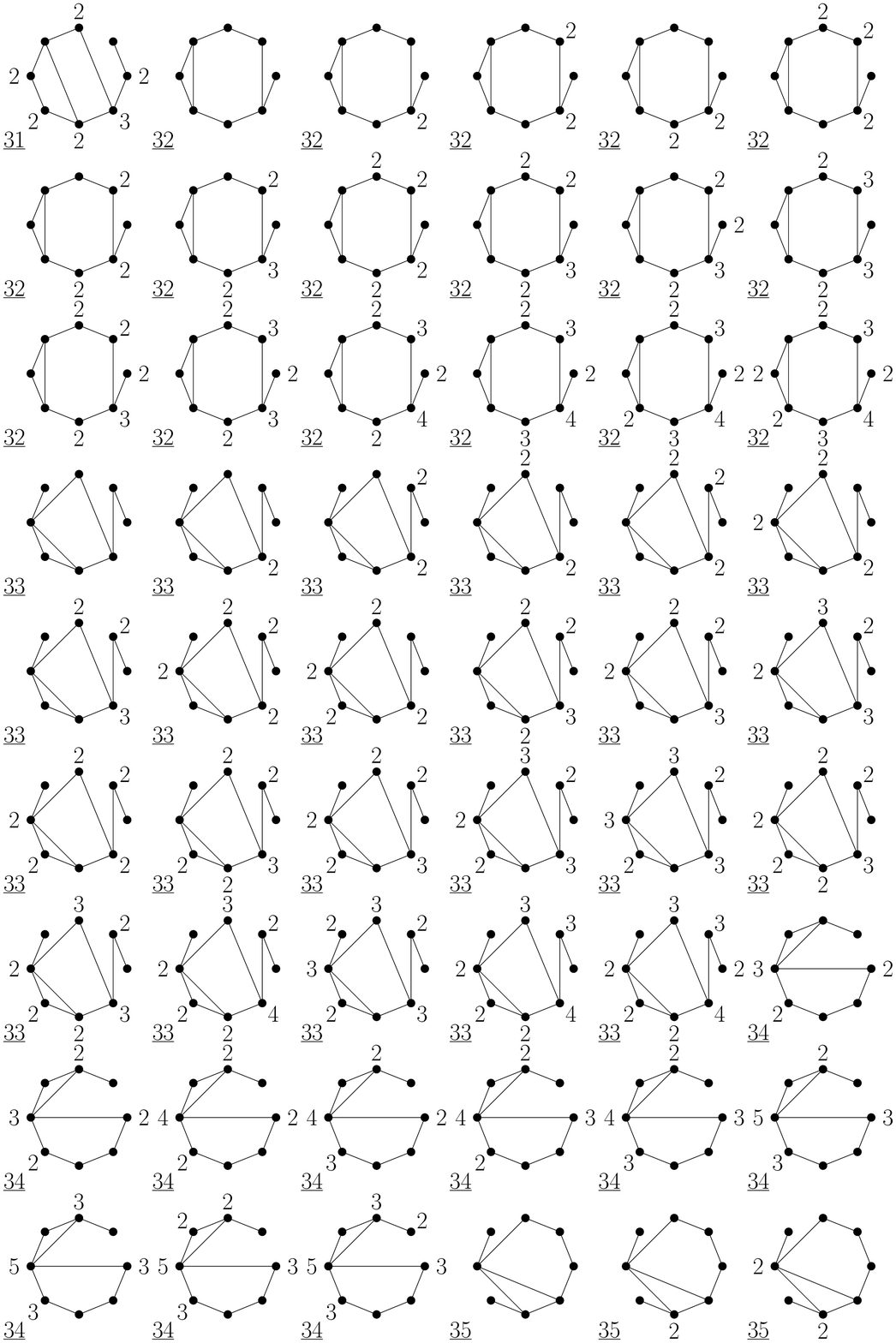}
	\end{center}
	\caption{The set $\mathcal{W}(E_8)$ (continued).}
\end{figure}
\begin{figure}[H]
	\begin{center}
		\includegraphics[scale=\scalingconstant]{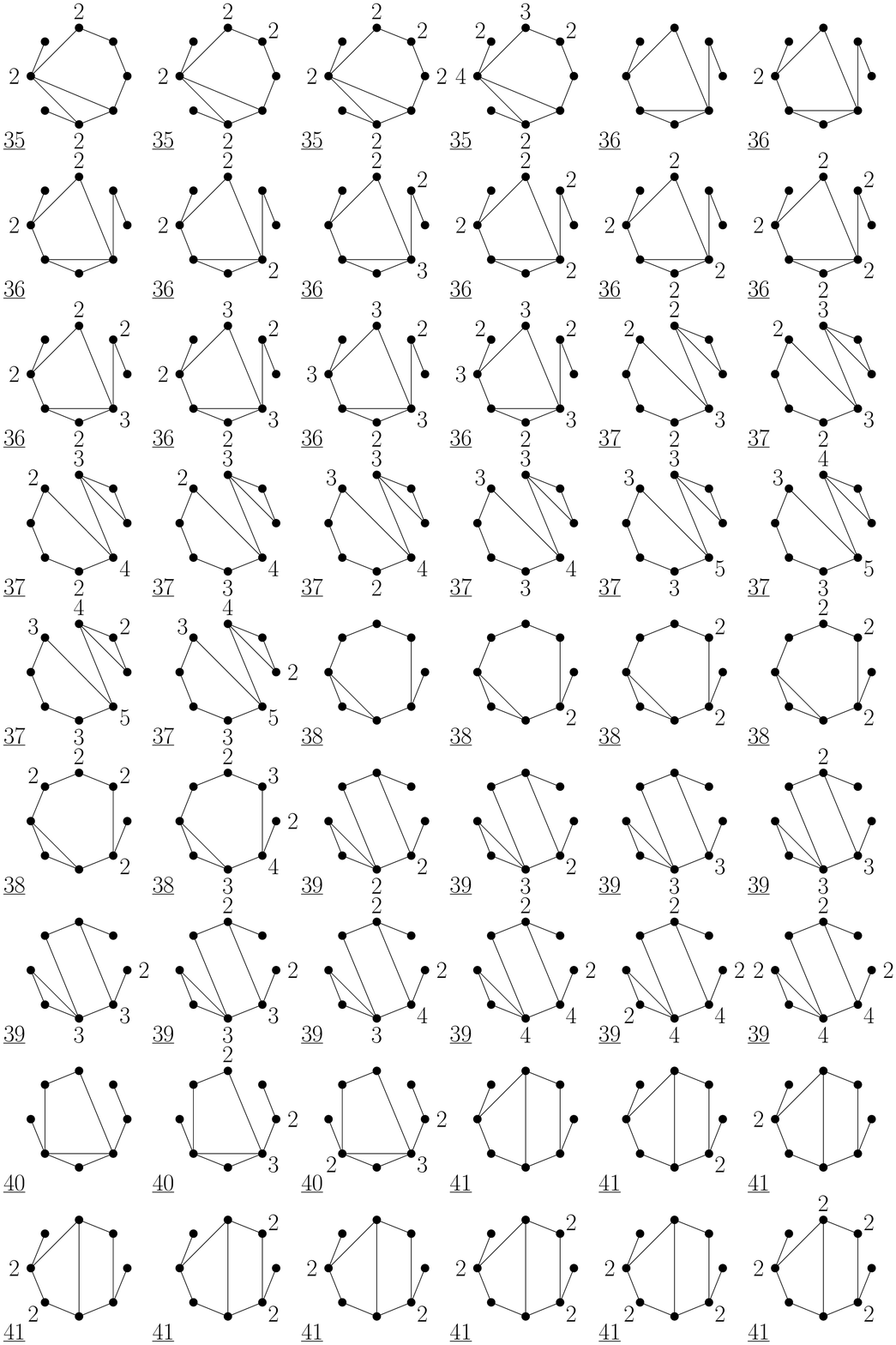}
	\end{center}
	\caption{The set $\mathcal{W}(E_8)$ (continued).}
\end{figure}
\begin{figure}[H]
	\begin{center}
		\includegraphics[scale=\scalingconstant]{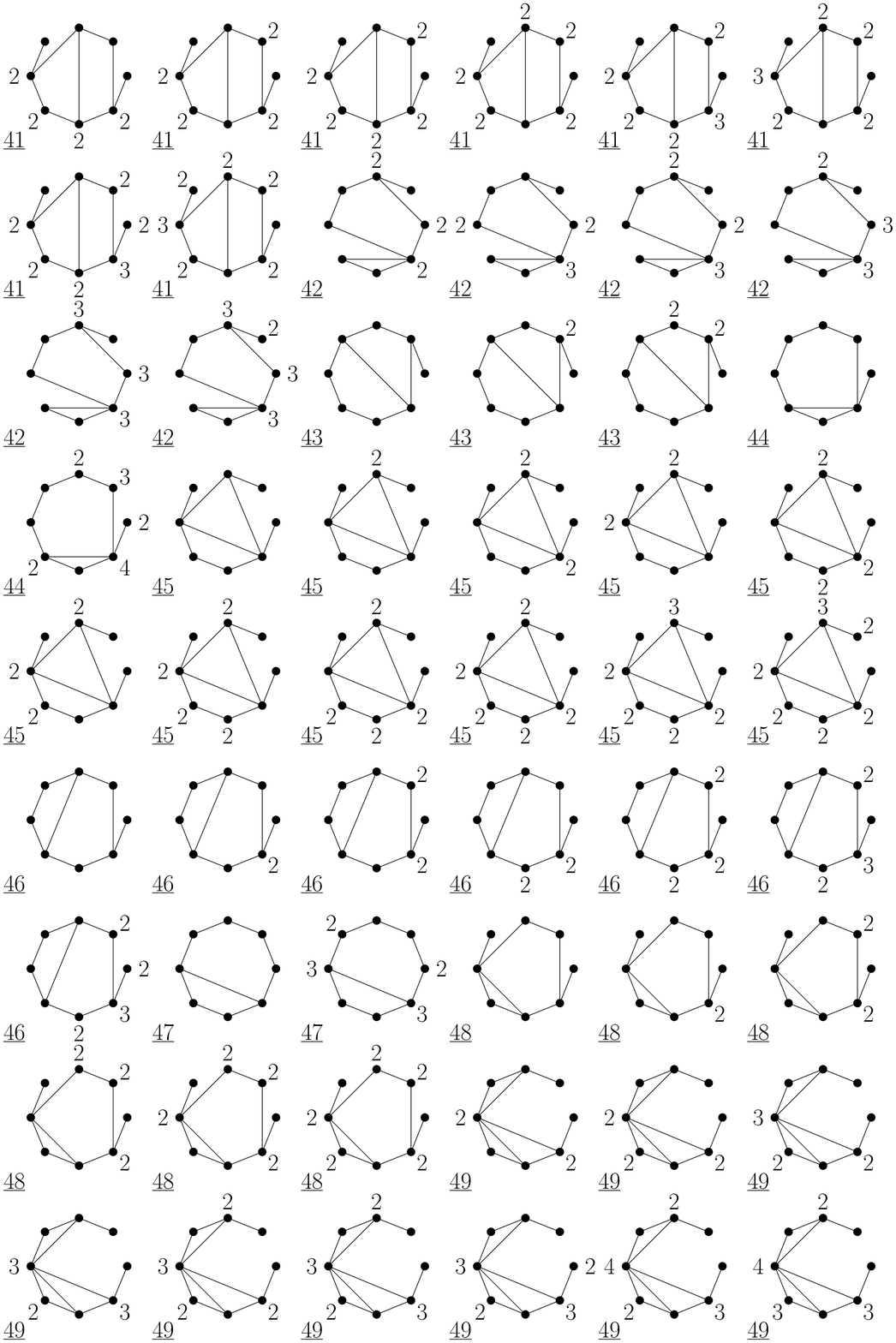}
	\end{center}
	\caption{The set $\mathcal{W}(E_8)$ (continued).}
\end{figure}
\begin{figure}[H]
	\begin{center}
		\includegraphics[scale=\scalingconstant]{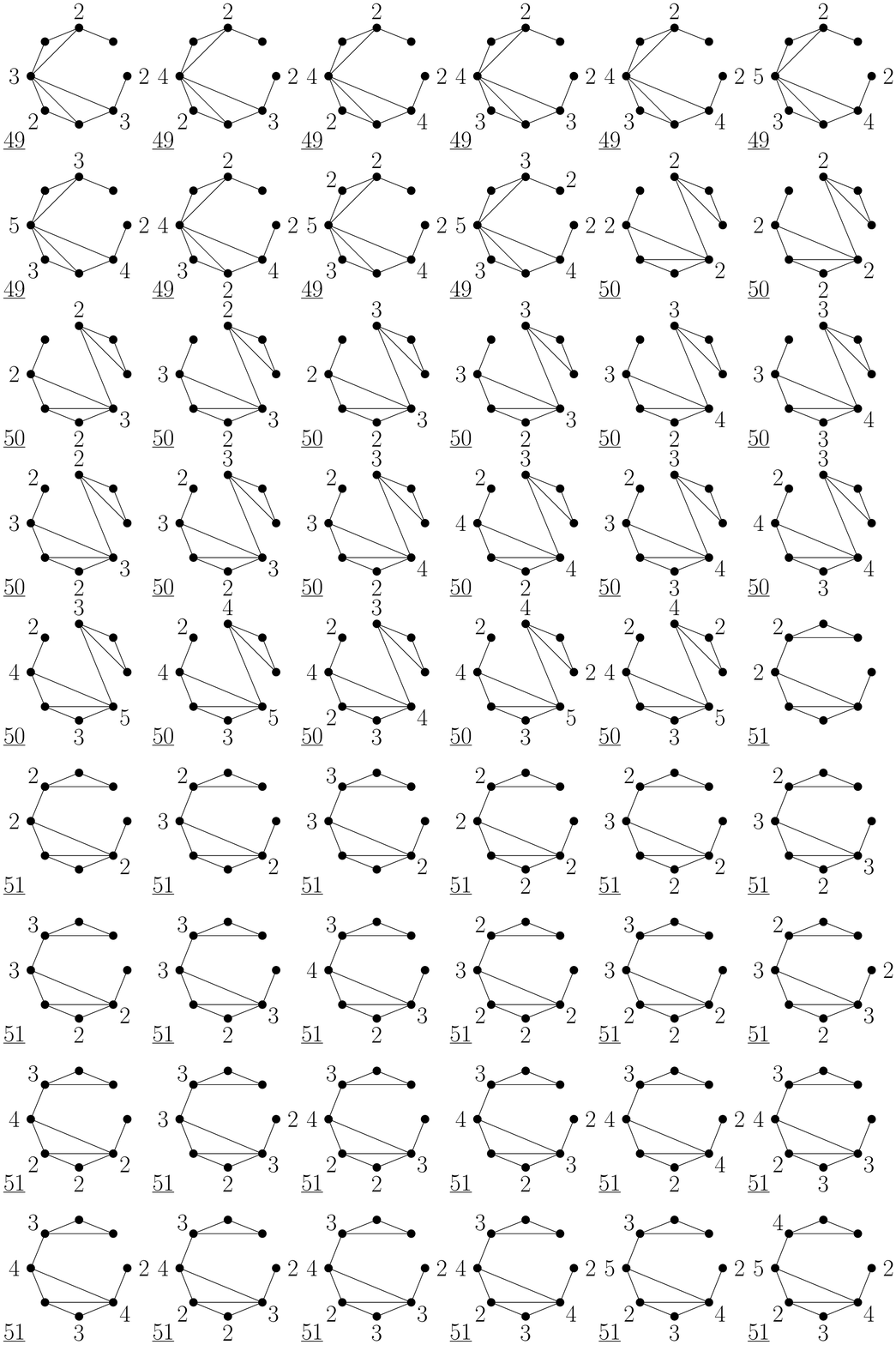}
	\end{center}
	\caption{The set $\mathcal{W}(E_8)$ (continued).}
\end{figure}
\begin{figure}[H]
	\begin{center}
		\includegraphics[scale=\scalingconstant]{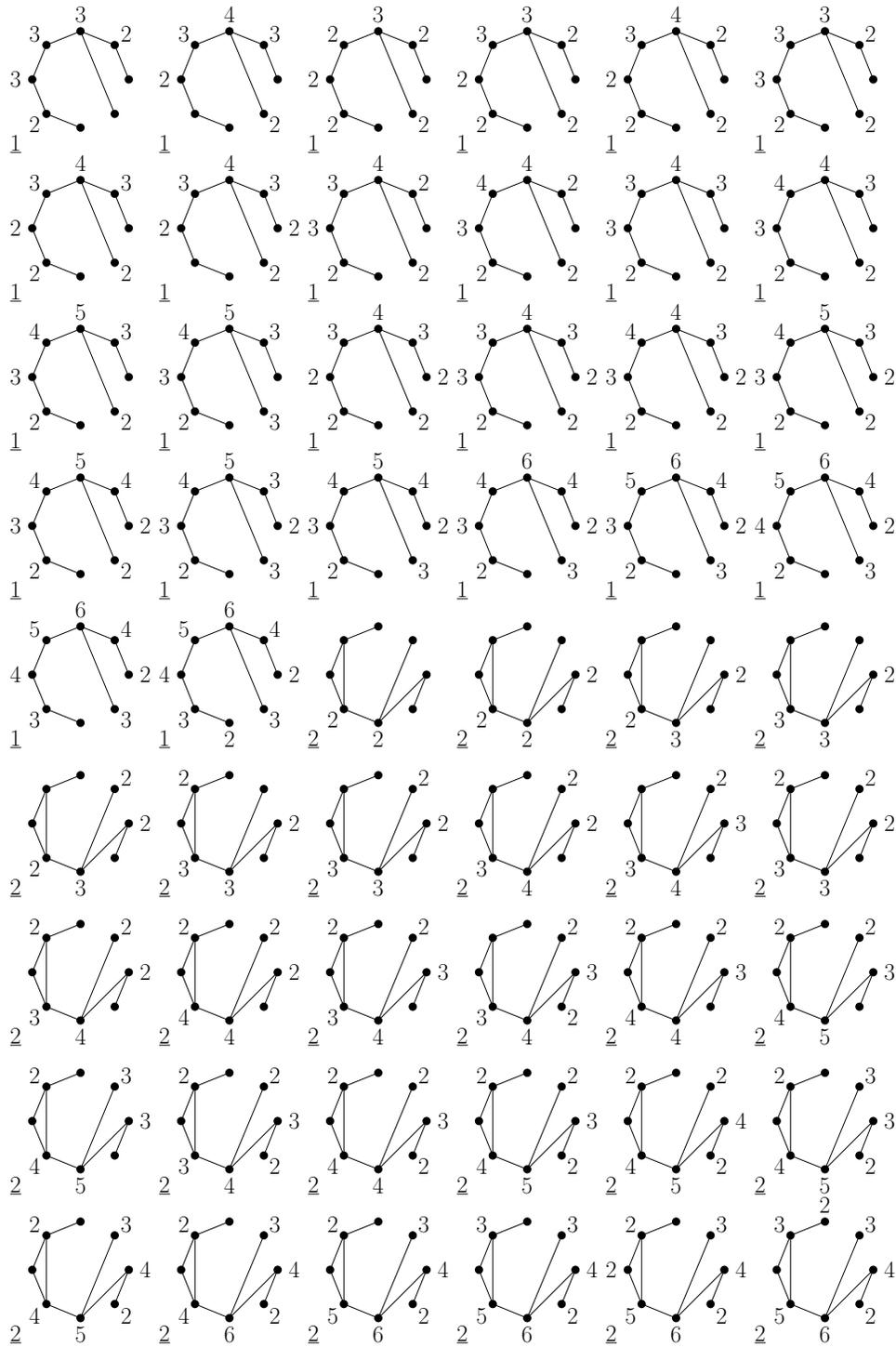}
	\end{center}
	\caption{The set $\mathcal{W}(E_8)$ (continued).}
\end{figure}
\begin{figure}[H]
	\begin{center}
		\includegraphics[scale=\scalingconstant]{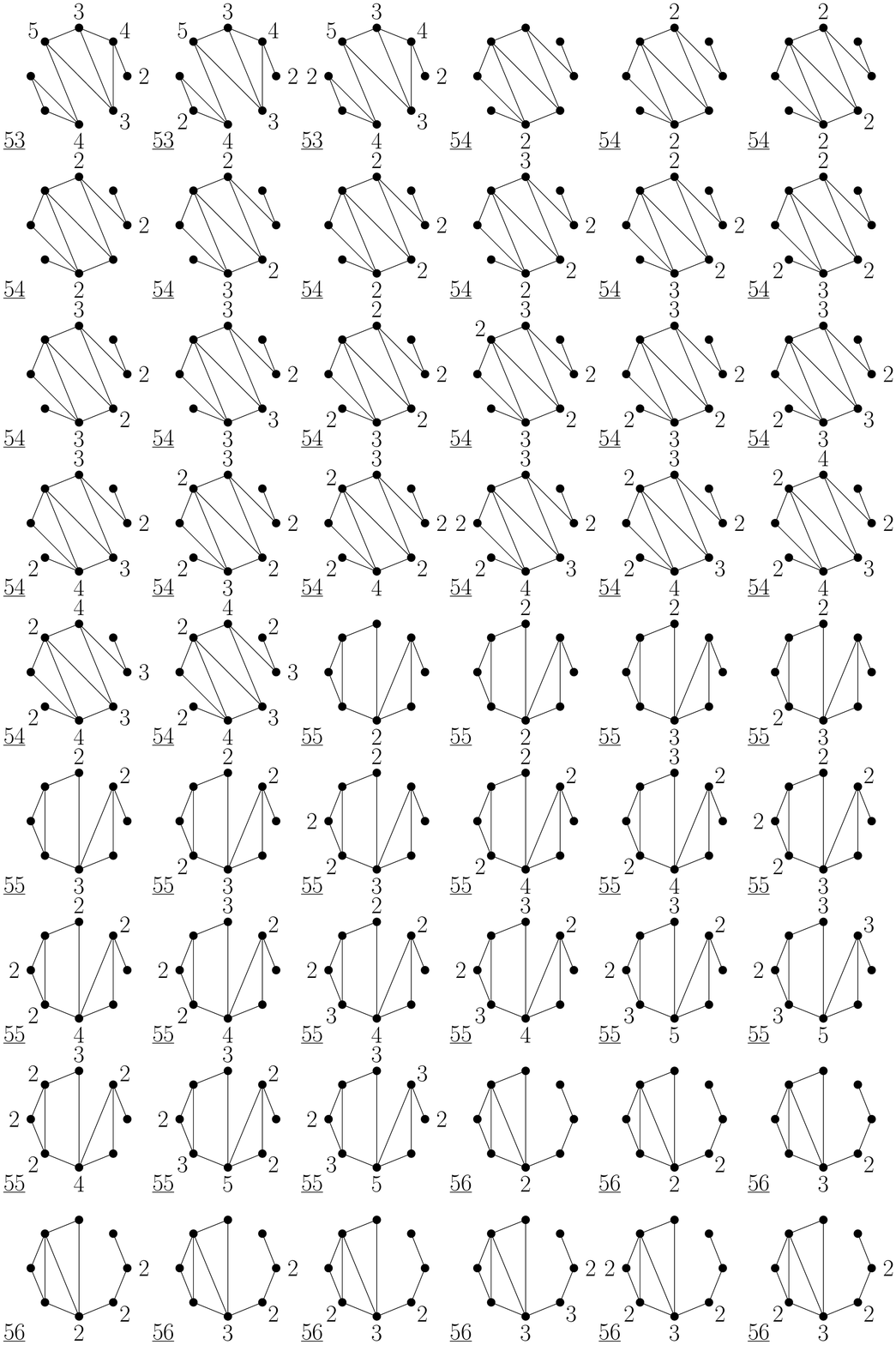}
	\end{center}
	\caption{The set $\mathcal{W}(E_8)$ (continued).}
\end{figure}
\begin{figure}[H]
	\begin{center}
		\includegraphics[scale=\scalingconstant]{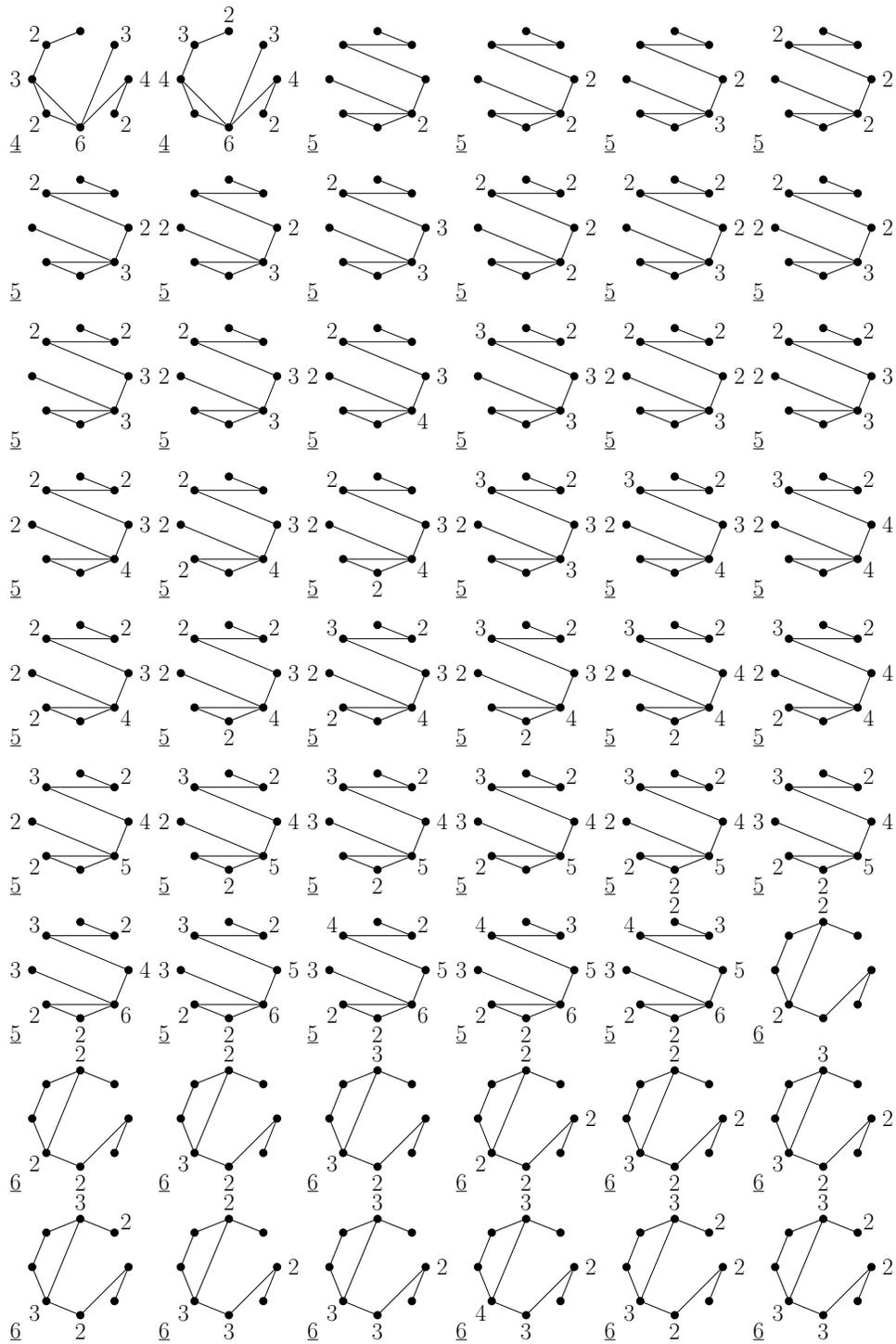}
	\end{center}
	\caption{The set $\mathcal{W}(E_8)$ (continued).}
\end{figure}
\begin{figure}[H]
	\begin{center}
		\includegraphics[scale=\scalingconstant]{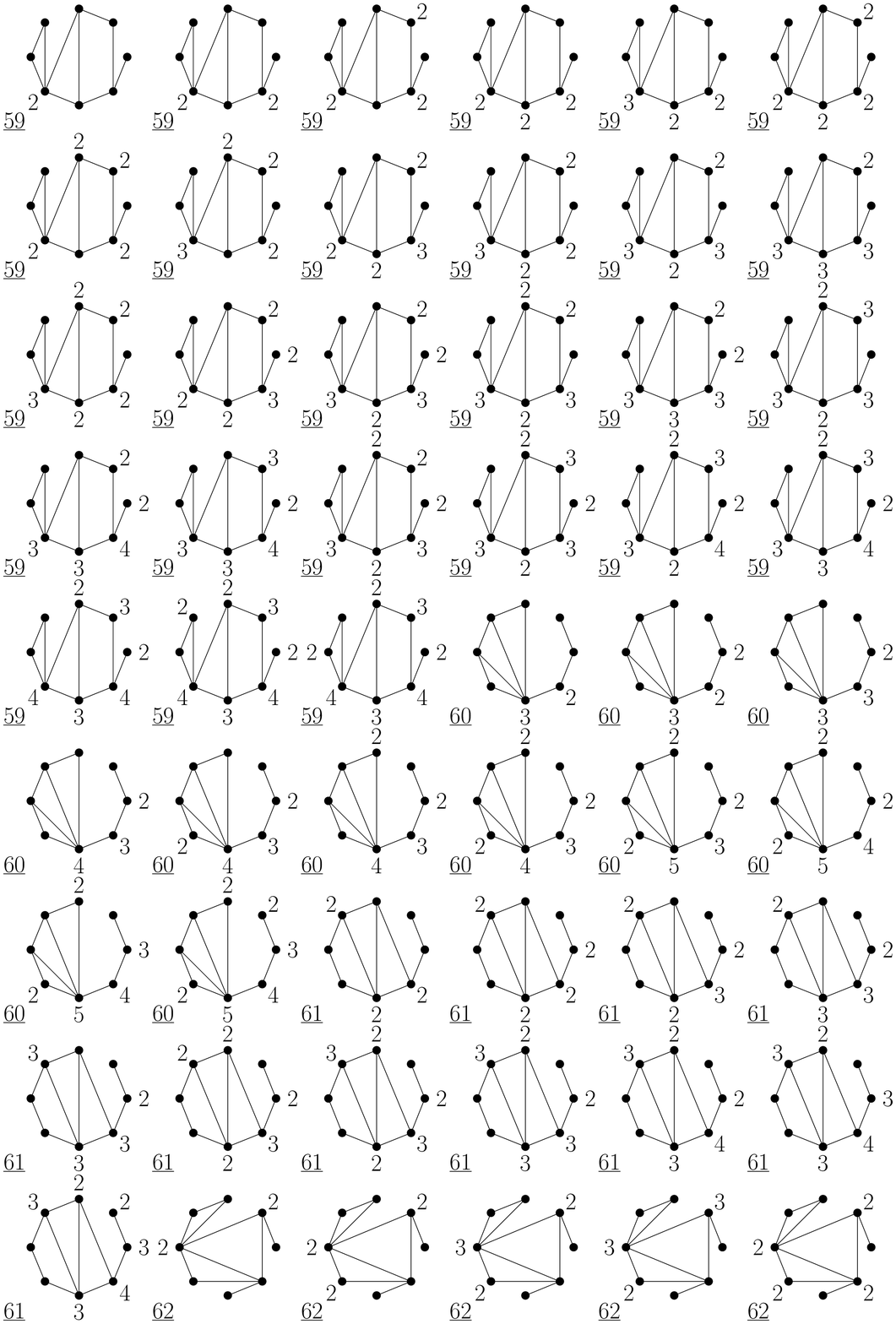}
	\end{center}
	\caption{The set $\mathcal{W}(E_8)$ (continued).}
\end{figure}
\begin{figure}[H]
	\begin{center}
		\includegraphics[scale=\scalingconstant]{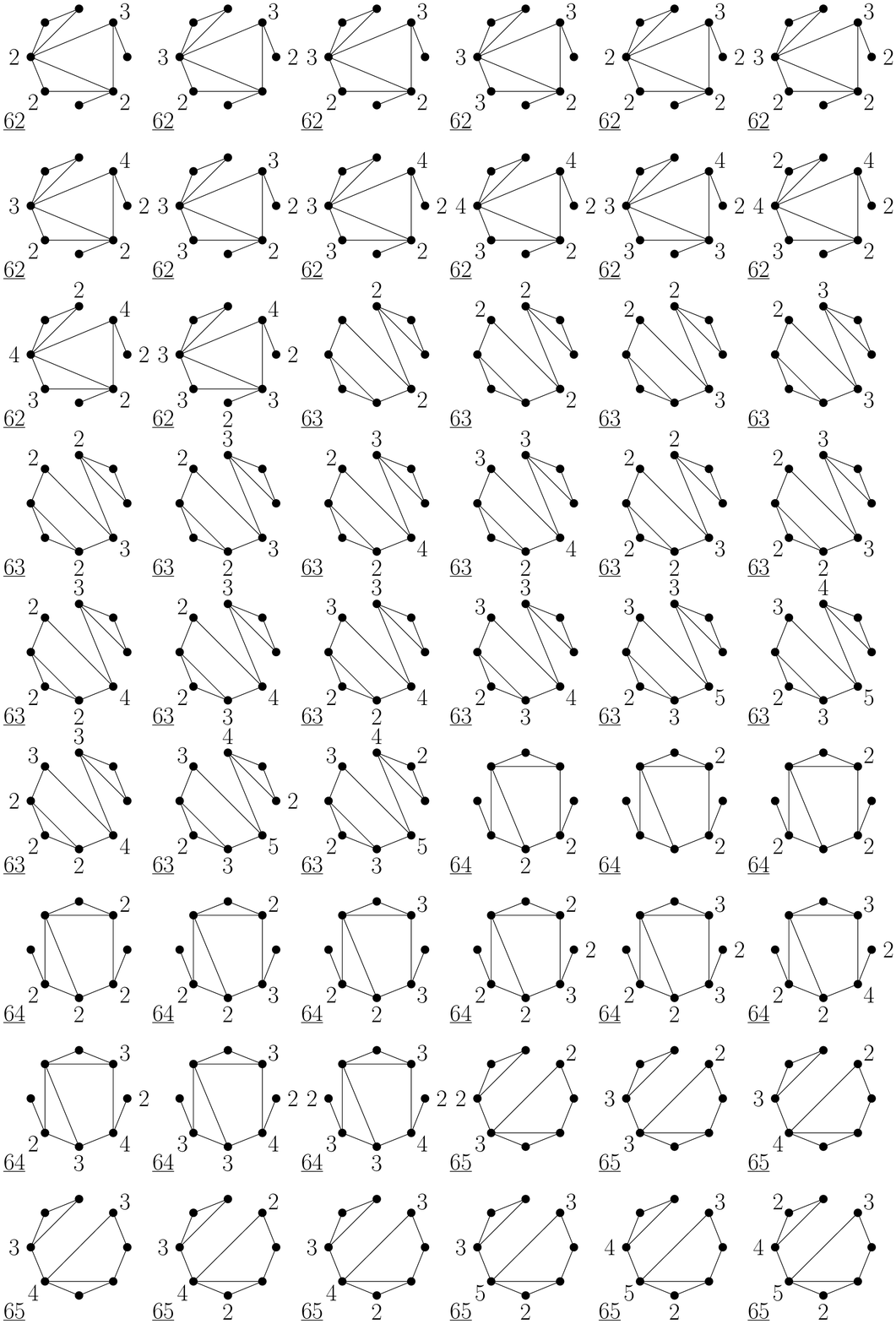}
	\end{center}
	\caption{The set $\mathcal{W}(E_8)$ (continued).}
\end{figure}
\begin{figure}[H]
	\begin{center}
		\includegraphics[scale=\scalingconstant]{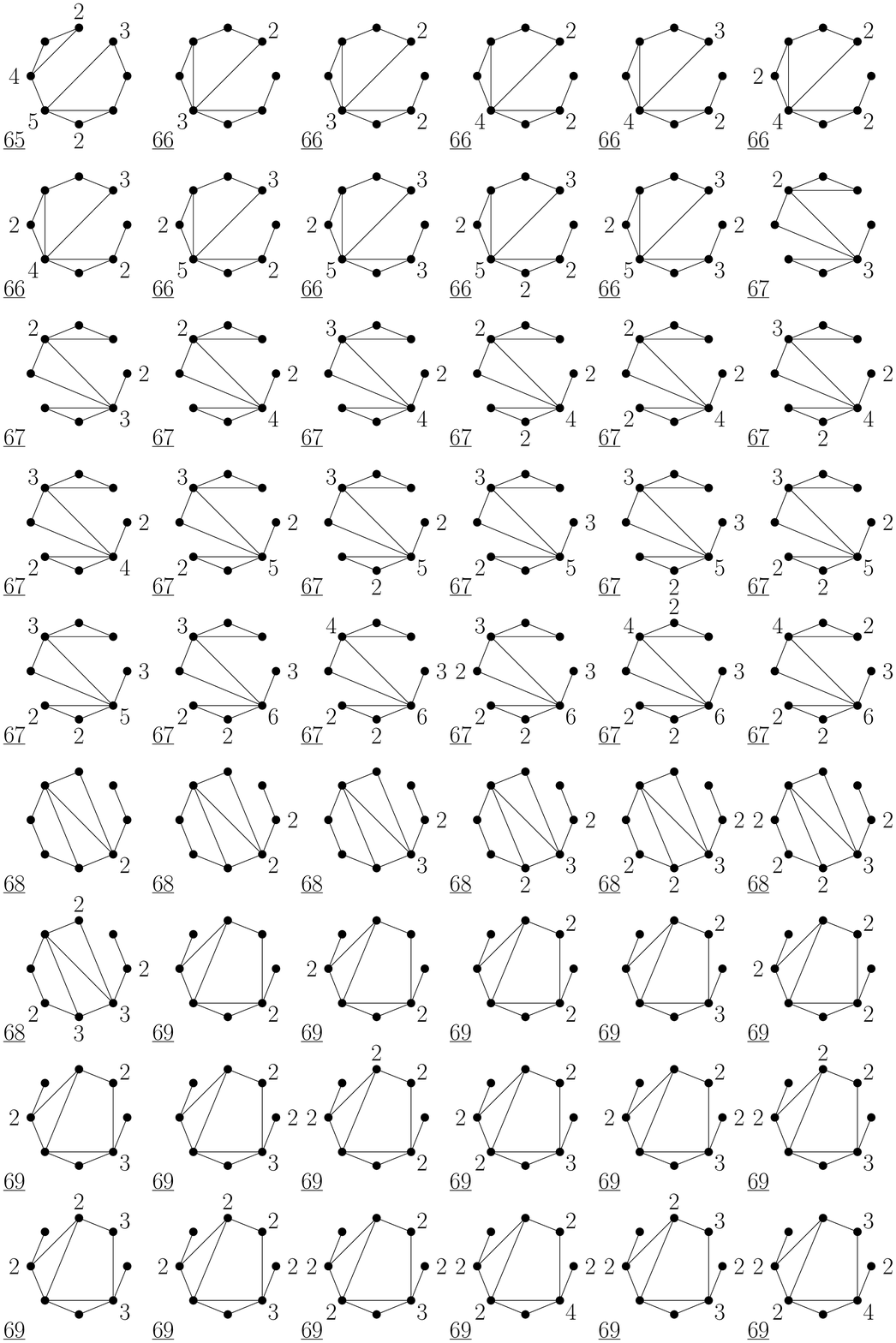}
	\end{center}
	\caption{The set $\mathcal{W}(E_8)$ (continued).}
\end{figure}
\begin{figure}[H]
	\begin{center}
		\includegraphics[scale=\scalingconstant]{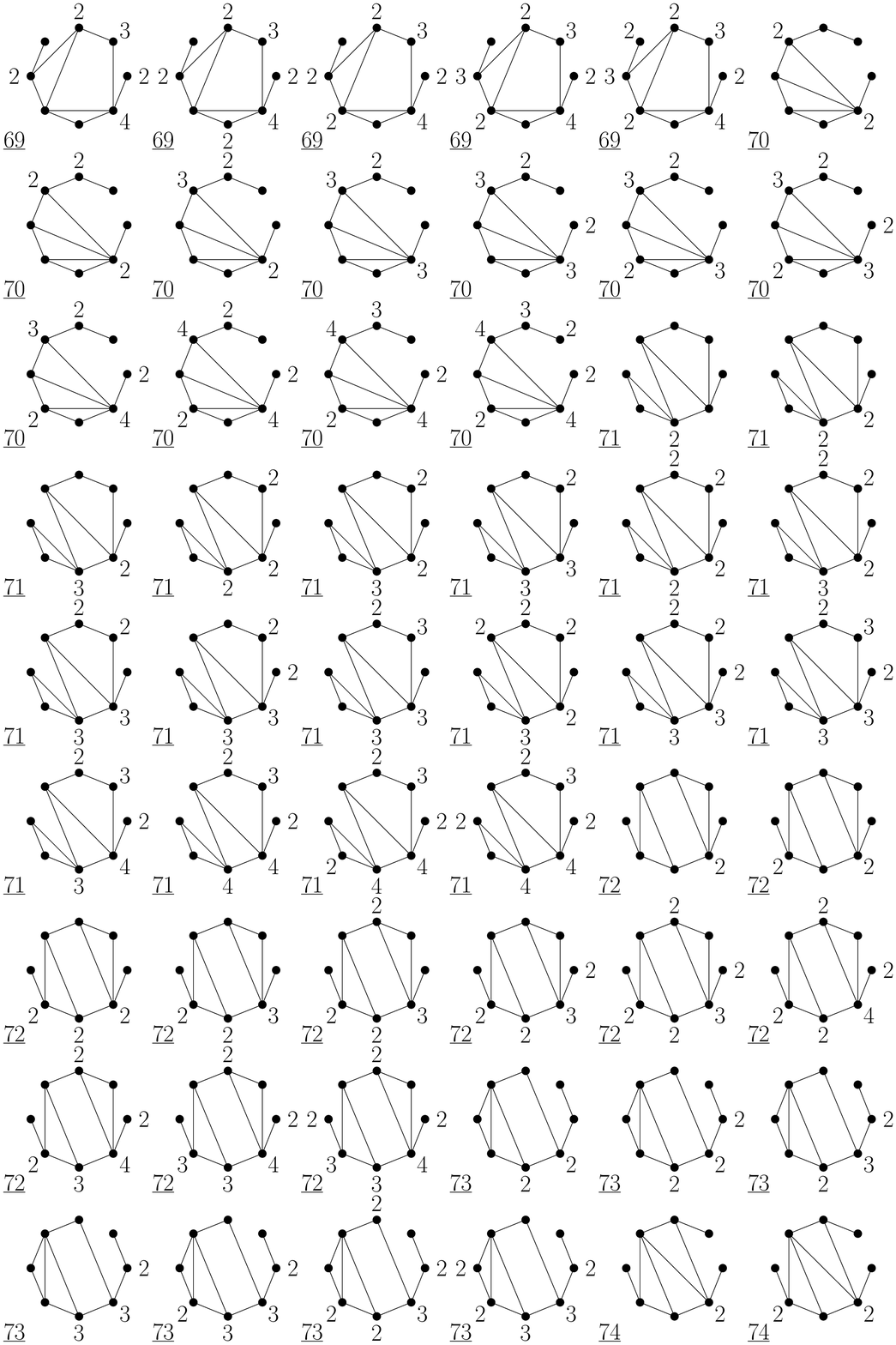}
	\end{center}
	\caption{The set $\mathcal{W}(E_8)$ (continued).}
\end{figure}
\begin{figure}[H]
	\begin{center}
		\includegraphics[scale=\scalingconstant]{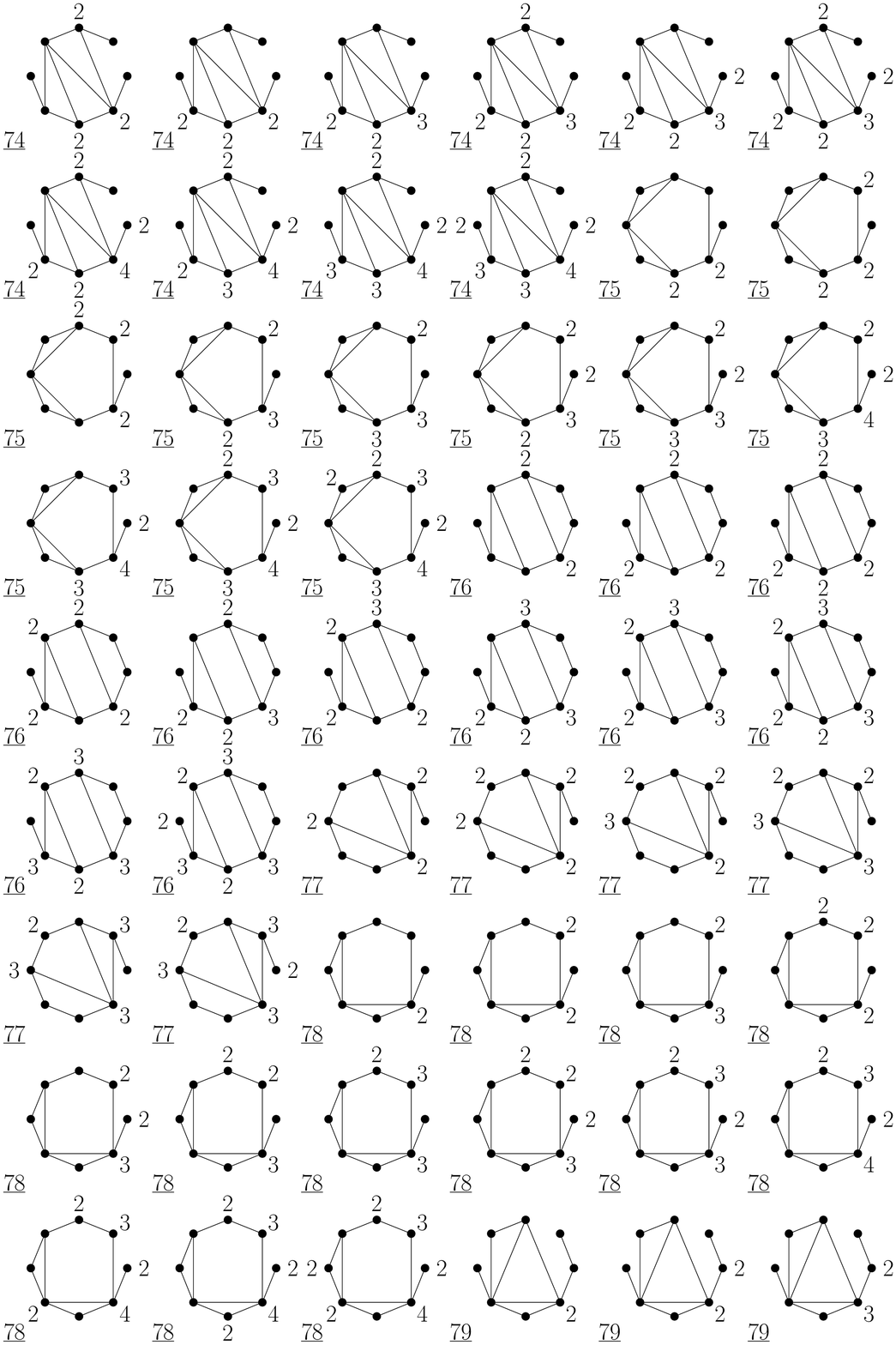}
	\end{center}
	\caption{The set $\mathcal{W}(E_8)$ (continued).}
\end{figure}
\begin{figure}[H]
	\begin{center}
		\includegraphics[scale=\scalingconstant]{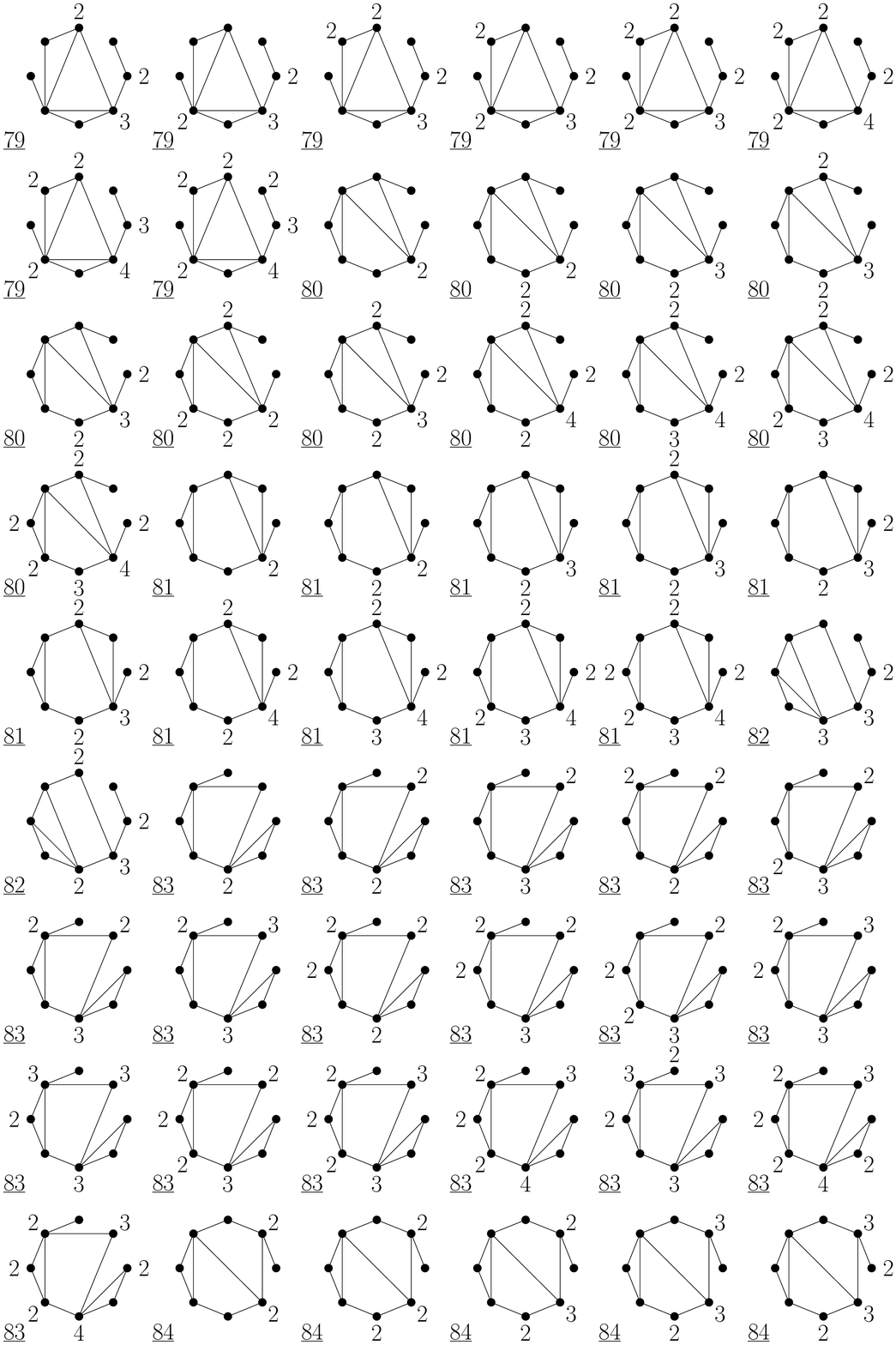}
	\end{center}
	\caption{The set $\mathcal{W}(E_8)$ (continued).}
\end{figure}
\begin{figure}[H]
	\begin{center}
		\includegraphics[scale=\scalingconstant]{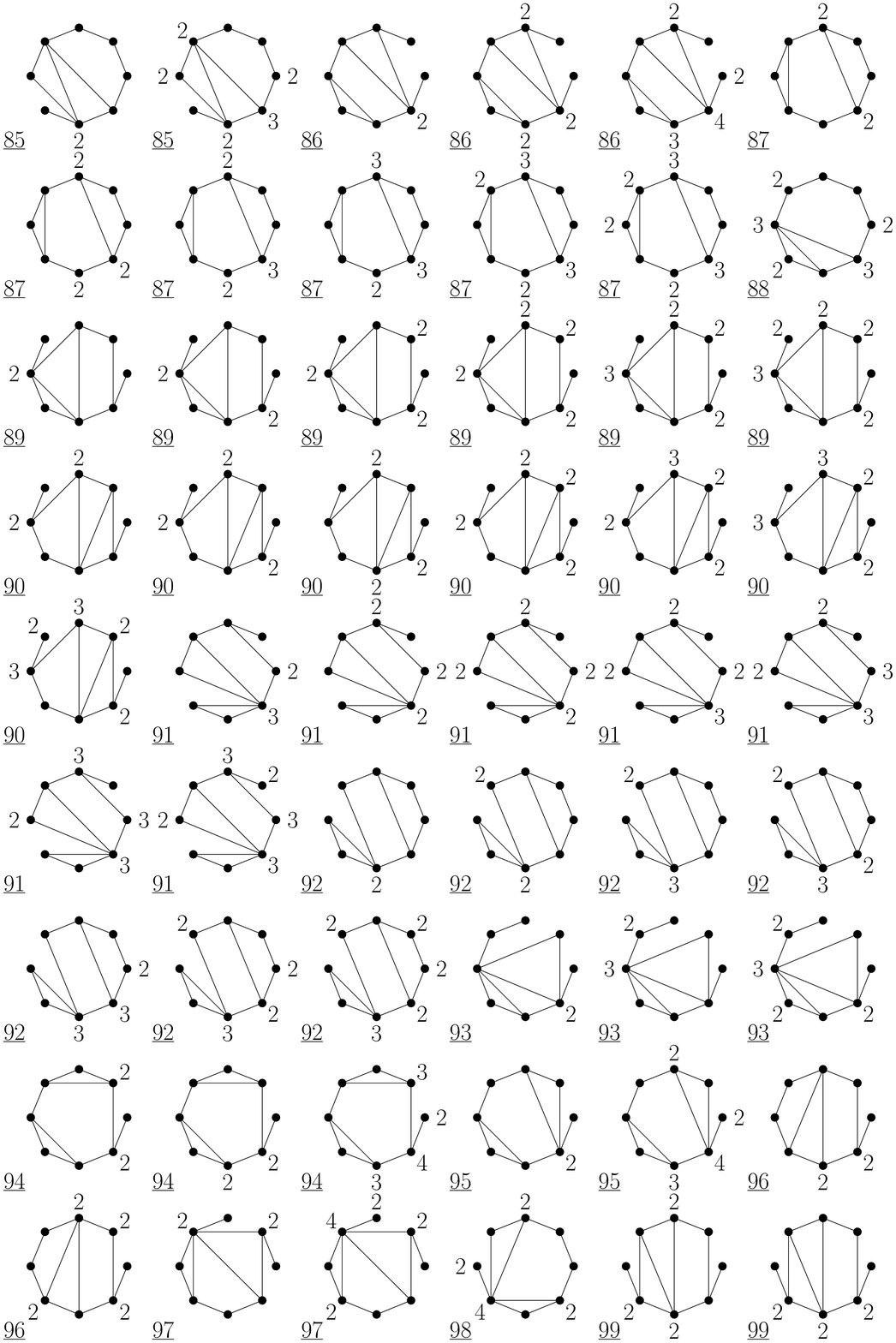}
	\end{center}
	\caption{The set $\mathcal{W}(E_8)$ (continued).}
\end{figure}
\begin{figure}[H]
	\begin{center}
		\includegraphics[scale=\scalingconstant]{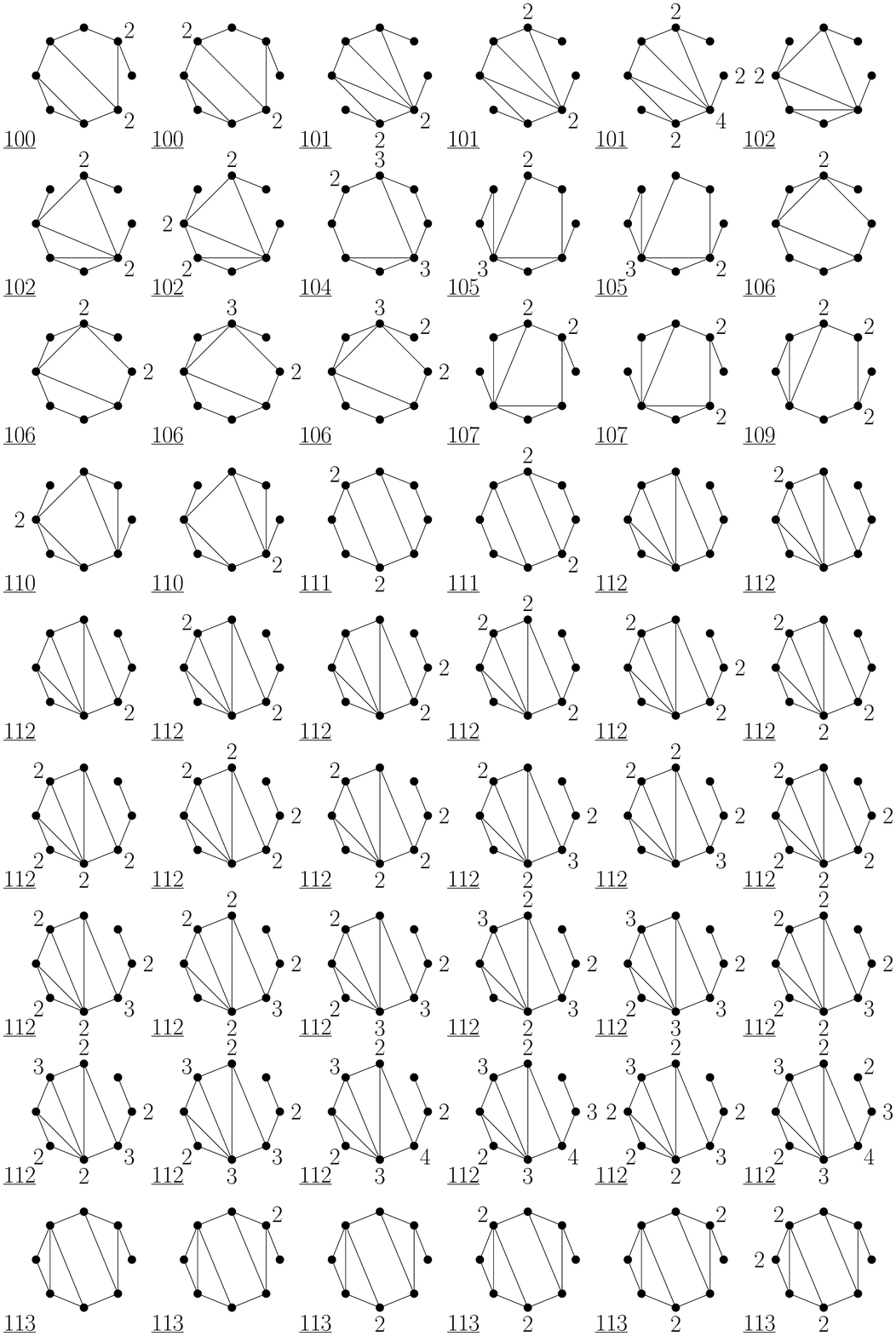}
	\end{center}
	\caption{The set $\mathcal{W}(E_8)$ (continued).}
\end{figure}
\begin{figure}[H]
	\begin{center}
		\includegraphics[scale=\scalingconstant]{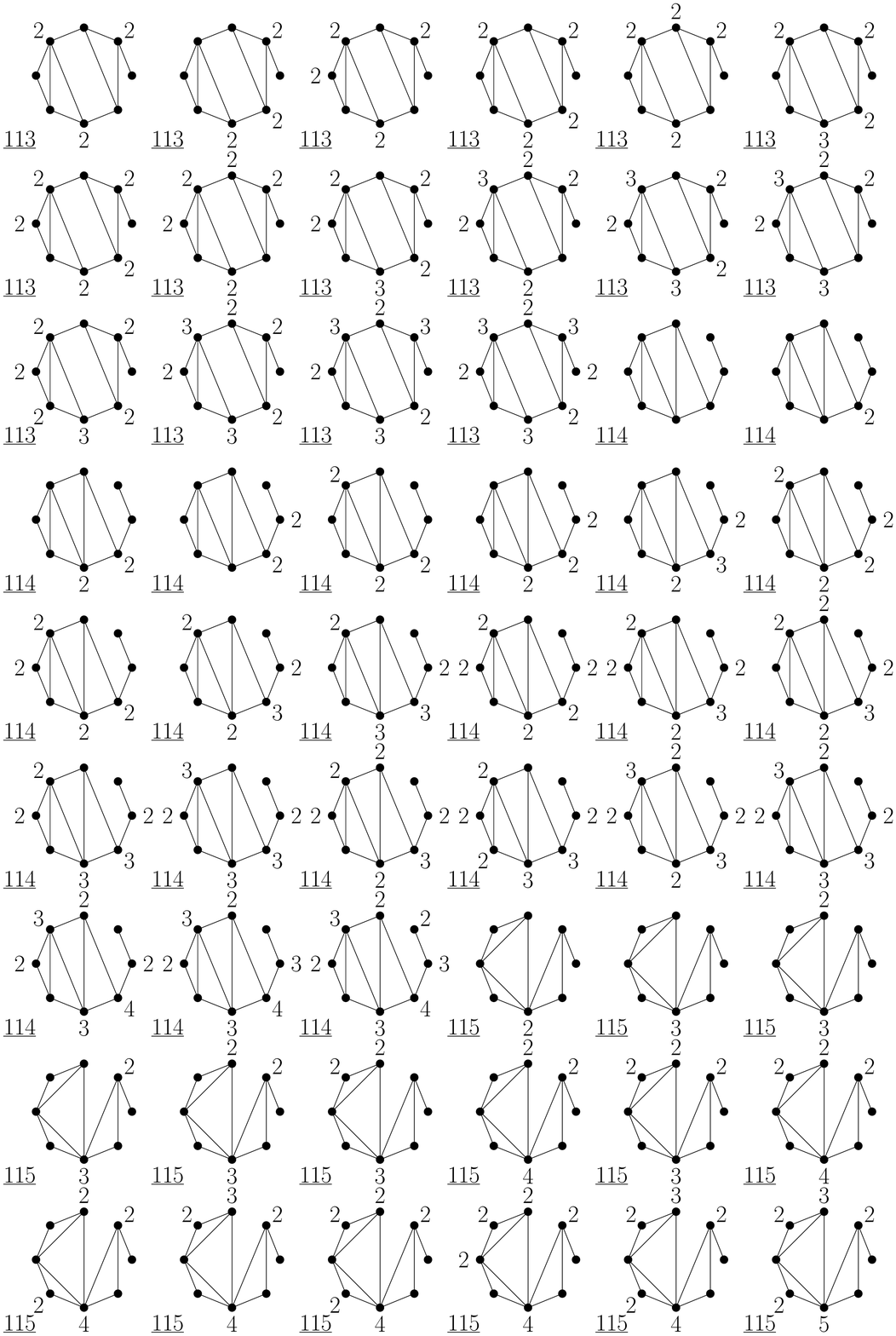}
	\end{center}
	\caption{The set $\mathcal{W}(E_8)$ (continued).}
\end{figure}
\begin{figure}[H]
	\begin{center}
		\includegraphics[scale=\scalingconstant]{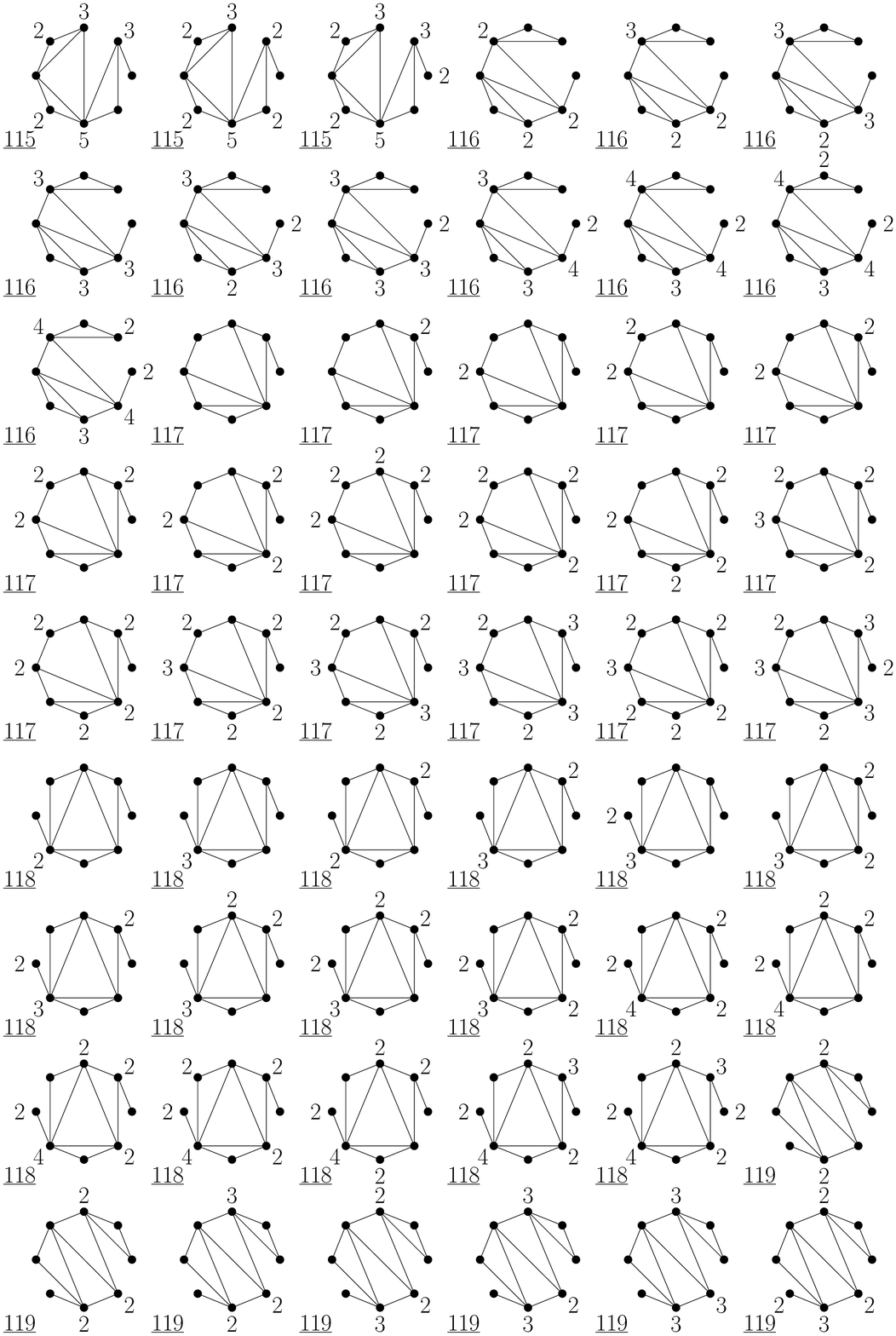}
	\end{center}
	\caption{The set $\mathcal{W}(E_8)$ (continued).}
\end{figure}
\begin{figure}[H]
	\begin{center}
		\includegraphics[scale=\scalingconstant]{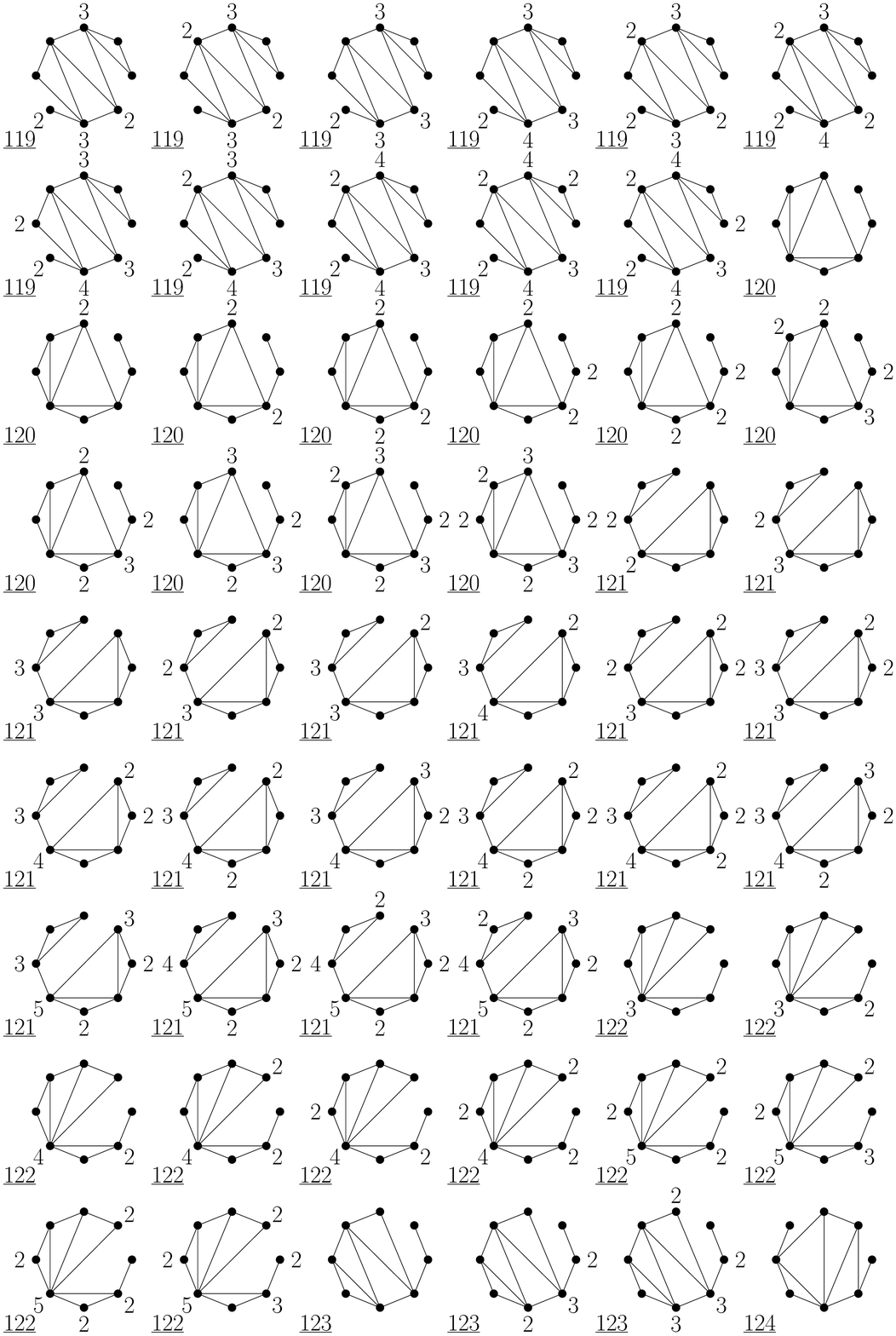}
	\end{center}
	\caption{The set $\mathcal{W}(E_8)$ (continued).}
\end{figure}
\begin{figure}[H]
	\begin{center}
		\includegraphics[scale=\scalingconstant]{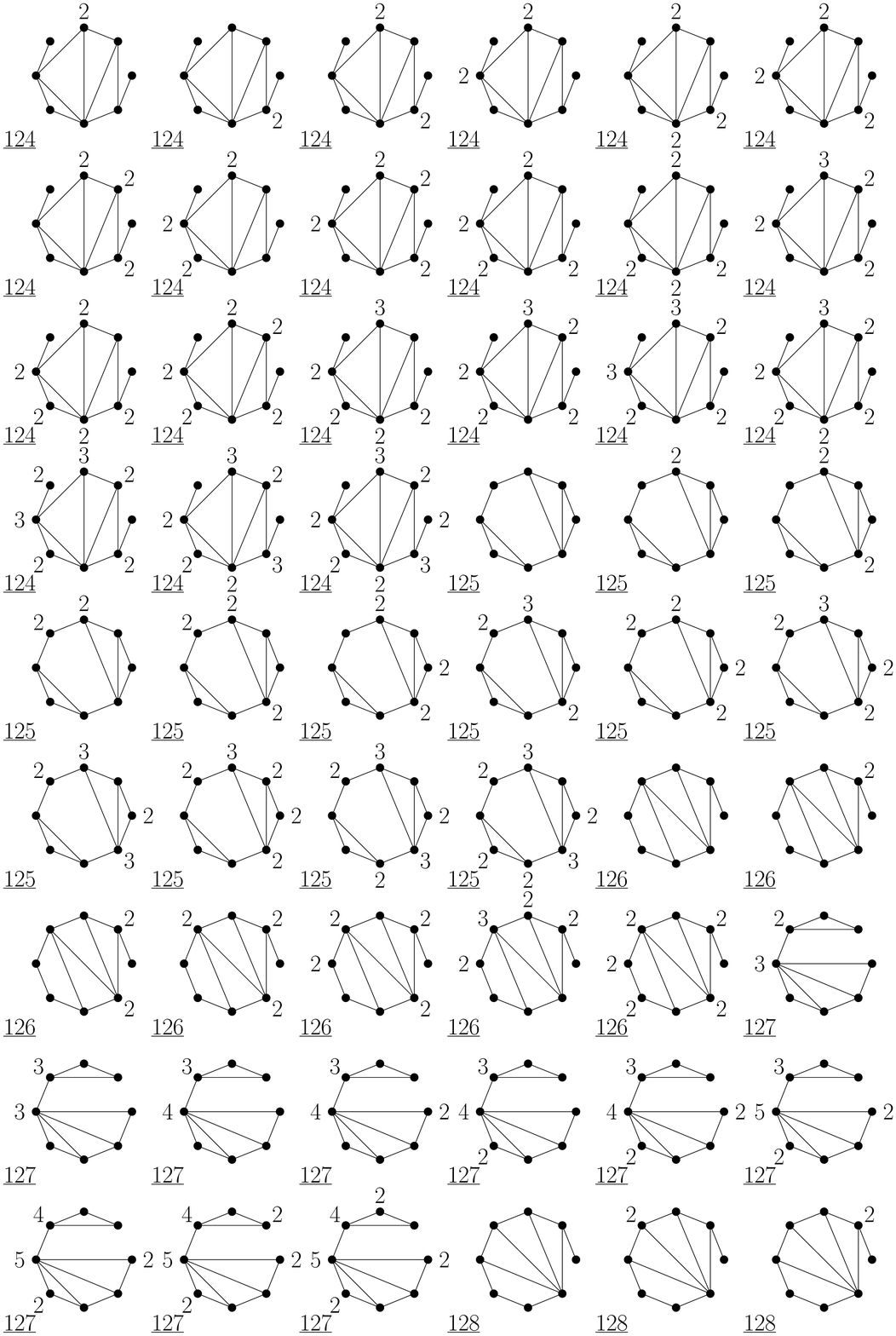}
	\end{center}
	\caption{The set $\mathcal{W}(E_8)$ (continued).}
\end{figure}
\begin{figure}[H]
	\begin{center}
		\includegraphics[scale=\scalingconstant]{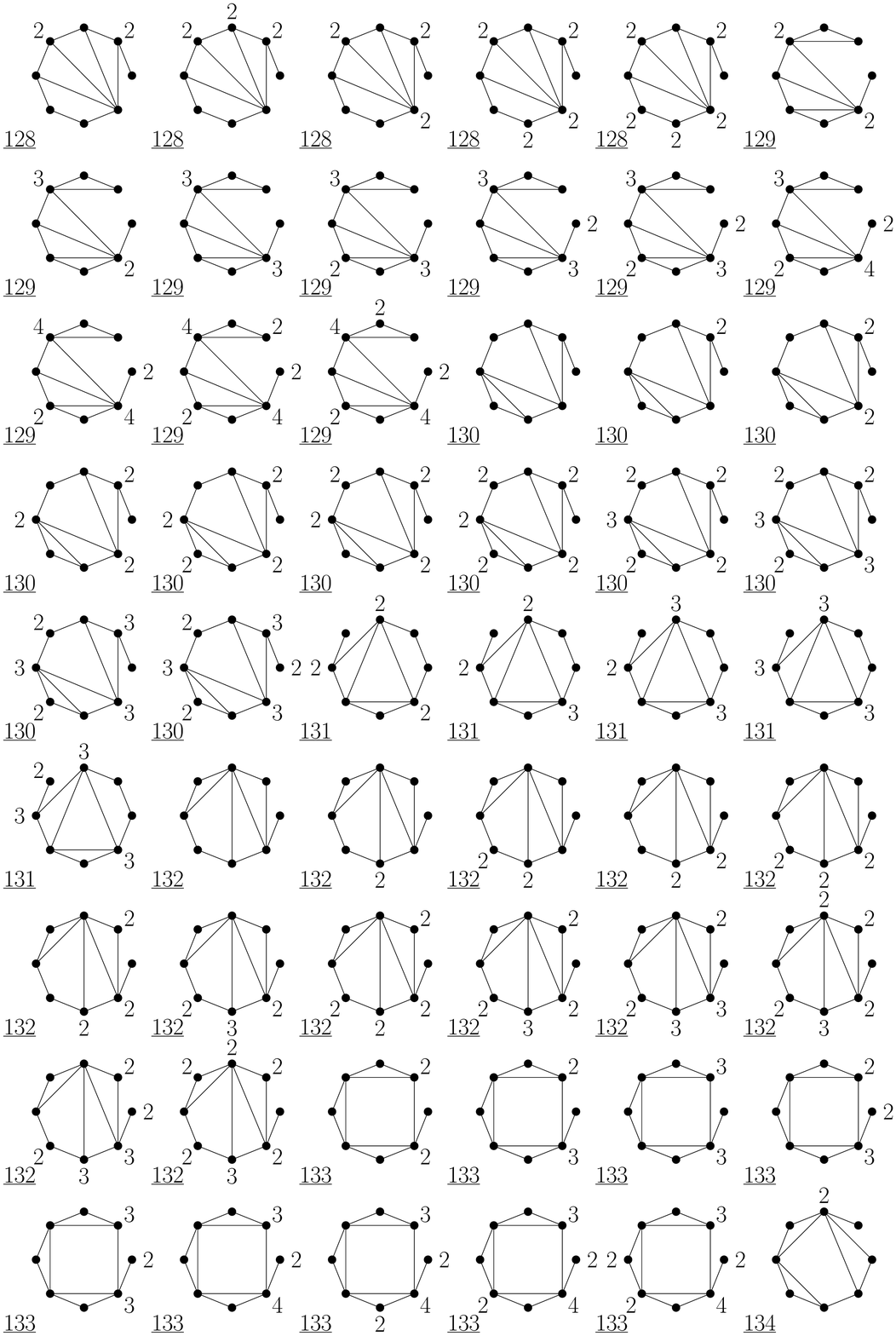}
	\end{center}
	\caption{The set $\mathcal{W}(E_8)$ (continued).}
\end{figure}
\begin{figure}[H]
	\begin{center}
		\includegraphics[scale=\scalingconstant]{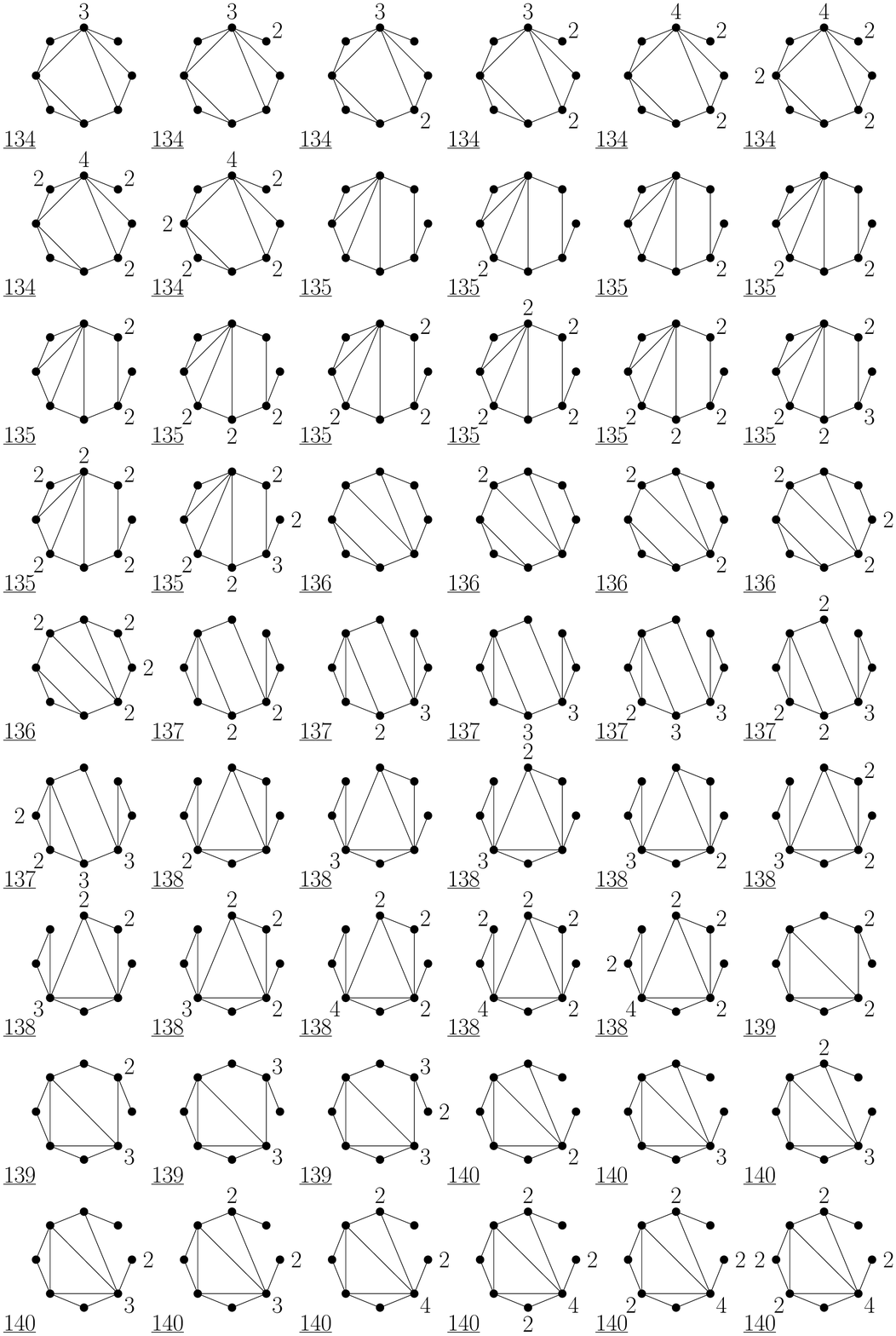}
	\end{center}
	\caption{The set $\mathcal{W}(E_8)$ (continued).}
\end{figure}
\begin{figure}[H]
	\begin{center}
		\includegraphics[scale=\scalingconstant]{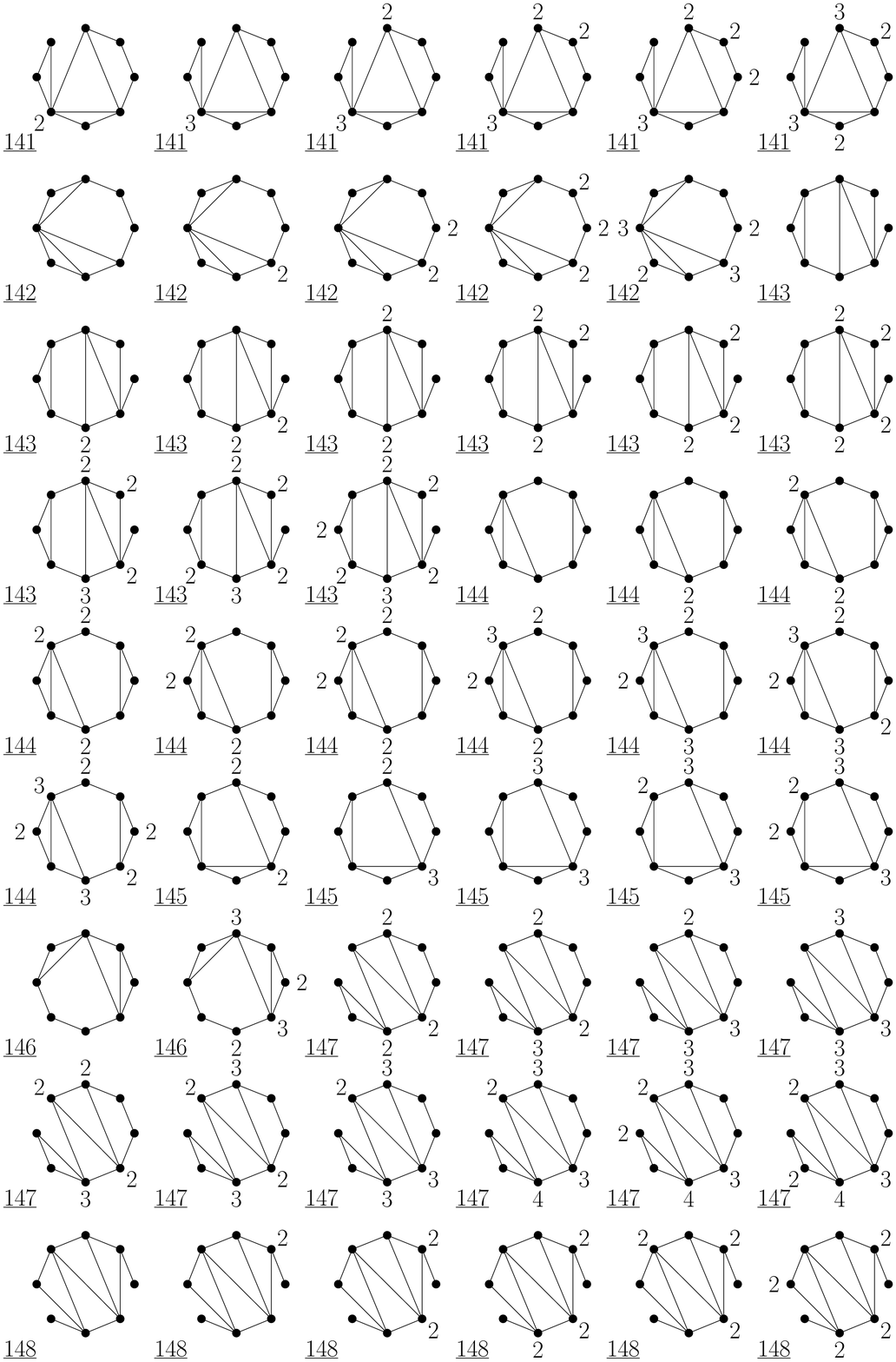}
	\end{center}
	\caption{The set $\mathcal{W}(E_8)$ (continued).}
\end{figure}
\begin{figure}[H]
	\begin{center}
		\includegraphics[scale=\scalingconstant]{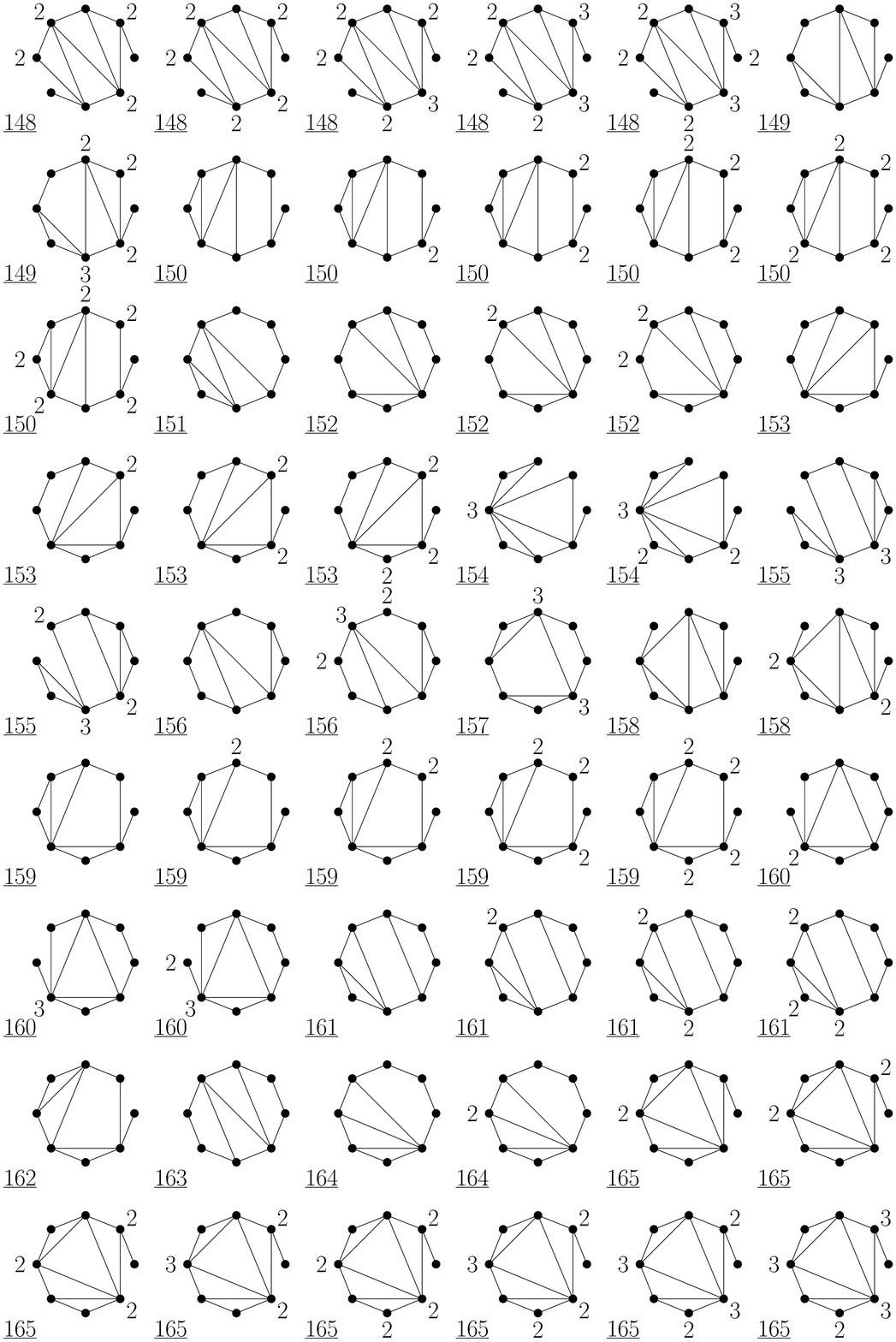}
	\end{center}
	\caption{The set $\mathcal{W}(E_8)$ (continued).}
\end{figure}
\begin{figure}[H]
	\begin{center}
		\includegraphics[scale=\scalingconstant]{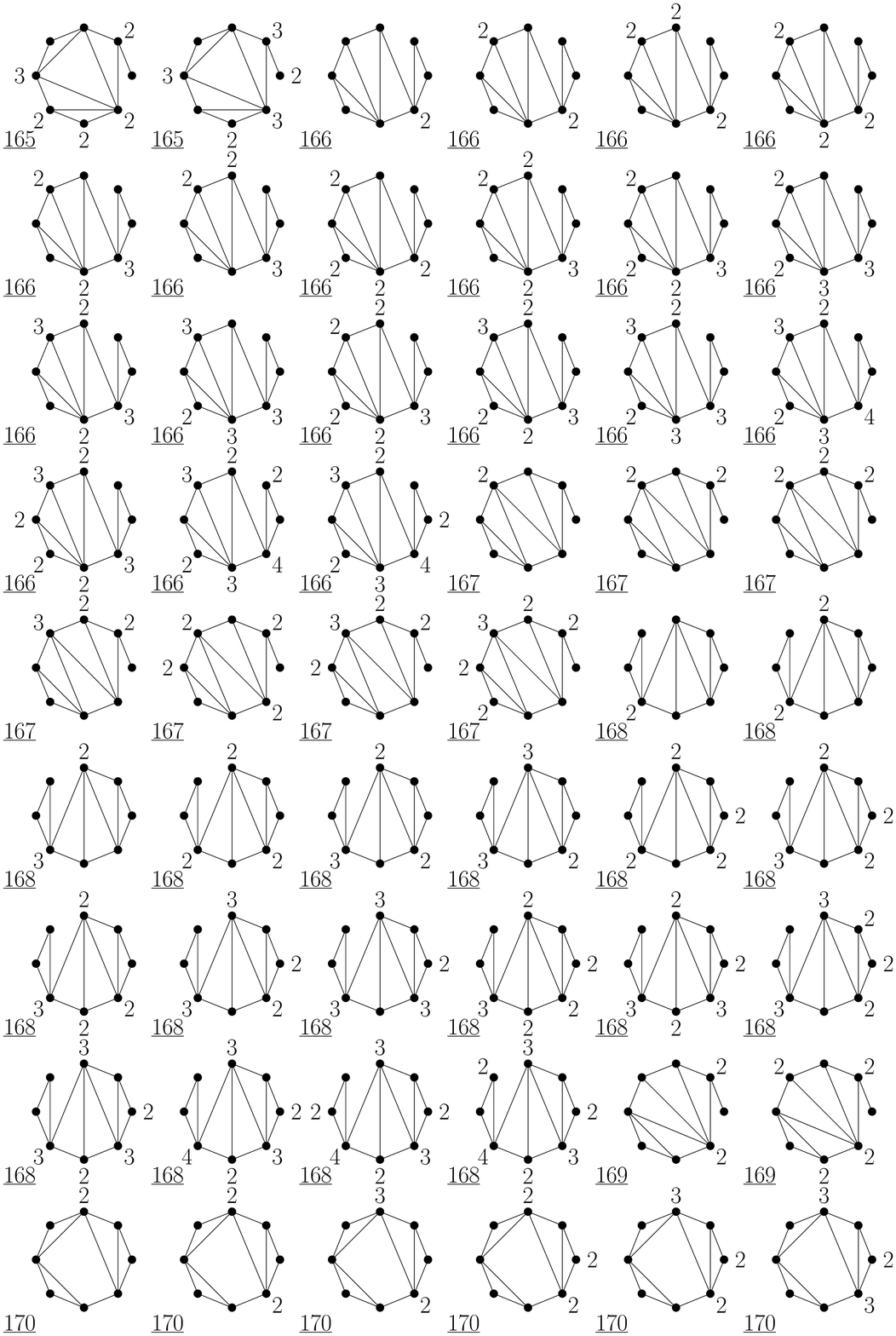}
	\end{center}
	\caption{The set $\mathcal{W}(E_8)$ (continued).}
\end{figure}
\begin{figure}[H]
	\begin{center}
		\includegraphics[scale=\scalingconstant]{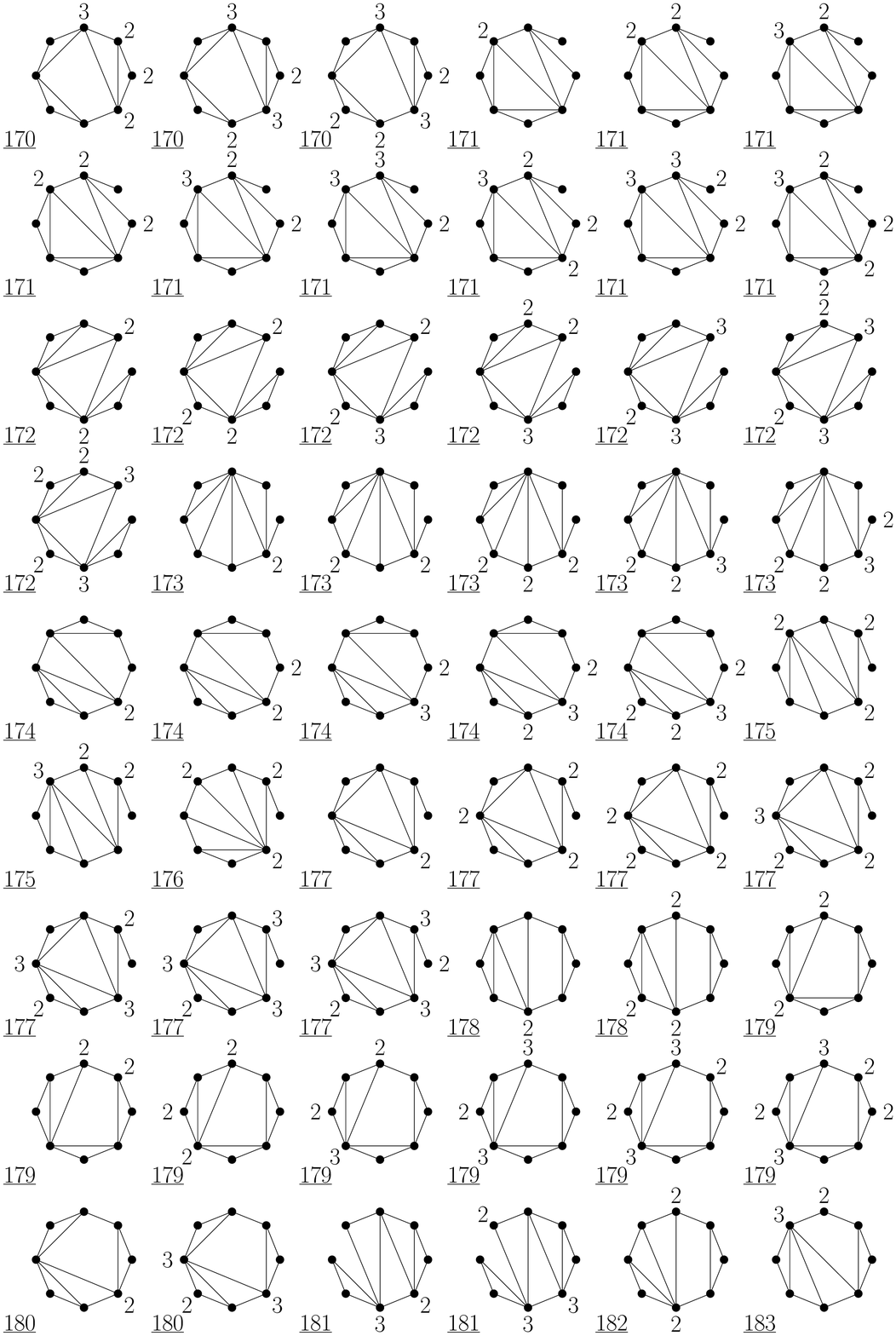}
	\end{center}
	\caption{The set $\mathcal{W}(E_8)$ (continued).}
\end{figure}
\begin{figure}[H]
	\begin{center}
		\includegraphics[scale=\scalingconstant]{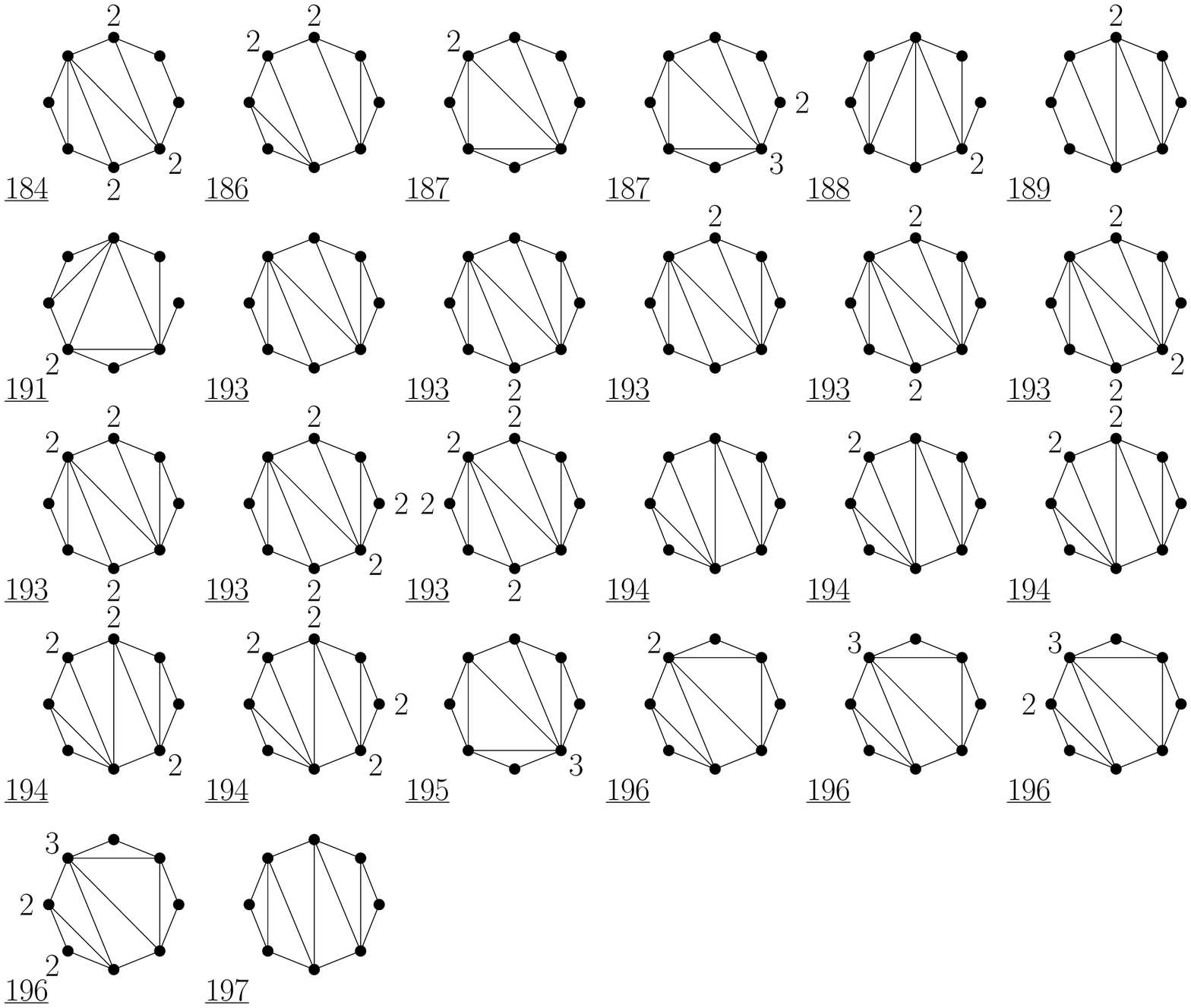}
	\end{center}
	\caption{The set $\mathcal{W}(E_8)$ (continued).}
  \label{fig:allowed_diagrams_E8-last}
	\hfill
	\vspace{3in}
\end{figure}

\end{document}